\documentclass[aos]{imsart}

\usepackage{amsmath, amssymb,amsfonts,bbm,verbatim,multirow,color,url}
\usepackage{epic,eepic,psfrag,epsfig}
\usepackage[sf,bf,SF,footnotesize]{subfigure}
\usepackage{graphicx}
\usepackage{amsthm,breakcites}
\usepackage{amscd}
\usepackage{epsfig}
\usepackage{natbib}
\usepackage{enumerate}
\usepackage{array}
\usepackage[small]{caption}
\RequirePackage[colorlinks,citecolor=blue,urlcolor=blue,breaklinks,hypertexnames=false]{hyperref}

\arxiv{arXiv:2001.05513}

\startlocaldefs
\newtheorem{thm}{Theorem}

\newtheorem{lemma}[thm]{Lemma}
\newtheorem{example}[thm]{Example}

\newtheorem{prop}[thm]{Proposition}
\newtheorem{cor}[thm]{Corollary}

\usepackage{xr}
\DeclareMathOperator*{\essinf}{ess\,inf}
\DeclareMathOperator*{\esssup}{ess\,sup}

\endlocaldefs

\begin{document}

\begin{frontmatter}

\title{Optimal rates for independence testing via $U$-statistic permutation tests}
\runtitle{Independence testing via permutation tests}

\begin{aug}
\author[A]{\fnms{Thomas} B. \snm{Berrett}\thanksref{t1}\ead[label=e1]{tom.berrett@warwick.ac.uk}}
\author[B]{\fnms{Ioannis} \snm{Kontoyiannis}\ead[label=e2]{yiannis@maths.cam.ac.uk}} \\
\and
\author[B]{\fnms{Richard} J. \snm{Samworth}\thanksref{t2}\ead[label=e3]{r.samworth@statslab.cam.ac.uk}}
\thankstext{t1}{Financial support from the French National Research Agency (ANR) under the grants Labex Ecodec (ANR-11-LABEX-0047 and ANR-17-CE40-0003.}
\thankstext{t2}{Research supported by Engineering and Physical Sciences Reseach Council (EPSRC) Programme grant EP/N031938/1 and EPSRC Fellowship EP/P031447/1.}

\runauthor{T. B. Berrett, I. Kontoyiannis and R. J. Samworth}


  \address[A]{Department of Statistics, University of Warwick, Coventry, CV4 7AL, United Kingdom\\ 
          \printead{e1}}

  \address[B]{Statistical Laboratory, Centre for Mathematical Sciences, Wilberforce Road, Cambridge, CB3 0WB, United Kingdom\\ \printead{e2}, \printead{e3}}
        




\end{aug}




\begin{abstract}
We study the problem of independence testing given independent and identically distributed pairs taking values in a $\sigma$-finite, separable measure space.  Defining a natural measure of dependence $D(f)$ as the squared $L^2$-distance between a joint density $f$ and the product of its marginals, we first show that there is no valid test of independence that is uniformly consistent against alternatives of the form $\{f: D(f) \geq \rho^2 \}$. We therefore restrict attention to alternatives that impose additional Sobolev-type smoothness constraints, and define a permutation test based on a basis expansion and a $U$-statistic estimator of $D(f)$ that we prove is minimax optimal in terms of its separation rates in many instances.  Finally, for the case of a Fourier basis on $[0,1]^2$, we provide an approximation to the power function that offers several additional insights.  Our methodology is implemented in the \texttt{R} package \texttt{USP}.
\end{abstract}



\end{frontmatter}

\section{Introduction}

Independence is a fundamental concept in both probability and statistics; it distinguishes the former from a mere branch of measure theory, and underpins both statistical theory and the way practitioners think about modelling.  For statisticians, it is frequently important to ascertain whether or not assumptions of independence are realistic, both to determine whether certain theoretical properties of procedures can be expected to hold, and to assess the goodness-of-fit of a statistical model.

Classical approaches to independence testing have focused on the simple setting of univariate Euclidean data, and have often only had power against restricted classes of alternatives.  These include tests based on Pearson's correlation \citep[e.g.][]{Pearson1920}, Spearman's rank correlation coefficient \citep{Spearman04}, Kendall's tau \citep{Kendall38} and Hoeffding's D \citep{Hoeffding48}.  However, motivated by a desire to handle the more general data types that are ubiquitous in modern-day practice, as well as to have power against broader classes of alternatives, the topic of independence testing has undergone a renaissance in recent years.  Since, in settings of interest, no uniformly most powerful test exists (see Theorem~\ref{Thm:Hardness} below and the surrounding discussion), several different perspectives and new tests have emerged, such as those based on the Hilbert--Schmidt independence criterion \citep{Gretton05,Pfister17,LiYuan2019,MALM2019}, distance covariance \citep{Szekely07,Sejdinovic13}, optimal transport and multivariate ranks \citep{DebSen2019,SDH2019}, copula transforms \citep{Kojadinovic2009}, sample space partitioning \citep{Heller16} and nearest neighbour methods \citep{BS2019}.  For practical studies with discrete data, Pearson's chi-squared independence test remains ubiquitous in the scientific literature, despite the drawback that its size guarantees rely on pointwise asymptotic arguments that may fail to control the Type I error in finite samples; see Section~\ref{Sec:Simulations} below.  Independence tests for continuous data are also common in applications such as linguistics \citep{NguyenEisenstein2017}, genetics \citep{Steuer2002} and public health \citep{Reshef2011}, and have also been applied to functional data arising from credit card activity and geomagnetic records \citep{GabrysKokoszka2007}. 

This plethora of approaches gives rise to natural theoretical questions about the fundamental statistical difficulty of independence testing.  In the setting where the marginal distributions are both univariate, early asymptotic results on minimax separation rates over certain classes of alternatives are given in \citet{Ingster1989}, \citet{Ermakov1990} and \citet{Ingster1996}.  There has been recent work on multivariate settings \citep{LiYuan2019,MALM2019}, but many open questions remain.  

Another issue with several of the tests mentioned above is that the asymptotic distribution of the test statistic under the null hypothesis of independence depends on unknown features of the relevant marginal distributions, so it is difficult to obtain an appropriate critical value.  An attractive approach, therefore, is to use a permutation test, which uses permutations to mimic the null behaviour of the test statistic.  Though the principle has been known for many decades \citep[e.g.][Chapter~21]{Pitman38,Fisher1935}, permutation tests are becoming increasingly popular in modern statistics and machine learning (e.g.~A/B testing), due to their ease of use and their guaranteed finite-sample Type I error control across the entire null hypothesis parameter space, assuming only that the data are exchangeable under the null. Besides (unconditional) independence testing, they have also been studied in problems such as conditional independence testing \citep{BWBS2019}, two-sample testing \citep{ChungRomano2013} and changepoint analysis \citep{Antoch2001}.  We also highlight the work of \citet{ChungRomano2016}, who show how a permutation test based on a $U$-statistic can extend the scope of the two-sample Wilcoxon test to null hypotheses of the form $\theta(P,Q) = \theta_0$ (where $P$ and $Q$ are the two underlying distributions), providing pointwise asymptotic size guarantees in general, and exact size guarantees when $P=Q$.  For an overview of the study of permutation tests, see, for example, \citet{LehmannRomano2005} and \citet{PesarinSalmaso2010}. 

In the context of permutation tests for independence, \citet{Romano1989} considered a class of plug-in test statistics of the form $T_n = n^{1/2} \delta( \hat{P}_n, \hat{P}_n^X \hat{P}_n^Y)$, where $\delta(P,Q) = \sup_{V \in \mathcal{V}} |P(V) - Q(V)|$ for a Vapnik--Chervonenkis class of sets $\mathcal{V}$, and where $\hat{P}_n$, $\hat{P}_n^X$ and~$\hat{P}_n^Y$ are the empirical distributions of the data pairs and their marginals, respectively. Fixing a sequence of alternatives $(P_n)$, he showed that, under the condition that $\mathbb{P}_{P_n}(T_n \leq t) \rightarrow H(t)$ for some continuous function $H$, the asymptotic power of his permutation test coincides with that of the test that uses the true critical value. In the case of univariate marginals, \citet[][Chapter 4]{AlbertThesis15} provides upper bounds on the minimax separation over Besov spaces using a test based on aggregrating many permutation tests. See also \citet{Albert15} and \citet{BS2019}. Despite these aforementioned works, however, there remains great interest in understanding better the power properties of permutation tests in the context of nonparametric independence testing.  Indeed, shortly after an earlier version of this paper was made publicly available, \citet{KBW2020} posted a complementary study of the power properties of permutation tests, with a greater focus on concentration inequalities for the test statistics as opposed to distributional results.   

In this paper, we study the problem of independence testing in a general framework, where our data consist of independent copies of a pair $(X,Y)$ taking values in a separable measure space $\mathcal{X} \times \mathcal{Y}$, equipped with a $\sigma$-finite measure $\mu$.  Assuming that the joint distribution of $(X,Y)$ has a density $f$ with respect to $\mu$, we may define a measure of dependence $D(f)$, given by the squared $L^2(\mu)$ distance between the joint density and the product of its marginal densities.  This satisfies the natural requirement that $D(f) = 0$ if and only if $X$ and $Y$ are independent.  In fact, however, our hardness result in Theorem~\ref{Thm:Hardness} reveals that it is not possible to construct a valid independence test with non-trivial power against all alternatives satisfying a lower bound on $D(f)$.  This motivates us to introduce classes satisfying an additional Sobolev-type smoothness condition as well as boundedness conditions on the joint and marginal densities.   

The first main goal of this work is to determine the minimax separation rate of independence testing over these classes, and to this end, we define a new permutation test of independence based on a $U$-statistic estimator of $D(f)$.  We refer to this test hereafter as the USP test, short for $U$-Statistic Permutation test.  Theorem~\ref{Thm:UpperBound} in Section~\ref{Sec:MainResults} provides a very general upper bound on the separation rate of independence testing; the framework is broad enough to include both discrete and absolutely continuous data, as well as data that may take values in infinite-dimensional spaces, for instance.  We show how the bound can be simplified in many special cases of interest, and, in Section~\ref{Sec:Adaptation}, how to construct adaptive versions of our tests that incur only a small loss in effective sample size.  Moreover, in Section~\ref{Sec:LowerBounds}, we go on to provide matching lower bounds in several instances, allowing us to conclude that our USP test attains the minimax optimal separation rate for independence testing in such settings.  In Section~\ref{Sec:PowerFunction}, we elucidate an approximation to the power function of our test at local alternatives, thereby providing a very detailed description of its properties.  Numerical properties of our procedure are studied in Section~\ref{Sec:Simulations}: we first show how an alternative representation of our test statistic dramatically reduces the computational complexity of our procedure, and then present a simulation study that reveals the strong empirical performance of our test in different settings.  Section~\ref{Sec:Discussion} provides further discussion.  Proofs of some of our main results are given in Section~\ref{Sec:Proofs}; for other results, designated with (BKS(2020)), the proofs appear in the supplementary material, where auxiliary results (labelled with an `S' prefix) are also given.  Our methodology is implemented in the \texttt{R} package \texttt{USP} \citep{BKS2020b}.

Further contributions of this paper are to introduce new sets of tools for studying both permutation tests and $U$-statistics; we believe both will find application beyond the scope of this work, in particular because many popular measures of dependence, such as distance covariance and the Hilbert--Schmidt independence criterion, can be estimated using $U$-statistics. 
Specifically, in the proof of Theorem~\ref{Thm:UpperBound}, we develop moment bounds for $U$-statistics computed on permuted data sets. Moreover, Proposition~\ref{Lemma:PermutedNormality} provides normal approximation error bounds in Wasserstein distance for degenerate $U$-statistics computed on permuted data sets (using Stein's method, and extending earlier results for unpermuted data, e.g., \citet{deJong1990,RinottRotar1997,DoblerPeccati2019}), and is the basis for our local power function result (Theorem~\ref{Thm:PowerFunction}). Finally, our minimax lower bound (Lemma~\ref{Lemma:LowerBound}) may also be of independent interest, in that it provides a general approach to constructing priors over the alternative hypothesis class whose distance from the null can be explicitly bounded.\\

\emph{Notation:} We write $\mathbb{N}=\{1,2,3,\ldots\}$, $\mathbb{N}_0=\mathbb{N} \cup \{0\}$ and, for $n \in \mathbb{N}$, let $[n]:=\{1,\ldots,n\}$. We also write $[\infty]:=\mathbb{N}$.  We write $a \lesssim b$ if there exists a universal constant $C > 0$ such that $a \leq Cb$, and write, e.g., $a \lesssim_x b$ if there exists $C > 0$, depending only on $x$, such that $a \leq Cb$.  We similarly define $a \gtrsim b$ and $a \gtrsim_x b$, and write $a \asymp b$ if $a \lesssim b$ and $a \gtrsim b$, as well as $a \asymp_x b$ if $a \lesssim_x b$ and $a \gtrsim_x b$.

Let $\mathcal{S}_n$ denote the set of permutations of $[n]$. For a measure space $(\mathcal{Z}, \mathcal{C}, \nu)$ define $L^2(\nu):= \{ f: \mathcal{Z} \rightarrow \mathbb{R} : \int_{\mathcal{Z}} f^2 \,d\nu < \infty\}$, with corresponding inner product $\langle f,g \rangle_{L^2(\nu)} := \int_{\mathcal{Z}} fg \, d\nu$ and norm $\|f\|_{L^2(\nu)} := \langle f,f \rangle_{L^2(\nu)}^{1/2}$. For a function $f : \mathcal{Z} \rightarrow \mathbb{R}$ we write $\|f\|_\infty:=\sup_{z \in \mathcal{Z}} |f(z)| \in [0,\infty]$; if it is also $\mathcal{C}$-measurable, we write $\essinf_{z \in \mathcal{Z}} f(z) := \sup\bigl\{y \in \mathbb{R}: \nu\bigl(f^{-1}(-\infty,y)\bigr) = 0\bigr\}$. 

Let $\Phi$ denote the standard normal distribution function and let $\bar{\Phi} := 1 - \Phi$.  Given a sample of independent and identically distributed random variables $(X_1,Y_1),\ldots,(X_n,Y_n)$ and a $\sigma(X_1,Y_1,\ldots,X_n,Y_n)$-measurable random variable $W$, we write $\mathbb{E}_P(W)$ or $\mathbb{E}_f(W)$ for the expectation of $W$ when $(X_1,Y_1)$ has distribution $P$ or density function $f$.  Given probability measures $\mu$ and $\nu$ on $\mathcal{Z}$, we write $d_\mathrm{TV}(\mu,\nu):= \sup_{C \in \mathcal{C}} |\mu(C) - \nu(C)|$ for their total variation distance and, if both $\mu$ and $\nu$ are absolutely continuous with respect to another measure $\lambda$, then we write $d_{\chi^2}(\mu,\nu)= \bigl\{ \int_{\mathcal{Z}} \frac{(d \mu/d\lambda)^2}{d\nu/d\lambda} d\lambda - 1\bigr\}^{1/2}$ for the square root of their $\chi^2$-divergence.  If $\mathcal{Z} = \mathbb{R}$, then we write
\[
  d_{\mathrm{W}}(\mu,\nu) := \inf_{(X,Y) \sim (\mu,\nu)} \mathbb{E}|X-Y|
\]
for the Wasserstein distance between $\mu$ and $\nu$, where the infimum is taken over all pairs $(X,Y)$ defined on the same probability space with $X \sim \mu$ and $Y \sim \nu$.  When $\mathcal{Z}=\mathbb{R}$ we will also write
\[
	d_\mathrm{K}(\mu,\nu):= \sup_{z \in \mathbb{R}} \bigl| \mu \bigl( (-\infty, z] \bigr) -  \nu \bigl( (-\infty, z] \bigr) \bigr|
\]
for the Kolmogorov distance between $\mu$ and $\nu$.  If $X \sim \mu$ and $Y \sim \nu$, we sometimes write $d_\mathrm{W}(X,Y)$ and $d_\mathrm{K}(X,Y)$ as shorthand for $d_\mathrm{W}(\mu,\nu)$ and $d_\mathrm{K}(\mu,\nu)$ respectively.  We use $\triangle$ to denote the symmetric difference operation on sets, so that $A \triangle B := (A \cap B^c) \cup (A^c \cap B)$.

Finally, for $x=(x_1,\ldots,x_d) \in \mathbb{R}^d$ and $q \in [1,\infty)$, we let $\|x\|_q := \bigl(\sum_{j=1}^d |x_j|^q\bigr)^{1/q}$, with the shorthand $\|x\| := \|x\|_2$, and for a matrix $A \in \mathbb{R}^{d_1 \times d_2}$, we let $\|A\|_{\mathrm{op}} := \sup_{x:\|x\| \leq 1} \|Ax\|$ and $\|A\|_{\mathrm{F}} := \bigl\{\sum_{j=1}^{d_1}\sum_{k=1}^{d_2} A_{jk}^2\bigr\}^{1/2}$ denote its operator and Frobenius norms respectively.

\section{Problem set-up and preliminaries}
\label{Sec:SetUp}

Let $(\mathcal{X}, \mathcal{A}, \mu_X)$ and $(\mathcal{Y},\mathcal{B},\mu_Y)$ be separable\footnote{Recall that we say a measure space $(\mathcal{Z}, \mathcal{C},\nu)$ is \emph{separable} if, when equipped with the pseudo-metric $d(A,B) := \nu(A \triangle B)$, it has a countable dense subset.}, $\sigma$-finite measure spaces.  In discrete settings, i.e.~when $\mathcal{X}$ is countable, $\mu_X$ would typically be counting measure on $\mathcal{X}$; more generally, it may be the relevant Lebesgue measure when $\mathcal{X}$ is a Euclidean space, or an appropriate measure on basis coefficients in infinite-dimensional examples such as Example~\ref{Ex:InfDim} below.  Both $L^2(\mu_X)$ and $L^2(\mu_Y)$ are then separable Hilbert spaces\footnote{Since we were unable to find this precise statement in the literature, we provide a proof in Lemma~\ref{Lemma:Separability}.}, so there exist orthonormal bases $(p_j^X)_{j \in \mathcal{J}}$ and $(p_k^Y)_{k \in \mathcal{K} }$ of $L^2(\mu_X)$ and $L^2(\mu_Y)$ respectively, where $\mathcal{J}$ and $\mathcal{K}$ are countable indexing sets. Writing $\mu:=\mu_X \otimes \mu_Y$ for the product measure on $\mathcal{X} \times \mathcal{Y}$, the product space $L^2(\mu)$ is also a separable Hilbert space\footnote{Likewise, we prove this statement in Lemma~\ref{Lemma:Separability2}.}, and has an orthonormal basis given by $(p_{jk})_{j \in \mathcal{J}, k \in \mathcal{K}}$, where $p_{jk}(\cdot,\ast):=p_j^X(\cdot)p_k^Y(\ast)$. 

We may now define the subset $\mathcal{F}$ of $L^2(\mu)$ that consists of all density functions, that is
\[
	\mathcal{F}:= \biggl\{ f \in L^2(\mu): \essinf_{(x,y) \in \mathcal{X} \times \mathcal{Y}} f(x,y) \geq 0, \int_{\mathcal{X} \times \mathcal{Y}} f \,d \mu =1 \biggr\}.
\]
Given $f \in \mathcal{F}$, we may define the marginal density $f_X$ by
\[
	f_X(x):= \int_{\mathcal{Y}} f(x,y) \,d\mu_Y(y), 
\]
and we may analogously define $f_Y$. From now on we will work over the restricted space $\mathcal{F}^*:=\{f \in \mathcal{F} : f_X \in L^2(\mu_X), f_Y \in L^2(\mu_Y) \}$, though we note that when $\mu_X$ and $\mu_Y$ are finite measures, we have $\mathcal{F}^* = \mathcal{F}$. 
For $f \in \mathcal{F}^*$, $j \in \mathcal{J}$ and $k \in \mathcal{K}$ we may define the coefficients
\[
	a_{jk}(f):= \! \int_{\mathcal{X} \times \mathcal{Y}} f p_{jk} \,d\mu, \! \quad a_{j\bullet}(f):= \! \int_{\mathcal{X}} f_X p_j^X \,d\mu_X, \! \quad a_{\bullet k}(f):= \! \int_{\mathcal{Y}} f_Y p_k^Y \, d \mu_Y.
      \]
Then
        \[
          f =  \sum_{j \in \mathcal{J}} \sum_{k \in \mathcal{K}} a_{jk}(f)p_{jk}, \quad f_X =  \sum_{j \in \mathcal{J}} a_{j\bullet}(f)p_j^X, \quad f_Y =  \sum_{k \in \mathcal{K}} a_{\bullet k}(f)p_k^Y.
        \]
We may therefore define the measure of dependence
\begin{align*}
  D(f) &:= \int_{\mathcal{X} \times \mathcal{Y}} \bigl\{f(x,y)-f_X(x)f_Y(y)\bigr\}^2 \, d \mu(x,y) \\
	& \phantom{:}= \sum_{j \in \mathcal{J}, k\in \mathcal{K}} \bigl\{ a_{jk}(f)-a_{j\bullet}(f)a_{\bullet k}(f)\bigr\}^2,
\end{align*}
which, for $(X,Y) \sim f$, has the property that $D(f)=0$ if and only if $X \perp \!\!\! \perp Y$.

Given a sample $(X_1,Y_1),\ldots,(X_n,Y_n)$ of independent and identically distributed copies of the pair $(X,Y)$, we wish to test the null hypothesis $H_0: X \perp \!\!\! \perp Y$ of independence.  A randomised independence test is measurable function $\psi: (\mathcal{X} \times \mathcal{Y})^n \rightarrow [0,1]$, with the interpretation that, after observing $(X_1,Y_1,\ldots,X_n,Y_n)=(x_1,y_1,\ldots,x_n,y_n)$, we reject $H_0$ with probability $\psi(x_1,y_1,\ldots,x_n,y_n)$.  We write $\Psi$ for the set of all such randomised independence tests.  Further, define the null space $\mathcal{P}_0$ as the set of all distributions on $\mathcal{X} \times \mathcal{Y}$ of pairs $(X,Y)$ such that $X \perp \!\!\! \perp Y$, and, for a given $\alpha \in (0,1)$, define the set of valid size-$\alpha $ independence tests
\begin{equation}
\label{Eq:psialpha}
	\Psi(\alpha):= \biggl\{ \psi \in \Psi: \sup_{P \in \mathcal{P}_0} \mathbb{E}_P(\psi) \leq \alpha \biggr\}.
\end{equation}

The first part of Theorem~\ref{Thm:Hardness} below provides a preliminary result on the hardness of the independence testing problem when the alternative hypothesis $H_1$ consists of all densities $f \in \mathcal{F}^*$ of $(X,Y)$ that satisfy a lower bound constraint on $D(f)$.  In fact, the result can be stated more generally, allowing in addition for the possibility of a constraint on the smoothness of the alternatives that we consider.  To this end, for an array $\theta=(\theta_{jk})_{j \in \mathcal{J}, k \in \mathcal{K}} \in [0,\infty]^{\mathcal{J} \times \mathcal{K}}$, we define
\[
	S_\theta(f):= \sum_{j \in \mathcal{J}} \sum_{k \in \mathcal{K}} \theta_{jk}^2 \bigl\{a_{jk}(f)-a_{j\bullet}(f)a_{\bullet k}(f)\bigr\}^2.
      \]
Observe that when $\theta = 0_{\mathcal{J} \times \mathcal{K}}$, any non-negative upper bound on $S_\theta(f)$ becomes vacuous, so that no smoothness constraint is imposed. This definition of smoothness is motivated by the nonparametric statistics literature \citep[e.g.][]{Laurent1996}.  An attractive feature is that, in contrast to some prior literature, smoothness is only imposed on the difference between the joint density and the product of the marginals, rather than on the individual densities themselves; \citet{MALM2019} also adopt a similar approach to ours in this respect. 
 At a high level, the first part of Theorem~\ref{Thm:Hardness} is inspired by the work of \citet{Janssen00} and \citet{ShahPeters2020} on the hardness of goodness-of-fit testing and conditional independence testing respectively, though the proofs are completely different.  The second part complements the first, as discussed below.  Note that when $\mu_X$ is a probability measure, the constant function~$1$ belongs to $L^2(\mu_X)$, so can be included in our basis (as below).       
\begin{thm}
\label{Thm:Hardness}
Suppose that $\mu_X$ and $\mu_Y$ are probability measures and that there exist $j_0\in \mathcal{J}$ and $k_0 \in \mathcal{K}$ such that $p_{j_0}^X(x)=p_{k_0}^Y(y)=1$ for all $x \in \mathcal{X}$ and $y \in \mathcal{Y}$. Let $n \in \mathbb{N}$ and $\alpha \in [0,1]$, and let $\psi \in \Psi$ be such that $\mathbb{E}_{p_{j_0 k_0}}(\psi) \leq \alpha$. Let $\theta =(\theta_{jk})_{j \in \mathcal{J}, k \in \mathcal{K}} \in [0,\infty)^{\mathcal{J} \times \mathcal{K}}$ be given and, for $T \in [0,\infty)$, define
\[
	\mathcal{M}_\theta(T):= \bigl\{ (j,k) \in (\mathcal{J} \setminus \{j_0\}) \times (\mathcal{K} \setminus \{k_0\}): \theta_{jk} \leq T\bigr\}.
\]
Let $\underline{\theta}:=\inf_{j \in \mathcal{J}, k \in \mathcal{K}} \theta_{jk}$.  Then, for any $\epsilon>0$, any $\rho \in (0, 1/\sup_{j,k} \|p_{jk}\|_\infty]$ and any $r \in (\underline{\theta}\rho,\infty)$, there exists $f^* \in \mathcal{F}$ with $S_\theta(f^*) \leq r^2$ and $D(f^*)=\rho^2$ such that
\[
	\mathbb{E}_{f^*}(\psi) \leq \alpha + \epsilon + \biggl[ \frac{\{(1+\rho^2)^n -1\}\alpha}{|\mathcal{M}_{\theta}(r/\rho)|} \biggr]^{1/2}.
\]
Moreover, there exists a permutation test $\psi_{f^*} \in \Psi(\alpha)$ such that given any $\beta \in (0,1-\alpha)$, we can find $C = C(\alpha,\beta) > 0$ with the property that $\mathbb{E}_{f^*}(\psi_{f^*}) \geq 1- \beta$ whenever $n > C/\rho^2$.
\end{thm}
As a first conclusion we can draw from Theorem~\ref{Thm:Hardness}, consider taking $\theta=0_{\mathcal{J} \times \mathcal{K}}$, so that $|\mathcal{M}_{\theta}(r/\rho)| = (|\mathcal{J}|-1)(|\mathcal{K}|-1)$.  In this case, Theorem~\ref{Thm:Hardness} shows that in infinite-dimensional problems (where $|\mathcal{J} \times \mathcal{K}| = \infty$) with probability measures as base measures, there are no valid tests of independence that have uniformly non-trivial power against alternatives of the form $\{f \in \mathcal{F} : D(f) \geq \rho^2\}$, at least for $\rho>0$ sufficiently small.  The second part of the theorem then implies that in this setting there is no uniformly most powerful test.  Thus, to develop a theory of minimax separation rates for independence testing, it is necessary to make additional assumptions about the structure of the alternative hypothesis.  More generally, under the conditions of Theorem~\ref{Thm:Hardness}, whenever the set $\mathcal{M}_\theta(r/\rho)$ is infinite, there are no valid uniformly non-trivial independence tests against alternatives $f \in \mathcal{F}$ with $S_\theta(f) \leq r^2$ and $D(f) \geq \rho^2$.  We will therefore assume the following in much of our subsequent work:
\begin{description}
\item[\textbf{(A1)}] The sets $\{(j,k) \in \mathcal{J} \times \mathcal{K} : \theta_{jk} \leq T\}$ are finite for each $T \in (0,\infty)$.
\end{description}
Motivated by Theorem~\ref{Thm:Hardness} above, for $\Xi := [0,\infty]^{\mathcal{J} \times \mathcal{K}} \times (0,\infty) \times [1,\infty)$, for $\xi=(\theta,r,A) \in \Xi$ and for $\rho > 0$, we will consider the space of alternatives given by
\[
	\mathcal{F}_\xi(\rho) := \Bigl\{ f \in \mathcal{F} : D(f) \geq \rho^2, S_\theta(f) \leq r^2, \max(\|f\|_\infty, \|f_X\|_\infty, \|f_Y\|_\infty) \leq \! A \Bigr\}.
\]
Although we make assumptions about the smoothness of our alternatives, we will not make any assumptions about the null distributions, and the fact that we are using a permutation test will guarantee uniform, non-asymptotic control of the probability of Type I error.  In other words, we will prove that our test $\psi$ belongs to $\Psi(\alpha)$ in~\eqref{Eq:psialpha}.

Given $n \in \mathbb{N}$, $\alpha \in (0,1)$, $\xi = (\theta,r,A) \in \Xi$ and $\rho > 0$ we define the minimax risk with respect to $\mathcal{F}_\xi(\rho)$ as 
\[
	\mathcal{R}(n,\alpha,\xi,\rho) := \alpha + \inf_{\psi \in \Psi(\alpha)} \sup_{f \in \mathcal{F}_\xi(\rho)} \mathbb{E}_f(1-\psi),
\]
with the convention that $\mathcal{R}(n,\alpha,\xi,\rho) := \alpha$ if $\mathcal{F}_\xi(\rho) = \emptyset$.  If we are also given a desired probability of Type II error $\beta \in (0,1-\alpha)$, then we can consider the minimax separation radius
\[
	\rho^*(n,\alpha,\beta,\xi):= \inf \bigl\{ \rho > 0 : \mathcal{R}(n,\alpha,\xi,\rho) \leq \alpha + \beta \bigr\}.
\]

\section{Upper bounds}
\label{Sec:MainResults}



We now introduce our USP test that will allow us to establish upper bounds on the minimax separation $\rho^*$. This is based on a $U$-statistic estimator of $D(f)$ with kernel
\begin{align}
\label{Eq:h}
	h \bigl( &(x_1,y_1), (x_2,y_2), (x_3,y_3), (x_4,y_4) \bigr) \nonumber \\
	&:= \sum_{(j,k) \in \mathcal{M}} \{ p_{jk}(x_1,y_1)p_{jk}(x_2,y_2) - 2p_{jk}(x_1,y_1) p_{jk}(x_2,y_3) +p_{jk}(x_1,y_2)p_{jk}(x_3,y_4) \bigr\},
\end{align} 
where $\mathcal{M} \subseteq \mathcal{J} \times \mathcal{K}$ is a truncation set to be chosen later.  The motivation for this definition comes from the observation that for any $f \in \mathcal{F}^*$ and when $\mathcal{M} = \mathcal{J} \times \mathcal{K}$, we have
\begin{equation}
\label{Eq:DEstimate}
\mathbb{E}_f\bigl\{h\bigl( (X_1,Y_1), (X_2,Y_2), (X_3,Y_3), (X_4,Y_4) \bigr)\bigr\} = D(f);
\end{equation}
moreover, as we will see in the proof of Theorem~\ref{Thm:UpperBound} below, whenever $\Pi$ is a uniformly random element of $\mathcal{S}_n$ that is independent of the data, we have 
\[
\mathbb{E}_f\bigl\{h\bigl( (X_1,Y_{\Pi(1)}), (X_2,Y_{\Pi(2)}), (X_3,Y_{\Pi(3)}), (X_4,Y_{\Pi(4)}) \bigr)\bigr\} = 0.
\]
To reduce the effects of noise accumulation in the estimation of the summands, it will typically be necessary to choose $\mathcal{M}$ in~\eqref{Eq:h} to be a proper subset of $\mathcal{J} \times \mathcal{K}$.  The equality in~\eqref{Eq:DEstimate} then no longer holds exactly for every $f \in \mathcal{F}^*$, but an appropriate choice of $\mathcal{M}$ allows us to control the bias-variance trade-off. 

For $m \geq 2$, let $\mathcal{I}_m:= \{(i_1,\ldots,i_m) \in [n]^m: i_1,\ldots,i_m \text{ all distinct}\}$. For $x = (x_1,\ldots,x_n) \in \mathcal{X}^n$ and $y = (y_1,\ldots,y_n) \in \mathcal{Y}^n$, it is convenient to define
\[
  \mathcal{T}_{x,y} := \{(x_i,y_i):i \in [n]\},
\]
and for $\sigma \in \mathcal{S}_n$, set $\mathcal{T}_{x,y}^{(\sigma)} := \{(x_i,y_{\sigma(i)}):i \in [n]\}$.  Given independent pairs $\mathcal{T}_{X,Y} := \{(X_i,Y_i):i=1,\ldots,n\}$ with $n \geq 4$, we consider the test statistic
\begin{align*}
	\hat{D}_n = \hat{D}_n^{\mathcal{M}}\bigl(\mathcal{T}_{X,Y}\bigr) := \frac{1}{4! \binom{n}{4}} \sum_{(i_1,\ldots,i_4) \in \mathcal{I}_4} h \bigl( (X_{i_1},Y_{i_1}), \ldots, (X_{i_4},Y_{i_4}) \bigr). 
\end{align*}
To define the critical value for our test, let $B \in \mathbb{N}$ and generate an independent sequence of uniform random permutations $\Pi_1,\ldots,\Pi_B$ taking values in $\mathcal{S}_n$, independently of $\mathcal{T}_{X,Y}$.  It is important to note that we can typically choose $B$ to be much smaller than $n!$ (the number of distinct permutations in $\mathcal{S}_n)$; indeed, the choice $B=99$ is common for permutation tests.  For each $b \in [B]$, we construct the null statistics
\begin{equation}
\label{Eq:Dnhat}
	\hat{D}_n^{(b)}:= \hat{D}_n^{\mathcal{M}} \bigl(\mathcal{T}_{X,Y}^{(\Pi_b)}\bigr).
\end{equation}
Finally, we can define the p-value
\begin{equation}
\label{Eq:pvalue}
	P:= \frac{1 + \sum_{b=1}^B \mathbbm{1}_{\{ \hat{D}_n \leq \hat{D}_n^{(b)} \}}}{1+B},
\end{equation}
and reject the null hypothesis if $P \leq \alpha$.  Formally, this corresponds to the randomised test $\psi_\alpha \in \Psi$, given by
\[
	\psi_\alpha \bigl(\mathcal{T}_{x,y} \bigr):= \mathbb{P}\biggl( 1 + \sum_{b=1}^B \mathbbm{1}_{\{ \hat{D}_n^{\mathcal{M}}(\mathcal{T}_{x,y}) \leq \hat{D}_n^{\mathcal{M}}(\mathcal{T}_{x,y}^{(\Pi_b)})\}} \leq (1+B)\alpha \biggr),
\]
where the only randomness here is in the permutations $\Pi_1,\ldots,\Pi_B$. Then, on observing $\mathcal{T}_{X,Y}$, we do indeed reject $H_0$ with probability $\psi_\alpha(\mathcal{T}_{X,Y})$. Under the null hypothesis, the sequence of data sets $\mathcal{T}_{X,Y},\mathcal{T}_{X,Y}^{(\Pi_1)},\ldots,\mathcal{T}_{X,Y}^{(\Pi_B)}$ is exchangeable, so every ordering of the components of $\bigl(\hat{D}_n^{\mathcal{M}}(\mathcal{T}_{X,Y}),\hat{D}_n^{\mathcal{M}}(\mathcal{T}_{X,Y}^{(\Pi_1)}),\ldots,\hat{D}_n^{\mathcal{M}}(\mathcal{T}_{X,Y}^{(\Pi_B)})\bigr)$ is equally likely if we break ties uniformly at random.  In particular, the rank of $\hat{D}_n^{\mathcal{M}}(\mathcal{T}_{X,Y})$ among these $B+1$ observations, which is a lower bound on the numerator in~\eqref{Eq:pvalue}, is uniformly distributed on $\{1,\ldots,B+1\}$, so $\psi_\alpha \in \Psi(\alpha)$.  

A naive implementation of the test has computational complexity $O(n^4 B |\mathcal{M}|)$, due to the need to calculate fourth order $U$-statistics. However, using an alternative representation of our test statistic inspired by \citet{Song2012}, we can reduce the complexity to $O\bigl(n^2B( |\mathcal{J}_0 | + |\mathcal{K}_0|\bigr)$ when $\mathcal{M}=\mathcal{J}_0 \times \mathcal{K}_0$.  See Section~\ref{Sec:CompTrick} for further details. 

The following theorem provides a general upper bound on the minimax separation rate, and is obtained using the above test.
\begin{thm}
\label{Thm:UpperBound}
Fix $\alpha,\beta \in (0,1)$ such that $\alpha+\beta<1$ and let $\xi=(\theta,r,A) \in \Xi$. Then there exists $C = C(\alpha,\beta,A) > 0$ such that when $n \geq 16$, we have
\begin{align*}
	\rho^*(n, \alpha, \beta, \xi) \leq C \inf_{\mathcal{M} \subseteq \mathcal{J} \times \mathcal{K}} \max \biggl\{ \frac{r}{\inf \{\theta_{jk} : (j,k) \not\in \mathcal{M} \} }, \frac{\min(\|h\|_\infty^{1/2}, |\mathcal{M}|^{1/4})}{n^{1/2}}, \frac{1}{n^{1/2}} \biggr\}.
\end{align*}
\end{thm}
An explicit upper bound showing the dependence of $C$ on its arguments is given in~\eqref{Eq:FullDetail} in the proof of Theorem~\ref{Thm:UpperBound}.  To give a heuristic explanation of the terms in the bound in Theorem~\ref{Thm:UpperBound}, observe that in order for our test to have high power, we want $\rho^2$ to dominate the sum of the bias of the test statistic and its standard deviation under the null.  The first term represents this bias, which is induced by truncating the sum in~\eqref{Eq:h} to indices that lie in $\mathcal{M}$.  The second term arises from bounding the variance of our $U$-statistic in terms of the symmetrised kernel $\bar{h}$, defined formally in~\eqref{Eq:barh} below.  More precisely, under the null, our test statistic is a degenerate $U$-statistic, i.e.~$\mathbb{E}\bigl\{\bar{h}\bigl((x,y),(X_2,Y_2),(X_3,Y_3),(X_4,Y_4)\bigr)\bigr\} = 0$ for all $x \in \mathcal{X}, y \in \mathcal{Y}$, so its variance can be bounded above by a constant multiple of $n^{-2}\mathrm{Var} \bigl\{\bar{h}\bigl((X_1,Y_1),(X_2,Y_2),(X_3,Y_3),(X_4,Y_4)\bigr)\bigr\}$.  This latter expression can in turn be bounded by $\min\bigl(\|h\|_\infty^2,|\mathcal{M}|\bigr)/n^2$.  The final term in the maximum represents the parametric rate of convergence, and is generally unavoidable.

\subsection{Discrete case}
\label{Sec:Discrete}

As a first application of Theorem~\ref{Thm:UpperBound}, consider the relatively simple problem of testing independence with discrete data, where for some $J,K \in \mathbb{N} \cup \{\infty\}$ we have $\mathcal{X} = [J]$ and  $\mathcal{Y}=[K]$ and we take $\mu_X$ and $\mu_Y$ to be the counting measures on $\mathcal{X}$ and $\mathcal{Y}$ respectively.  For $j,x \in [J]$ and $ k,y \in [K]$ we can define the basis functions $p_j^X(x) := \mathbbm{1}_{\{x=j\}}$ and $p_k^Y(y):=\mathbbm{1}_{\{y=k\}}$. In this case we have $\|h\|_\infty \leq 2$ independently of $\mathcal{M}$, and we may take $\mathcal{M}=[J] \times [K]$ so that there is in fact no truncation and our test statistic is an unbiased estimator of $D$. Note here that, since $\mu_X$ and $\mu_Y$ are not probability measures, Theorem~\ref{Thm:Hardness} does not apply, and we will see that no structural assumptions are necessary on the alternative hypothesis.  Indeed, we take $\xi=(0_{[J] \times [K]},1,1) \in \Xi$, so that our alternative hypothesis class is simply
\[
	\mathcal{F}_\xi(\rho) = \biggl\{ f \in \mathcal{F} : \sum_{j \in [J], k \in [K]} \{f(j,k) - f_X(j)f_Y(k)\}^2 \geq \rho^2 \biggr\}.
      \]
The following result is a straightforward corollary of Theorem~\ref{Thm:UpperBound}, noting that the cases where $n < 16$ can be handled using the fact that $\rho^*(n,\alpha,\beta,\xi) \leq 2^{1/2}$ for all $n$.
\begin{cor}
\label{Cor:Discrete}
Fix $\alpha,\beta \in (0,1)$ such that $\alpha+\beta < 1$. Then there exists $C=C(\alpha,\beta) \in (0,\infty)$ such that 
\[
	\rho^*(n,\alpha,\beta,\xi) \leq Cn^{-1/2}.
      \]
\end{cor}
This behaviour should be contrasted with that found in \citet{Diakonikolas2016}, where the strength of the dependence is measured by the $L_1$ distance rather than the $L_2$ distance, and where the minimax optimal separation rates depend on the alphabet sizes; in fact, they are given by $\frac{J^{1/4}K^{1/4}}{n^{1/2}} \max\bigl(1,J^{1/4}/n^{1/4},K^{1/4}/n^{1/4}\bigr)$.  

In fact, in this discrete setting, we can give a relatively simple, explicit form for the test.  To this end, for $j \in [J]$, $k \in [K]$, let $N_{jk} := \sum_{i=1}^n \mathbbm{1}_{\{X_i = j,Y_i=k\}}$, let $N_{j+} := \sum_{k =1}^K N_{jk}$, and let $N_{+k} := \sum_{j =1}^J N_{jk}$.  Then, omitting terms that only depend on $N_{j+}$ and $N_{+k}$ (and hence remain fixed under permutation, so are irrelevant for the test), our test statistic becomes
\[
\hat{T}_n := \frac{1}{n(n-3)}\sum_{j=1}^J \sum_{k=1}^K \biggl(N_{jk} - \frac{N_{j+}N_{+k}}{n}\biggr)^2 - \frac{4}{n^2(n-2)(n-3)}\sum_{j=1}^J\sum_{k=1}^K N_{jk}N_{j+}N_{+k}.
\]
Thus, the test statistic can be computed using only the contingency table counts, as opposed to the original data.  Moreover, the permutated data sets may also be generated using only these counts: indeed, writing $N_{jk}^{(1)}$ for the $(j,k)$th cell count under an independent, uniformly random permutation of the original data, we have
\[
\mathbb{P}\bigl((N_{jk}^{(1)}) = (n_{jk})|\mathcal{T}_{X,Y}\bigr) = \frac{\bigl(\prod_{j =1}^J N_{j+}!\bigr) \bigl(\prod_{k=1}^K N_{+k}!\bigr)}{n!\prod_{j=1}^J\prod_{k=1}^K n_{jk}!},
\]
whenever $(n_{jk})$ is such that $\sum_{k=1}^K n_{jk} = N_{j+}$ for all $j \in [J]$ and $\sum_{j=1}^J n_{jk} = N_{+k}$ for all $k \in [K]$.  This formula simplifies the computation of the permuted data sets, and one can sample from this distribution using Patefield's algorithm \citep{Patefield1981}, which is implemented in the \texttt{R} function \texttt{r2dtable}.  

\subsection{Sobolev and infinite-dimensional examples}
\label{SubSec:Sobolev}

To apply Theorem~\ref{Thm:UpperBound} in general, when a useful bound on $\|h\|_\infty$ is not available, we instead control the right-hand side by controlling $|\mathcal{M}|$. We remark that, when there exist $j_0 \in \mathcal{J}$ and $k_0 \in \mathcal{K}$ such that $p_{j_0}^X(x)=p_{k_0}^Y(y)=1$ for all $x,y$, then $a_{j_0k}=a_{\bullet k},a_{jk_0}=a_{j\bullet},a_{j_0 \bullet}=1,a_{\bullet k_0}=1$, so the $j=j_0$ and $k=k_0$ terms do not contribute to the value of $D(\cdot)$ and $S_\theta(\cdot)$ does not depend on $(\theta_{j_0k})_{k \in \mathcal{K}}$ or $(\theta_{jk_0})_{j \in \mathcal{J}}$.  Thus the choice of $\mathcal{M}$ in the definition of $\hat{D}_n^\mathcal{M}$ does not need to contain any $(j,k)$ with $j=j_0$ or $k=k_0$. For notational convenience, we will adopt the convention that, in such cases, $\theta_{jk}=\infty$ if either $j=j_0$ or $k=k_0$.  When \textbf{(A1)} holds it is possible to arrange $\{\theta_{jk}:\theta_{jk} < \infty\}$ in increasing order, so that there exists a bijection $\omega : \mathbb{N} \rightarrow \{(j,k): \theta_{jk} < \infty\}$ such that $\theta_{\omega(1)} \leq \theta_{\omega(2)} \leq \ldots$. Given $t \in (0,\infty)$, define\footnote{Here and throughout, if $\omega(m) = (j,k)$, we interpret $\theta_{\omega(m)}$ as $\theta_{jk}$ and $p_{\omega(m)}$ as $p_{jk}$.}
\[
	m_0(t) := \min \bigl\{ m \in \mathbb{N} : m^{1/2} \theta_{\omega(m)}^2 > t \bigr\}.
      \]
We can now simplify the conclusion of Theorem~\ref{Thm:UpperBound} under \textbf{(A1)}:
\begin{cor}
\label{Cor:Main}
Fix $\alpha,\beta \in (0,1)$ such that $\alpha+\beta<1$ and let $\xi=(\theta,r,A) \in \Xi$.  Assume \textbf{(A1)}.  Then there exists $C = C(\alpha,\beta,A) > 0$ such that when $n \geq 16$, we have
\begin{equation}
\label{Eq:m0bound}
\rho^*(n, \alpha, \beta, \xi) \leq C \inf_{m \in \mathbb{N}} \max \biggl\{ \frac{r}{\theta_{\omega(m)}}, \frac{m^{1/4}}{n^{1/2}} \biggr\} \leq \frac{C m_0^{1/4}(nr^2)}{n^{1/2}}.
\end{equation}
\end{cor}
We now further specialise our upper bound by making a specific choice of $\mathcal{J}$, $\mathcal{K}$ and weights $(\theta_{jk}:j \in \mathcal{J},k \in \mathcal{K})$; such a choice yields a concrete upper bound on the minimax rate of independence testing for densities lying in a Sobolev space, as we illustrate in the example that follows.  See Example~\ref{Ex:Sobolev2} and Proposition~\ref{Thm:FourierLowerBound} below for a discussion of optimality of this bound.  
\begin{cor}
\label{Cor:Sobolev}
Fix $\alpha,\beta \in (0,1)$ such that $\alpha+\beta<1$, fix $d_X,d_Y \in \mathbb{N}$ and $s_X,s_Y,r,A >0$. Writing $\mathcal{J}=\mathbb{N}_0^{d_X}$, $\mathcal{K}=\mathbb{N}_0^{d_Y}$, set $\theta_{jk}=\|j\|_1^{s_X} \vee \|k\|_1^{s_Y}$ whenever $j \neq 0_{[d_X]}$ and $k \neq 0_{[d_Y]}$ and $\theta_{jk} = \infty$ otherwise.  Then, with $\theta = \{\theta_{jk}:j \in \mathcal{J},k \in \mathcal{K}\}$, there exists $C=C(d_X,d_Y,\alpha,\beta,A)>0$ such that if $n \geq 16$ and $nr^2 \geq 1$, then
\[
	\rho^*(n,\alpha,\beta,\xi) \leq C \Bigl( \frac{r^d}{n^{2s}} \Bigr)^{1/(4s+d)},
\]
where $d:=d_X+d_Y$, $s:=d/(d_X/s_X+d_Y/s_Y)$ and $\xi = (\theta,r,A)$.
\end{cor}
The upper bound in Corollary~\ref{Cor:Sobolev} is obtained using our $U$-statistic permutation test.  Here, $\theta_{\omega(m)} \asymp_{s_X,s_Y,d_X,d_Y} m^{s/d}$, so we can balance the two terms in the maximum in Corollary~\ref{Cor:Main} by taking $\mathcal{M} = \bigl\{\omega(1),\ldots,\omega(m)\bigr\}$ with $m \asymp_{s_X,s_Y,d_X,d_Y}(nr^2)^{2d/(4s+d)}$.  A natural application of (a minor variant of) this corollary is to absolutely continuous data, which for simplicity we restrict to lie in $[0,1]^{d_X} \times [0,1]^{d_Y}$.  In this setting, the Fourier basis functions are an obvious choice. 
\begin{example}
  \label{Ex:Sobolev}
  Let $\mathcal{X} = [0,1]^{d_X}$ and $\mathcal{Y} = [0,1]^{d_Y}$, equipped with $d_X$-dimensional Lebesgue measure $\mu_X$ and $d_Y$-dimensional Lebesgue measure $\mu_Y$ respectively.  Taking $\mathcal{J} := \bigl\{(a,m): a \in \{0,1\},m \in \mathbb{N}_0^{d_X}\bigr\} \setminus \{(1,0_{[d_X]})\}$ and $\mathcal{K} := \bigl\{(a,m): a \in \{0,1\},m \in \mathbb{N}_0^{d_Y}\bigr\} \setminus \{(1,0_{[d_Y]})\}$, we can define the orthonormal Fourier basis functions\footnote{The fact that these functions form an orthonormal basis for $L^2\bigl([0,1]^{d_X}\bigr)$ follows from a very similar (in fact, slightly simpler) argument to that given in Lemma~\ref{Lemma:StoneWeierstrass}, which relates to Example~\ref{Ex:InfDim} below.  The main difference is that in this example our functions are defined on finite-dimensional spaces.} for $L^2\bigl([0,1]^{d_X}\bigr)$ given by $p_{0,0}^X := 1$ and for $m = (m_1,\ldots,m_{d_X}) \neq 0_{[d_X]}$,
  \begin{equation}
    \label{Eq:Fourier}
 p_{a,m}^X(x_1,\ldots,x_{d_X}) :=  2^{1/2} \mathrm{Re} \biggl( e^{-a\pi i/2} \prod_{\ell=1}^{d_X} e^{-2\pi i m_\ell x_\ell} \biggr).
\end{equation}
The Fourier basis functions $\{p_{a,m}^Y:(a,m) \in \mathcal{K}\}$ for $L^2\bigl([0,1]^{d_Y}\bigr)$ are defined similarly, but with $d_Y$ replacing $d_X$.  For $j = (a_X,m_X) \in \mathcal{J}$, $k = (a_Y,m_Y) \in \mathcal{K}$ and $s_X,s_Y > 0$, we can then take $\theta_{jk} = \|m_X\|_1^{s_X} \vee \|m_Y\|_1^{s_Y}$, $\theta = \{\theta_{jk}: j \in \mathcal{J},k \in \mathcal{K}\}$ and $\xi = (\theta,r,A) \in \Xi$ to conclude from Corollary~\ref{Cor:Main} that $\rho^*(n,\alpha,\beta,\xi) \leq C \bigl(r^d/n^{2s}\bigr)^{1/(4s+d)}$ when $n \geq 16$ and $nr^2 \geq 1$, as in Corollary~\ref{Cor:Sobolev}.
\end{example}
We mention here that \citet{LiYuan2019} and \citet{MALM2019} consider Gaussian kernel-based Hilbert--Schmidt Independence Criterion tests of independence in similar Sobolev settings to that in Example~\ref{Ex:Sobolev}.  Assuming the same level of Sobolev smoothness $s$ for both the joint and marginal distributions, \citet{LiYuan2019} show that the critical consistency level is of order $n^{-2s/(4s+d)}$ over tests that have asymptotically nominal size.  \citet{MALM2019} obtain the same rate in a non-asymptotic setting and only impose smoothness conditions on the difference between the joint and marginal distributions, at the expense of restricting the smoothness $s$ to be at most $2$, and having bounded null densities.

In fact, Corollary~\ref{Cor:Main} also provides explicit upper bounds for certain infinite-dimensional models.  Corollary~\ref{Cor:InfDim} below illustrates this for a particular choice of $\mathcal{J}$, $\mathcal{K}$ and weights $(\theta_{jk}:j \in \mathcal{J},k \in \mathcal{K})$.
\begin{cor}[BKS(2020)]
  \label{Cor:InfDim}
Let $\mathbb{N}_0^{< \infty}:= \{m = (m_1,m_2,\ldots) \in \mathbb{N}_0^{\mathbb{N}} : \sum_{\ell=1}^\infty \mathbbm{1}_{\{m_\ell \neq 0\}} < \infty \}$, and let $\mathcal{J}=\mathcal{K}:=\{(a,m): a \in \{0,1\}, m \in \mathbb{N}_0^{< \infty}\} \setminus \{(1,0)\}$.  For $m = (m_1,m_2,\ldots) \in \mathbb{N}_0^{< \infty}$, write $|m| := \max_{\ell \in \mathbb{N}} \ell^2 m_{\ell}$, and if $j=(a,m) \in \mathcal{J}$, write $|j|:=|m|$.  For $j \in \mathcal{J}, k \in \mathcal{K}$ with $|j| \wedge |k| > 0$, and $s_X,s_Y>0$, set
\[
	\theta_{jk} = \exp( s_X |j|^{1/2}) \vee \exp (s_Y |k|^{1/2}),
\]
and if either $|j|=0$ or $|k|=0$ then set $\theta_{jk}=\infty$.  
Define the increasing function $M:[0,\infty) \rightarrow [0,\infty)$ by
\[
	M(t):= \exp\biggl( \sum_{\ell=1}^\infty \log\Bigl( 1+ \Bigl\lfloor \frac{t}{\ell^2} \Bigr\rfloor \Bigr) \biggr)-1
\]
and write
\[
	m_{0,s_X,s_Y}(t):= \min\biggl\{ m \in \mathbb{N} :  M\biggl( \frac{\log^2(t/m^{1/2})}{4s_X^2} \biggr) M\biggl( \frac{\log^2(t/m^{1/2})}{4s_Y^2} \biggr)  < \frac{m}{4} \biggr\}.
\]
(i) Fix $\alpha,\beta \in (0,1)$ such that $\alpha+\beta<1$ and fix $r,s_X,s_Y,A >0$.  Then, with $\xi=(\theta,r,A) \in \Xi$ there exists $C=C(\alpha,\beta,s_X,s_Y,A)>0$ such that when $n \geq 16$ and $nr^2 \geq C$ we have
\[
	\rho^*(n, \alpha, \beta, \xi) \leq \frac{Cm_{0,s_X,s_Y}^{1/4}(nr^2)}{n^{1/2}}.
\]
(ii) Writing $s:=2/(s_X^{-1}+s_Y^{-1})$ and given $\epsilon \in (0,4s)$, there exists $C'=C'(s_X,s_Y,\epsilon) > 0$ such that when $t \geq C'$ we have 
\[
	t^{\frac{2c_0-\epsilon}{2s+c_0}} \leq m_{0,s_X,s_Y}(t) \leq t^{\frac{2c_0+\epsilon}{2s+c_0}},
\]
where $c_0:=\sum_{\ell=1}^\infty \{\ell^{-1/2}-(\ell+1)^{-1/2}\} \log(1+\ell)=1.65\ldots$.
\end{cor}
We will see in Proposition~\ref{Prop:InfDimLowerBound} below that the rate given in the first part of Corollary~\ref{Cor:InfDim} is optimal in regimes of $n$ and $r$ of interest in the context of Example~\ref{Ex:InfDim} below. The second part of the corollary shows that, if we ignore subpolynomial factors in $nr^2$, then we have $\rho^*(n,\alpha,\beta,\xi) \lesssim_{\alpha,\beta,s_X,s_Y,A} (r^{c_0}/n^{s})^{1/(2s+c_0)}$. By comparison with Corollary~\ref{Cor:Sobolev}, we can therefore interpret $c_0$ as the `effective dimension' of each of $\mathcal{X}$ and $\mathcal{Y}$, when $\theta$ is selected in this way.
\begin{example}
  \label{Ex:InfDim}
  As an application of Corollary~\ref{Cor:InfDim}, consider the infinite-dimensional setting where $\mathcal{X} = \mathcal{Y} = [0,1]^{\mathbb{N}} :=\{(x_1,x_2,\ldots) : x_\ell \in [0,1] \text{ for all } \ell \in \mathbb{N} \}$, equipped with the Borel $\sigma$-algebra in the product topology, and where $\mu_X=\mu_Y$ is the distribution of an infinite sequence $(U_1,U_2,\ldots)$ of $\mathrm{Unif}[0,1]$ random variables.  It follows from an application of the Stone--Weierstrass theorem (see Lemma~\ref{Lemma:StoneWeierstrass}) that an orthonormal basis for $L^2(\mu_X)$ is then given by $\{ p_{a,m}^X(\cdot) : (a,m) \in \mathcal{J}\}$, where $p_{0,0}^X:=1$ and for $m \neq 0_{\mathbb{N}}$,
\[
	p_{a,m}^X(x_1,x_2,\ldots) :=  2^{1/2} \mathrm{Re} \biggl(e^{-a\pi i/2} \prod_{\ell=1}^\infty e^{-2\pi i m_\ell x_\ell} \biggr).
\]
We may take the same basis for $L^2(\mu_Y)$, so that $p_{a,m}^Y=p_{a,m}^X$ for all $a\in \{0,1\}$ and $m\in \mathbb{N}_0^{<\infty}$.  Then Corollary~\ref{Cor:InfDim} provides an upper bound on the minimax separation rate of independence testing in this example.
\end{example}

\section{Adaptation}
\label{Sec:Adaptation}

The practical implementation of our USP tests requires a choice of the truncation set $\mathcal{M}$.  The optimal choice of $\mathcal{M}$, which yields the separation rates described in the previous section, typically depends on both $\theta$ and $r$, which  may be unknown in practice.  In this section, we therefore describe adaptive versions of our tests, that do not require knowledge of any unknown parameters and whose minimax risk can be shown in many cases to be only slightly inflated compared with the optimal tests.  Our initial setting is rather general, but assumes that $\mathcal{J} \times \mathcal{K}$ has an ordering that is respected by every $\theta$ considered.  Since this assumption does not hold in the setting of Corollary~\ref{Cor:Sobolev} unless $s_X = s_Y$ (as the relative magnitudes of $s_X$ and $s_Y$ affect the ordering of $\theta$), we also illustrate the way in which this assumption can be relaxed, so that it remains possible to adapt to both of the unknown parameters separately in this Sobolev example.

To describe this initial setting, let $\omega: \mathbb{N} \rightarrow \mathcal{J} \times \mathcal{K}$ be injective, and, for a given $\theta_0>0$, let $\Theta(\omega,\theta_0) \subseteq [0,\infty]^{\mathcal{J} \times \mathcal{K}}$ denote the set of all $\theta = (\theta_{jk})_{j \in \mathcal{J}, k \in \mathcal{K}}$ such that $\omega$ is a bijection from $\mathbb{N}$ to $\{(j,k) \in \mathcal{J} \times \mathcal{K} : \theta_{jk} < \infty\}$ and
\[
	\theta_0 \leq \theta_{\omega(1)} \leq \theta_{\omega(2)} \leq \ldots.
\]
Here $\omega$ denotes an ordering of $\mathcal{J} \times \mathcal{K}$ that ranks the importance of departures from independence in each direction. In our Sobolev example with $s_X=s_Y$, we could take $\omega$ to be any ordering of $(\mathbb{N}_0^{d_X} \setminus \{0_{[d_X]}\} ) \times (\mathbb{N}_0^{d_Y} \setminus \{0_{[d_Y]}\} )$ such that, writing $(j_m,k_m) := \omega(m)$, we have that $\max(\|j_1\|_1,\|k_1\|_1) \leq \max(\|j_2\|_1,\|k_2\|_1) \leq \ldots$. Taking $\gamma := \lceil 2 \log_2 n \rceil$, let $K_* := \{ 2^j : j\in [\gamma]\}$. Our adaptive procedure can now be described as follows.  Given a desired Type II error probability $\beta \in (0,1-\alpha)$, for each $m \in K_*$, carry out the permutation test from Section~\ref{Sec:MainResults} with $\mathcal{M}=\{ \omega(1), \ldots, \omega(m)\}$ and $B \geq 2( \frac{\gamma}{\alpha \beta} -1)$ to yield p-values $p^{(1)},\ldots,p^{(\gamma)}$. If $\min_{i \in [\gamma]} p^{(i)} < \alpha/\gamma$, then we reject $H_0$. As we have applied a standard Bonferroni correction, the Type I error of this omnibus test is controlled at the level~$\alpha$. The following result concerns its power.

\begin{prop}
\label{Prop:GeneralAdapt}
Let $\omega$ and $\theta_0>0$ be as above, and suppose that $\alpha \in (0,1), \beta \in (0,1-\alpha), R_0 >0$ and $A \geq 1$. Assume further that $f \in \mathcal{F}_\xi(\rho)$ for some $\xi=(\theta,r,A) \in \Xi$ with $ \theta \in \Theta(\omega,\theta_0)$ and $r \in (0,R_0]$. Then there exists $C=C(\alpha,\beta,R_0,\theta_0,A)>0$ such that we reject~$H_0$ with probability at least $1-\beta$ whenever $n \geq C$ and
\[
	\rho \geq C \max \biggl\{ \frac{\log^{1/4}n}{n^{1/2}} m_0^{1/4} \biggl( \frac{n r^2}{ \log^{1/2}n} \biggr), \frac{\log^{1/2} n}{n^{1/2}} \biggr\}.
\]
\end{prop}
Comparing this result with the upper bound on the optimal separation in Corollary~\ref{Cor:Main}, we see that the price we pay for adaptation is that our effective sample size is reduced from $n$ to $n/ \log^{1/2} n$, at least provided that $m_0(n r^2 / \log^{1/2} n) \gtrsim \log n$.

As mentioned above, in some applications, the set $\mathcal{J} \times \mathcal{K}$ will not be naturally ordered. Nevertheless, it may be the case that $\mathcal{J}$ and $\mathcal{K}$ are ordered separately, and in these cases it is still possible to adapt to unknown parameters. Consider the setting of Corollary~\ref{Cor:Sobolev}, and define $\gamma_X := \lceil (2/d_X) \log_2n\rceil$ and $K_X:= \{ 2^j : j \in [\gamma_X]\}$ (with $\gamma_Y$ and $K_Y$ defined similarly).  Similarly to before, given a desired Type II error probability $\beta \in (0,1-\alpha)$, for each $(m_X,m_Y) \in K_X \times K_Y$, carry out the permutation test from Section~\ref{Sec:MainResults} with $\mathcal{M} \equiv \mathcal{M}_{m_X,m_Y} = \{(j,k) \in \mathbb{N}_0^{d_X}  \times \mathbb{N}_0^{d_Y} : 1 \leq \|j\|_1 \leq m_X, 1 \leq \|k\|_1 \leq m_Y \}$ and $B \geq 2( \frac{\gamma_X \gamma_Y}{\alpha \beta} -1)$ to yield p-values $\bigl\{p^{(m_X m_Y)}:(m_X,m_Y) \in K_X \times K_Y\bigr\}$.  This test again controls the Type I error at level $\alpha$, and the following result shows that the critical separation radius is inflated by at most a logarithmic factor in $n$.
\begin{prop}
\label{Prop:SobolevAdapt}
Assume the setting of Corollary~\ref{Cor:Sobolev}. Given $R_0 >0$, suppose that $r \leq R_0$. Then there exists $C=C(\alpha,\beta,R_0, s_X,s_Y,d_X,d_Y,A)>0$ such that we reject $H_0$ with probability at least $1-\beta$ whenever $n \geq C$ and
\begin{equation}
\label{Eq:SobolevAdaptBound}
	\rho \geq C  \biggl\{ \frac{r^d}{(n/\log n)^{2s}} \biggr\}^{1/(4s+d)}.
\end{equation}
\end{prop}
We note that a similar procedure could be applied in the setting of Corollary~\ref{Cor:InfDim} to obtain an adaptive test there too.  Finally in this section, we remark that in a more restricted setting it may be possible to improve the $\log n$ dependence to $\log \log n$ dependence using the very recent concentration results of \citet{KBW2020}.

\section{Lower bounds}
\label{Sec:LowerBounds}

The goal of this section is to provide lower bounds to allow us to study the optimality of our USP test in different contexts.  Slightly more precisely, we wish to determine the maximal departure from independence (measured in terms of our quantity $D(\cdot)$) that no valid independence test could reliably detect; equivalently, we seek the minimal separation level at which a valid independence test could have non-trivial power, uniformly over the alternatives in our classes.  To this end, we first prove a general lemma (Lemma~\ref{Lemma:LowerBound} below), and then illustrate how it can be applied in different settings of interest.

Our lower bound results actually apply to a weaker notion of minimax risk, and will hold in settings where our base measures on $\mathcal{X}$ and $\mathcal{Y}$ are probability measures, and where our orthonormal bases contain the constant function 1, so that there exist $j_0 \in \mathcal{J}$ and $k_0 \in \mathcal{K}$ such that $p_{j_0}^X(x) = 1$ and $p_{k_0}(y) = 1$ for all $x \in \mathcal{X}$ and $y \in \mathcal{Y}$. Define
\[
	\tilde{\mathcal{R}}(n,\xi,\rho):= \inf_{\psi \in \Psi(1)} \biggl\{ \mathbb{E}_{p_{j_0 k_0}}(\psi) + \sup_{f \in \mathcal{F}_\xi(\rho)} \mathbb{E}_f(1-\psi) \biggr\},
\]
which only controls the sum of the error probabilities, and only considers a simple null, and further define
\[
	\tilde{\rho}^*(n,\gamma,\xi):= \inf \bigl\{ \rho > 0 : \tilde{\mathcal{R}}(n,\xi,\rho) \leq \gamma \bigr\}.
\]
Then, for any $n \in \mathbb{N}, \xi \in \Xi$, $\alpha,\beta \in (0,1)$ with $\alpha + \beta < 1$, and $\rho \in(0,\infty)$, we have that $\tilde{\mathcal{R}}(n, \xi, \rho) \leq \mathcal{R}(n,\alpha,\xi,\rho)$, and therefore also that $\tilde{\rho}^*(n, \alpha+\beta,\xi) \leq \rho^*(n,\alpha,\beta,\xi)$. When our upper and lower bounds match, in terms of the separation rates, the problems of independence testing with simple and composite nulls are equivalent, and we have the same rates of convergence if we control the sum of error probabilities or if we control the error probabilities separately.

We are now in a position to state our main, general lower bound lemma.  Recall that a Rademacher random variable $\xi$ takes values $1$ and $-1$, each with probability $1/2$.
\begin{lemma}
\label{Lemma:LowerBound}
Suppose that $\mu_X$ and $\mu_Y$ are probability measures and that there exist $j_0\in \mathcal{J}$ and $k_0 \in \mathcal{K}$ such that $p_{j_0}^X(x)=p_{k_0}^Y(y)=1$ for all $x \in \mathcal{X}$ and $y \in \mathcal{Y}$. Let $(a_{jk})_{j \in \mathcal{J} \setminus \{j_0\}, k \in \mathcal{K} \setminus \{k_0\}}$ be a deterministic square-summable array of real numbers, let $(\xi_{jk})_{j \in \mathcal{J} \setminus \{j_0\}, k \in \mathcal{K} \setminus \{k_0\}}$ be an independent and identically distributed array of Rademacher random variables, and define a random element of $L^2(\mu)$ by
\[
	p := p_{j_0k_0} + \sum_{j \in \mathcal{J} \setminus \{j_0\}, k \in \mathcal{K} \setminus \{k_0\}} a_{jk} \xi_{jk} p_{jk}.
\]
Assume $\{p \in \mathcal{F}\}$ is an event, and define $f$ to be a random element of $\mathcal{F}$ that has the same distribution as $p | \{p \in \mathcal{F}\}$. Writing $\mathbb{E} \mathbb{P}_f^{\otimes n}$ for the resulting mixture distribution on $(\mathcal{X} \times \mathcal{Y})^n$ and $\mathbb{P}_{p_{j_0k_0}}$ for the distribution on $\mathcal{X} \times \mathcal{Y}$ with density $p_{j_0k_0}$, we have that
\[
	d_{\mathrm{TV}}^2 \bigl( \mathbb{P}_{p_{j_0 k_0}}^{\otimes n}, \mathbb{E} \mathbb{P}_f^{\otimes n} \bigr) \leq \frac{\exp\bigl( \frac{(n+1)^2}{2} \sum_{j \in \mathcal{J} \setminus \{j_0\}, k \in \mathcal{K} \setminus \{k_0\}} a_{jk}^4\bigr)}{4\mathbb{P}(p \in \mathcal{F})^2} - \frac{1}{4}.
\]
\end{lemma}
Suppose that the $f$ defined in Lemma~\ref{Lemma:LowerBound} takes values in $\mathcal{F}_{\xi}(\rho)$ with probability one. Then we have that
\[
	\tilde{\mathcal{R}}(n,\xi,\rho) \geq \inf_{\psi \in \Psi(1)} \Bigl\{ \mathbb{E}_{p_{j_0 k_0}}(\psi) + \mathbb{E}\mathbb{P}_f (1-\psi) \Bigr\} \geq 1 - d_\mathrm{TV} \bigl( \mathbb{P}_{p_{j_0 k_0}}^{\otimes n}, \mathbb{E} \mathbb{P}_f^{\otimes n} \bigr),
\]
which reduces the problem of finding lower bounds for the minimax risk $\tilde{\mathcal{R}}(n,\xi,\rho)$ to the choice of an appropriate separation $\rho$ and prior distribution over $\mathcal{F}_\xi(\rho)$.

The main challenge in applying Lemma~\ref{Lemma:LowerBound} is in finding a suitable upper bound for $\mathbb{P}(p \not\in \mathcal{F})$.  Provided $\bar{p}:= \sup_{j \in \mathcal{J},k \in \mathcal{K}} \|p_{jk}\|_\infty < \infty$, we can ensure that $\mathbb{P}(p \not\in \mathcal{F})=0$ by simply imposing the constraint that $\sum_{j \in \mathcal{J} \setminus \{j_0\}, k \in \mathcal{K} \setminus \{k_0\}} |a_{jk}| \leq 1/\bar{p}$. If we do this then we can prove the lower bound in Theorem~\ref{Thm:LowerBound1} below. 

\begin{thm}
\label{Thm:LowerBound1}
Suppose that $\mu_X$ and $\mu_Y$ are probability measures and that there exist $j_0\in \mathcal{J}$ and $k_0 \in \mathcal{K}$ such that $p_{j_0}^X(x)=p_{k_0}^Y(y)=1$ for all $x \in \mathcal{X}$ and $y \in \mathcal{Y}$.  Assume that $\bar{p} < \infty$, and fix $\gamma \in (0,1)$ and $\xi=(\theta,r,A) \in \Xi$ such that \textbf{(A1)} holds. Then there exists $c=c(\gamma,A) \in(0,\infty)$ such that
\begin{align*}
	\tilde{\rho}^*(n, \gamma, \xi) &\geq c \sup_{m \in \mathbb{N}} \min \biggl( \frac{r}{\theta_{\omega(m)}}, \frac{m^{1/4}}{n^{1/2}}, \frac{1}{m^{1/2} \bar{p}} \biggr).
\end{align*}
\end{thm}
Thinking of $\gamma = \alpha + \beta$, this lower bound matches the upper bound in Theorem~\ref{Thm:UpperBound} in certain cases, up to terms depending only on $\alpha,\beta$ and $A$, as we now explain.  Suppose that $nr^2 \geq \underline{\theta}^2$, which means that $m_0(nr^2) \geq 2$, so we only rule out the case where the sample size is so small that the optimal truncation level is to include only one basis function.   Suppose further that $m_0(nr^2) \leq Cn^{2/3}/\bar{p}^{4/3}$ for some $C=C(\alpha,\beta,A)$, which amounts to asking that the optimal truncation level does not grow too fast, or equivalently, that our alternatives are not too rough.  Then
\begin{align}
  \label{Eq:SimpleLower}
	\sup_{m \in \mathbb{N}} \min \biggl( \frac{r}{\theta_{\omega(m)}}, \frac{m^{1/4}}{n^{1/4}}, \frac{1}{m^{1/2} \bar{p}} \biggr) &\geq \min \biggl( \frac{\{m_0(nr^2)-1\}^{1/4}}{n^{1/2}}, \frac{1}{\{m_0(nr^2) -1 \}^{1/2}\bar{p}} \biggr) \nonumber \\
	& \geq \frac{m_0(nr^2)^{1/4}}{n^{1/2}} \min \bigl( 2^{-1/4}, C^{-3/4} \bigr).
\end{align}
A comparison of Corollary~\ref{Cor:Main} and~\eqref{Eq:SimpleLower} allows us to conclude that our $U$-statistic permutation test attains the minimax optimal separation rate in wide generality (i.e.~with few restrictions on the underlying spaces and the sequence $\theta$), provided that $nr^2$ is sufficiently large and $m_0(nr^2) \leq Cn^{2/3}/\bar{p}^{4/3}$.  The following example illustrates this latter condition in a specific case.
\begin{example}
  \label{Ex:Sobolev2}
Write $\zeta=(s_X,s_Y,d_X,d_Y,\alpha,\beta,A)$. In our $d$-dimensional Sobolev setting of Example~\ref{Ex:Sobolev}, when $t \geq 1$, we have $m_0(t) \asymp_\zeta t^{2d/(4s+d)}$ and hence when $n^{2s-d} \gtrsim_\zeta r^{3d}$ we have that $m_0(nr^2) \lesssim_\zeta n^{2/3}$.  Since we may take $\bar{p} = 2^{1/2}$, it therefore follows that when $nr^2 \geq 1$ and $n^{2s-d} \gtrsim_\zeta r^{3d}$, the lower bound~\eqref{Eq:SimpleLower} holds, and this matches the upper bound from Corollary~\ref{Cor:Main}.
\end{example}
Despite the attractive conclusions that can be drawn from Theorem~\ref{Thm:LowerBound1}, it remains desirable to weaken further the smoothness requirements on our alternatives.  It turns out that in certain settings, we can use empirical process techniques to lower bound the $\mathbb{P}(p \in \mathcal{F})$ term in Lemma~\ref{Lemma:LowerBound} without a bound on $\sum_{j \in \mathcal{J} \setminus \{j_0\}, k \in \mathcal{K} \setminus \{k_0\}} |a_{jk}|$.  This allows us to substantially widen the range of smoothnesses under which our upper and lower bounds match. We first illustrate this approach in our Sobolev example.
\begin{prop}
\label{Thm:FourierLowerBound}
In the context of Example~\ref{Ex:Sobolev}, fix $\gamma \in (0,1)$.  Then there exist $c_1, c_2 \in (0,\infty)$, each depending only on $d_X,d_Y,\gamma,s_X,s_Y$ and $A$, such that if $nr^2 \geq 2$ and $(r^{d}/n^{2s})^{1/(4s+d)} \leq c_1/ \log^{1/2}(nr^2)$, then
\[
	\tilde{\rho}^*(n, \gamma, \xi) \geq c_2 \biggl( \frac{r^{d}}{n^{2s}} \biggr)^{1/(4s+d)}.
\]
\end{prop}
Thus, the lower bound of Proposition~\ref{Thm:FourierLowerBound} matches the upper bound of Example~\ref{Ex:Sobolev} when $(r^{d}/n^{2s})^{1/(4s+d)} \leq c_1/ \log^{1/2}(nr^2)$, or equivalently when $m_0(nr^2) \lesssim_{\zeta} n^2/\log^2(nr^2)$.  This condition is rather weak, and holds whenever the minimax separation rate is polynomially decreasing in $r^d/n^{2s}$.  Compared with Example~\ref{Ex:Sobolev2}, Proposition~\ref{Thm:FourierLowerBound} extends the parameter regime over which the lower bound on the minimax separation rate for independence testing matches the upper bound of Example~\ref{Ex:Sobolev}, by also covering lower smoothness cases where $n^{2s-d} \ll r^{3d}$.

We remark that Proposition~\ref{Thm:FourierLowerBound} generalises to more abstract settings.  Assume that $\mathcal{X}$ and $\mathcal{Y}$ are equipped with metrics $\tau_{\mathcal{X}}$ and $\tau_{\mathcal{Y}}$ respectively, and write $H(\cdot,\mathcal{X})$ and $H(\cdot,\mathcal{Y})$ for the corresponding metric entropies.  Suppose that there exist $\kappa_1,\kappa_2 \geq 0$ and functions $\ell_1,\ell_2:(0,\infty) \rightarrow (0,\infty)$ that are slowly varying at infinity such that $H(u,\mathcal{X}) = u^{-2 \kappa_1} \ell_1(1/u)$ and $H(u,\mathcal{Y}) = u^{-2 \kappa_2} \ell_2(1/u)$; thus, if $\mathcal{X} = [0,1]^{d_X}$, then we may take $\kappa_1 = 0$ and $\ell_1(u) = d_X \log u$.  Suppose further that there exist $\alpha_1,\alpha_2,\beta_1,\beta_2 >0$ such that
\[
	|p_{jk}(x,y) - p_{jk}(x',y')| \lesssim_\zeta \|j\|_1^{\alpha_1}\tau_{\mathcal{X}}(x,x')^{\beta_1} + \|k\|_1^{\alpha_2}\tau_{\mathcal{Y}}(y,y')^{\beta_2}
\]
for all $x,x' \in \mathcal{X}, y,y' \in \mathcal{Y}, j \in \mathcal{J}, k \in \mathcal{K}$, where $\zeta$ does not depend on $n,r,x,x',y,y',j,k$.  In our Sobolev example, then, we may take $\alpha_1 = \alpha_2 = \beta_1 = \beta_2 = 1$. 
Finally assume that $\bar{p} < \infty$.  Then, taking $\xi = (\theta,r,A) \in \Xi$ and $\gamma \in (0,1)$, writing $\gamma_1 :=\frac{\kappa_1}{\beta_1 ((s_X/\alpha_1) \wedge 1)}$ and $\gamma_2:=\frac{\kappa_2}{\beta_2 ((s_Y/\alpha_2) \wedge 1)}$,  and setting $s = d(d_X/s_X + d_Y/s_Y)^{-1}$, similar calculations to those in the proof of Proposition~\ref{Thm:FourierLowerBound} reveal that
\[
	\tilde{\rho}^*(n, \gamma, \xi) \gtrsim_\zeta \biggl( \frac{r^d}{n^{2s}} \biggr)^{1/(4s+d)}
\]
whenever $\max(\gamma_1,\gamma_2)<1$ and $r \lesssim_{\zeta,\epsilon} \min \bigl( n^{\frac{2s(1-\gamma_1)}{d+4s\gamma_1}- \epsilon}, n^{\frac{2s(1-\gamma_2)}{d+4s\gamma_2} - \epsilon} \bigr)$ for some $\epsilon>0$.  Thus, we match the upper bound of Corollary~\ref{Cor:Sobolev} even in this more general setting.

Our final lower bound applies similar empirical process techniques to show that the rate found by applying the first part of Corollary~\ref{Cor:InfDim} to Example~\ref{Ex:InfDim} for our infinite-dimensional example is optimal in certain regimes of $(n,r)$.
\begin{prop}[BKS(2020)]
\label{Prop:InfDimLowerBound}
Let $\mathcal{X},\mathcal{Y},\mu_X,\mu_Y,(p_{jk}),\mathcal{J}$ and $\mathcal{K}$ be as in Corollary~\ref{Cor:InfDim} and Example~\ref{Ex:InfDim}. Fix $\alpha,\beta \in (0,1)$ such that $\alpha+\beta<1$ and $r,s_X,s_Y,A >0$.  For $j\in \mathcal{J}, k \in \mathcal{K}$ let $\theta_{jk}=\exp( s_X |j|^{1/2}) \vee \exp (s_Y |k|^{1/2})$, and let $\xi=(\theta,r,A) \in \Xi$.  Recalling the definitions of $s$ and $c_0$ from Corollary~\ref{Cor:InfDim}, suppose that $r^2 \leq n^{s/(s+c_0) - \epsilon}$ for some $\epsilon>0$. Then there exist $C=C(\alpha,\beta,s_X,s_Y,A,\epsilon)>0$ and $C' = C'(\alpha,\beta,s_X,s_Y,A,\epsilon)>0$ such that when $\min(n,nr^2) \geq C'$ we have
\[
	\rho^*(n, \alpha, \beta, \xi) \geq \frac{C m_{0,s_X,s_Y}^{1/4}(nr^2)}{n^{1/2}}.
\]
\end{prop}

\section{Power function}
\label{Sec:PowerFunction}

In this section we provide an approximation to the power function of our USP test from Section~\ref{Sec:MainResults}. For simplicity of exposition we will restrict attention to the case where the $\mathcal{X}=\mathcal{Y}=[0,1]$, and work with the Fourier basis~\eqref{Eq:Fourier} with respect to the respective Lebesgue base measures $\mu_X$ and $\mu_Y$.  Recall that in this case, $\mathcal{J}=\mathcal{K} = \bigl( \{0,1\} \times \mathbb{N}_0 \bigl) \setminus \{(1,0)\}$.
We will consider test statistics $\hat{D}_n$ with
\[
	\mathcal{M} = \bigl( \{0,1\} \times [M] \bigr) \times \bigl( \{0,1\} \times [M] \bigr)
\]
for a tuning parameter $M \in \mathbb{N}$ which will typically be large so that $\hat{D}_n$ is approximately normally distributed.  When $M$ is large and the dependence between $X$ and $Y$ is weak, we will see that the variance of $\hat{D}_n$ can be approximately expressed in terms of
\begin{align*}
  \sigma_{M,X}^2 \equiv \sigma_{M,X}^2(f) &:= 2M+1 +  \sum_{m=1}^{2M} (2M+1-m)\bigl\{a_{(0,m)\bullet}(f)^2 + a_{(1,m)\bullet}(f)^2\bigr\} \\
  &\phantom{:}\asymp M \|f_X\|_{L^2(\mu_X)}^2
\end{align*}
as $M \rightarrow \infty$, and the corresponding quantity $\sigma_{M,Y}^2$, in which $f_X$ and $\mu_X$ above are replaced with $f_Y$ and $\mu_Y$ respectively and $a_{j\bullet}(f)$ for $j \in \mathcal{J}$ is replaced with $a_{\bullet k}(f)$ for $k \in \mathcal{K}$.

Define $A_{M,X} \equiv A_{M,X}(f) := 1 + \sum_{m=1}^{2M} (|a_{(0,m)\bullet}(f)| + |a_{(1,m)\bullet}(f)|)$, with the corresponding definition of $A_{M,Y}$.  We will see that the quantities $A_{M,X}$ and $A_{M,Y}$, which when $f \in \mathcal{F}$ are both $o(M^{1/2})$ as $M \rightarrow \infty$ by Lemma~\ref{Prop:SquareSummableSequences} in the supplement, will play a role in controlling the normal approximation error of our test statistic and the corresponding null statistics.  

\begin{thm}[BKS(2020)]
\label{Thm:PowerFunction}
In the above setting, let $f \in \mathcal{F}$ with $\|f\|_\infty<\infty$, let $\alpha \in (0,1)$ and let $B \in \mathbb{N}$.  Write
\[
	\Delta_f:= \frac{\binom{n}{2}^{1/2} \sum_{(j,k) \in \mathcal{M}} \bigl\{a_{jk}(f)-a_{j\bullet}(f)a_{\bullet k}(f)\bigr\}^2}{ \sigma_{M,X}\sigma_{M,Y}}
\]
and, with $s=\lceil \alpha(B+1) \rceil -1$, let $\mathrm{B}_{B-s,s+1} \sim \mathrm{Beta}(B-s,s+1)$.  Let
\[
  \delta_* := \max\biggl\{ \frac{\Delta_f^{1/2}}{M^{1/2}}, \frac{1}{M^{1/2}}, D(f)^{1/4}, \Bigl( \frac{M^2}{n} \Bigr)^{1/2}, \frac{A_{M,X} A_{M,Y}}{M} \biggr\}^{1/3}.
  \]
Then there exists $C=C(\|f\|_\infty,\alpha)>0$ such that the p-value $P$ in~\eqref{Eq:pvalue} satisfies
\begin{align*}
	\bigl| \mathbb{P}_f(P \leq \alpha) - \mathbb{E} \bar{\Phi} \bigl( \Phi^{-1}( \mathrm{B}_{B-s,s+1}) &- \Delta_f \bigr) \bigr| \leq C \min\bigl\{B^{4/3}\delta_*,(B^{-1/3} \vee \delta_*^{1/3})\bigr\}.
\end{align*}
\end{thm}
To understand the implications of this theorem, first consider the case where the null hypothesis holds, so that $\Delta_f = 0$, and further assume for simplicity that $\alpha(B+1)$ is an integer.  Then the conclusion states that
\[
  |\mathbb{P}_f(P \leq \alpha) - \alpha| \leq C \min\bigl\{B^{4/3}\delta_*,(B^{-1/3} \vee \delta_*^{1/3})\bigr\},
  \]
  though in fact, we already know that $\mathbb{P}_f(P \leq \alpha) = \alpha$ in this special case.  More generally, Theorem~\ref{Thm:PowerFunction} provides an approximation to the local power of our test when $D(f)$ is small and both $n$ and $M$ are large, with $M^2/n$ small.  It could be used by practitioners to guide the choice of $B$ in cases where computation is expensive: given an anticipated effect size~$\Delta_f$, one can compare $\mathbb{E} \bar{\Phi} \bigl( \Phi^{-1}( \mathrm{B}_{B-s,s+1}) - \Delta_f \bigr)$ to $\bar{\Phi} \bigl( \Phi^{-1}(1 - \alpha) - \Delta_f \bigr)$ to understand the trade-off between computation and power.  Note also that $\bar{\Phi} \bigl( \Phi^{-1}(1 - \alpha) - \Delta_f \bigr)$ is the limiting power of the oracle test that has access to the marginal distributions.

To illustrate Theorem~\ref{Thm:PowerFunction}, we conducted some simulations to verify the accuracy of the approximate power function. For a parameter $\rho \in [0,1/2]$, we considered independent copies of pairs $(X,Y)$ with density function 
\begin{equation}
  \label{Eq:frho}
	f_\rho(x,y) = 1 + 2 \rho \sin(2 \pi x) \sin(2 \pi y)	
\end{equation}
for $x,y \in [0,1]$, so that, marginally, $X,Y \sim U[0,1]$. For these densities we have $D(f_\rho)=\rho^2$ and $\sigma_{M,X}^2 = \sigma_{M,Y}^2 = 2M+1$. In our simulations we take $n=300, M=7, B=99$ and $\alpha=0.1$ so that 
\[
	\Delta_{f_\rho} = \binom{300}{2}^{1/2} \rho^2 / 15 = \rho^2 \times 14.1 \ldots.
\]
Figure~\ref{Fig:Simulation} plots the theoretical approximate power function, given by $\mathbb{E} \bar{\Phi} \bigl( \Phi^{-1}( \mathrm{B}_{B-s,s+1}) - \Delta_f \bigr)$, and the empirical power function, which was computed by averaging over 700 independent repetitions of the experiment for each value of $\rho$.  The simulations reveal a good agreement between our approximations and empirical performance.

\begin{figure}
\centering
\includegraphics[width=10cm]{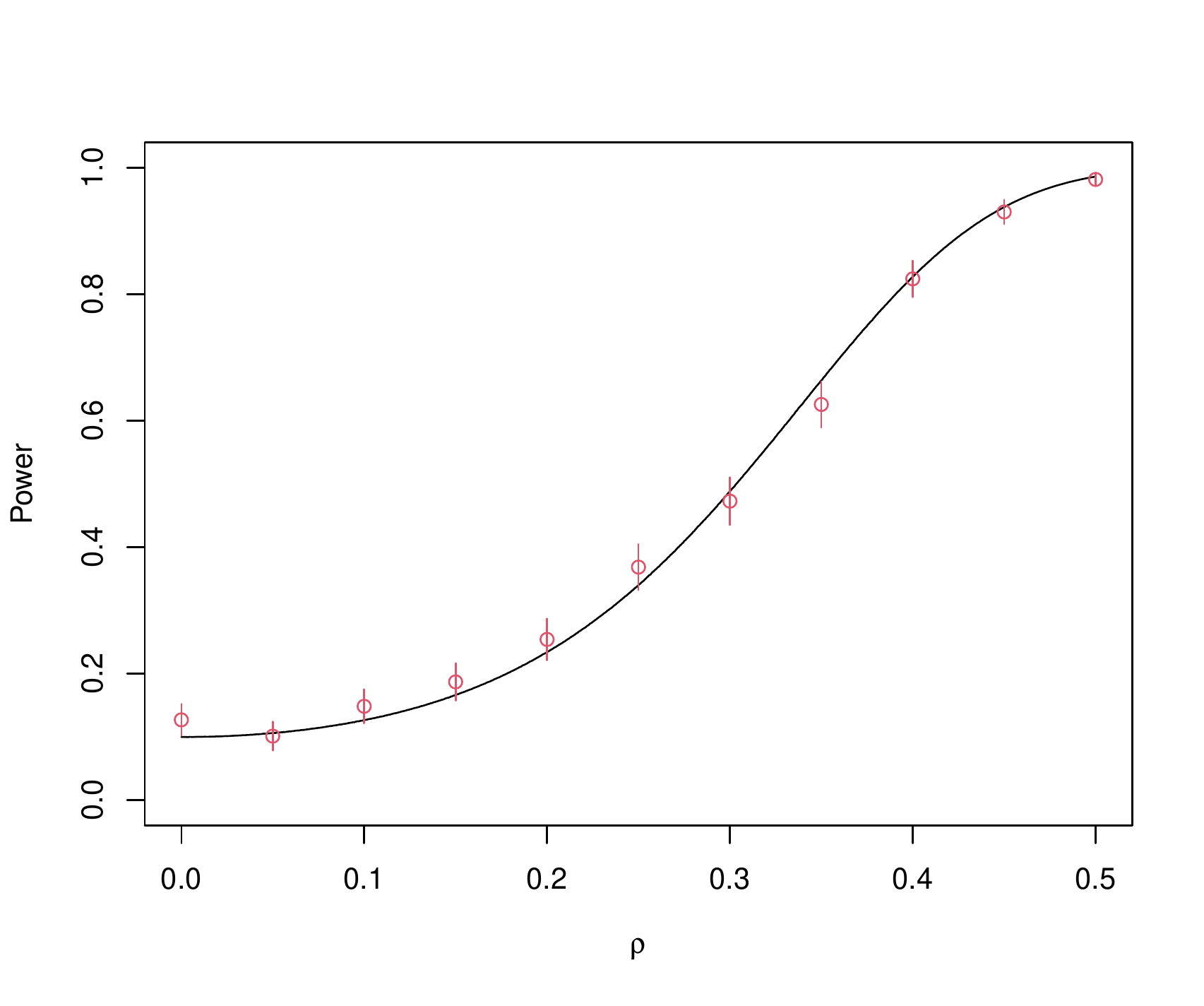}
\caption{\label{Fig:Simulation} The theoretical approximate power function from Theorem~\ref{Thm:PowerFunction} (black), and an empirical estimate of the true power function (red); error bars show two standard deviations. Here, the data were generated according to~\eqref{Eq:frho} with $n=300,B=99,\alpha=0.1,M=7$.}
\end{figure}

The proof of Theorem~\ref{Thm:PowerFunction} uses careful bounds for the error in normal approximations to degenerate $U$-statistics, as well as corresponding bounds in the case where the $U$-statistic is computed on a permuted data set.  In the unpermuted case, such bounds have been well studied, inspired by the work of \citet{Hall84} and \citet{deJong1990}, who  established asymptotic normality results for degenerate $U$-statistics.  This is interesting because, in the classical theory, the asymptotic distribution of a degenerate $U$-statistic of order 2, for a fixed $h$, is given by a weighted infinite sum of independent chi-squared random variables \citep[e.g.][p.~194]{Serfling1980}.  Indeed, from the form of the first term on the right-hand side of~\eqref{Eq:dWbound} below, it is not clear that a normal approximation error will be small.  However, if we allow $h$ to depend on the sample size $n$, then the weights in the infinite sum may become more diffuse, so that a normal approximation may be more appropriate.  In our setting, the truncation set $\mathcal{M}$ will typically depend on $n$, in which case we are in a situation where the $U$-statistic kernel depends on the sample size. 
\citet{RinottRotar1997} derived error bounds in the normal approximation with respect to classes of probability integral metrics that include the Kolmogorov distance. \citet{DoblerPeccati2017,DoblerPeccati2019} extended these results in two directions, first by working with multivariate $U$-statistics, and second by controlling the normal approximation error in the $L_1$-Wasserstein distance.  We present a consequence of \citet[][Theorem~3.3]{DoblerPeccati2019} below, because it it will help to contextualise our (new) error bound in the permuted case, which appears as Proposition~\ref{Lemma:PermutedNormality}.
\begin{prop}[\citet{DoblerPeccati2019}, Theorem~3.3]
\label{Lemma:Normality}
For $n \geq 2$, let $Z_1,\ldots,Z_n$ be independent and identically distributed random elements in a measurable space $\mathcal{Z}$, and let $h: \mathcal{Z} \times \mathcal{Z} \rightarrow \mathbb{R}$ be a symmetric measurable function that satisfies $\mathbb{E} h(z,Z_1) = 0$ for all $z \in \mathcal{Z}$ and $\mathbb{E}\{ h(Z_1,Z_2)^2\} =1$. Write $g(x,y) := \mathbb{E}\{h(x,Z_1) h(y,Z_1)\}$ and $U:=\frac{1}{2}\binom{n}{2}^{-1/2} \sum_{i \in \mathcal{I}_2} h(Z_{i_1},Z_{i_2})$.  With $W \sim N(0,1)$, there exists a universal constant $C>0$ such that for $n \geq 2$ we have
\begin{equation}
  \label{Eq:dWbound}
	d_\mathrm{W}(U,W) \leq C \max \biggl[ \frac{\mathbb{E}^{1/2}\{ h^4(Z_1,Z_2)\}}{n^{1/2}},  \mathbb{E}^{1/2}\{ g^2(Z_1,Z_2) \} \biggr].
\end{equation}
\end{prop}

As mentioned above, Proposition~\ref{Lemma:PermutedNormality} below extends Proposition~\ref{Lemma:Normality} to the case of a permuted data set, and therefore provides a useful stepping stone for analysing the power properties of permutation tests based on degenerate $U$-statistics.


\begin{prop}[BKS(2020)]
\label{Lemma:PermutedNormality}
For $n \geq 4$, let $(X_1,Y_1),\ldots,(X_n,Y_n)$ be independent and identically distributed random elements in a product space $\mathcal{Z}=\mathcal{X} \times \mathcal{Y}$ and let $\Pi$ be a uniformly random element of $\mathcal{S}_n$, independent of $(X_i,Y_i)_{i=1}^n$.  Let $h: \mathcal{Z} \times \mathcal{Z} \rightarrow \mathbb{R}$ be a symmetric measurable function that satisfies
\[
	\mathbb{E} h\bigl( (x,y),(x',Y_1) \bigr) = \mathbb{E} h \bigl( (x,y) , (X_1,y') \bigr) = 0 
\]
for all $x,x' \in \mathcal{X}$ and $y,y' \in \mathcal{Y}$, and also satisfies $\mathbb{E} h^2\bigl((X_1,Y_2),(X_3,Y_4)\bigr) = 1$. Write $g((x,y),(x',y')) := \mathbb{E}\{ h((x,y),(X_1,Y_2)) h((x',y'),(X_1,Y_2)) \}$ and
\[
	U:=\frac{1}{2}\binom{n}{2}^{-1/2} \sum_{(i_1,i_2) \in \mathcal{I}_2} h\bigl((X_{i_1},Y_{\Pi(i_1)}),(X_{i_2},Y_{\Pi(i_2)}) \bigr).
\]
Then, with $W \sim N(0,1)$, there exists a universal constant $C>0$ such that
\begin{align}
\label{Eq:PermutedNormalityBound}
	d_\mathrm{W}(U,W) \leq C \max &\biggl[ \frac{1}{n^{1/2}} \max_{\sigma \in \mathcal{S}_4} \mathbb{E}^{1/2}\bigl\{ h^4 \bigl((X_1,Y_{\sigma(1)}),(X_2,Y_{\sigma(2)})\bigr)\bigr\}, \nonumber \\
	&\mathbb{E}^{1/2}\bigl\{ g^2 \bigl((X_1,Y_2),(X_3,Y_4)\bigr) \bigr\}, \mathbb{E} \bigl| \mathbb{E} \bigl\{ h \bigl( (X_1,Y_2), (X_3,Y_1) \bigr) | X_3,Y_2 \bigr\} \bigr|  \biggr]. 
\end{align}
\end{prop}
Comparing the bounds in Propositions~\ref{Lemma:Normality} and~\ref{Lemma:PermutedNormality}, we see three differences caused by the permutation.  The first term in~\eqref{Eq:PermutedNormalityBound} is slightly inflated by the maximum over the 24 permutations in $\mathcal{S}_4$; the second term involves distinct indices, which is to be expected since most permutations of $\mathcal{S}_n$ have only a small number of fixed points; and finally, there is an additional third term, which vanishes if $X_1$ and $Y_1$ are independent. 

In fact, for a full description of the power properties of our permutation test, we require a multivariate normal approximation error bound for the random vector consisting of the original test statistic and the $B$ test statistics computed on the permuted data sets.  Since this statement is more complicated, we defer it to the online supplement~(Lemma~\ref{Lemma:JointNormality}).  Its main message for our purposes, however, is that these $B+1$ statistics are approximately independent, which is what facilitates the power function approximation in Theorem~\ref{Thm:PowerFunction}.

\section{Numerical results}
\label{Sec:Simulations}

In this section, we examine the empirical performance of our USP test, comparing it with alternative approaches where appropriate.  We consider discrete, absolutely continuous and infinite-dimensional settings, following the main examples given earlier.  First, however, we show how our test statistic can be computed much more efficiently than might initially appear to be the case.

\subsection{Computational trick}
\label{Sec:CompTrick}

Our test statistic $\hat{D}_n$ can be rewritten similarly to the test statistics in \citet{Song2012} to allow for quicker computation, in the case that $\mathcal{M}=\mathcal{J}_0 \times \mathcal{K}_0$ for some $\mathcal{J}_0 \subseteq \mathcal{J}$ and $\mathcal{K}_0 \subseteq \mathcal{K}$. Define matrices $J = (J_{i_1 i_2})_{i_1,i_2=1}^n, K = (K_{i_1 i_2})_{i_1,i_2=1}^n$ by
\[
	J_{i_1 i_2} := \sum_{j \in \mathcal{J}_0} p_j^X(X_{i_1}) p_j^X(X_{i_2}) \quad \text{and} \quad K_{i_1 i_2} := \sum_{k \in \mathcal{K}_0} p_k^Y(Y_{i_1}) p_k^Y(Y_{i_2}),
\]
and let $\tilde{J}$ and $\tilde{K}$ be the corresponding matrices with the diagonal entries set to zero. Then, writing $\mathbf{1} \in \mathbb{R}^n$ for the all-ones vector, we have that
\begin{align*}
	\hat{D}_n &= \frac{1}{n(n-1)} \sum_{(i_1,i_2) \in \mathcal{I}_2} J_{i_1 i_2}K_{i_1 i_2} - \frac{2}{n(n-1)(n-2)} \sum_{(i_1,i_2,i_3) \in \mathcal{I}_3} J_{i_1 i_2} K_{i_1 i_3} \\
	& \hspace{150pt} + \frac{1}{n(n-1)(n-2)(n-3)} \sum_{(i_1,i_2,i_3,i_4) \in \mathcal{I}_4} J_{i_1 i_3} K_{i_2 i_4} \\
	& = \frac{1}{n(n\!-\!1)} \sum_{i_1,i_2=1}^n \tilde{J}_{i_1 i_2} \tilde{K}_{i_1 i_2} \!-\! \frac{2}{n(n\!-\!1)(n\!-\!2)} \biggl( \sum_{i_1,i_2,i_3=1}^n \tilde{J}_{i_1i_2} \tilde{K}_{i_1 i_3} - \sum_{i_1,i_2=1}^n \tilde{J}_{i_1 i_2} \tilde{K}_{i_1 i_2} \biggr) \\
	&\hspace{1pt}+ \frac{1}{n(n\!-\!1)(n\!-\!2)(n\!-\!3)} \biggl( \sum_{i_1,i_2,i_3,i_4=1}^n \! \! \! \! \! \! \tilde{J}_{i_1 i_3} \tilde{K}_{i_2 i_4} - 4 \! \! \sum_{i_1,i_2,i_3=1}^n \! \! \! \! \tilde{J}_{i_1 i_2} \tilde{K}_{i_1 i_3} + 2 \! \! \sum_{i_1,i_2=1}^n \! \! \! \tilde{J}_{i_1 i_2} \tilde{K}_{i_1 i_2} \biggr) \\
	& = \biggl\{ \frac{1}{n(n-1)} + \frac{2}{n(n-1)(n-2)} + \frac{2}{n(n-1)(n-2)(n-3)} \biggr\} \mathrm{tr}( \tilde{J} \tilde{K}) \\
	&\hspace{5pt} - \biggl\{ \frac{2}{n(n-1)(n-2)} + \frac{4}{n(n-1)(n-2)(n-3)} \biggr\} \mathbf{1}^T \tilde{J} \tilde{K} \mathbf{1} + \frac{\mathbf{1}^T \tilde{J} \mathbf{1} \mathbf{1}^T \tilde{K} \mathbf{1}}{n(n-1)(n-2)(n-3)}  \\
	&= \frac{\mathrm{tr}( \tilde{J} \tilde{K})}{n(n-3)}  - \frac{2\mathbf{1}^T \tilde{J} \tilde{K} \mathbf{1}}{n(n-2)(n-3)} + \frac{\mathbf{1}^T \tilde{J} \mathbf{1} \mathbf{1}^T \tilde{K} \mathbf{1}}{n(n\!-\!1)(n\!-\!2)(n\!-\!3)}.
\end{align*}
From this final expression, we can see that $\hat{D}_n$ can be computed in $O\bigl(n^2(|\mathcal{J}_0| + |\mathcal{K}_0|)\bigr)$ operations, with the most time-consuming part being the computation of the matrices $\tilde{J}$ and~$\tilde{K}$.

\subsection{Discrete settings}

Here we study two different examples, to illustrate the effects of sparse and dense dependence.  The first is a $6 \times 6$ contingency table, so that $J = K = 6$, where the cell probabilities are of the form
\[
f(j,k) = \frac{2^{-(j+k)}}{(1-2^{-J})(1-2^{-K})} + \epsilon( \mathbbm{1}_{\{j=k=1\}} + \mathbbm{1}_{\{j=k=2\}}) - \epsilon( \mathbbm{1}_{\{j=1,k=2\}} + \mathbbm{1}_{\{j=2,k=1\}}),
\] 
for $j,k \in [6]$. Here, $\epsilon \geq 0$ measures the strength of the dependence; in fact, $D(f) = 4\epsilon^2$.  Our second example has $J = K = 8$ and cell probabilities of the form
\[
	f(j,k) = \frac{1}{JK} + (-1)^{j+k-1} \epsilon,
\]
for which $D(f)=JK \epsilon^2$.  Thus, the main difference between the examples is in the number of cells affected by the perturbation: in the first case, only the summands in $D(f)$ corresponding to $(j,k) \in \{1,2\} \times \{1,2\}$ are non-zero, whereas in the second example, all summands are non-zero.  

Figure~\ref{Fig:Discrete} plots estimates, computed as sample averages over 10000 repetitions, of the power of our USP test as a function of $\epsilon$ in the two examples, with $n=100$ in Figure~\ref{Fig:Discrete}(a) and $n=50$ in Figure~\ref{Fig:Discrete}(b).  In both cases, we set $\alpha = 0.05$ and $B = 99$.  For comparison, we also plot corresponding power estimates for two versions of Pearson's chi-squared test.  The first, corresponding to the more usual practice in applications, uses as a critical value for the test the $(1-\alpha)$th quantile of the chi-squared distribution with $(J-1)(K-1)$ degrees of freedom; the second computes the critical value using a permutation procedure similar to that employed for our USP test.  The advantage of the second approach is that it controls the Type I error at the nominal level.   In both cases, our USP test has greater power than both versions of Pearson's test, particularly in the first example, which is especially striking given that the chi-squared quantile version of Pearson's test is anti-conservative there.
\begin{figure}
(a) \hspace{6cm} (b) \\ 
\includegraphics[width=0.45\textwidth]{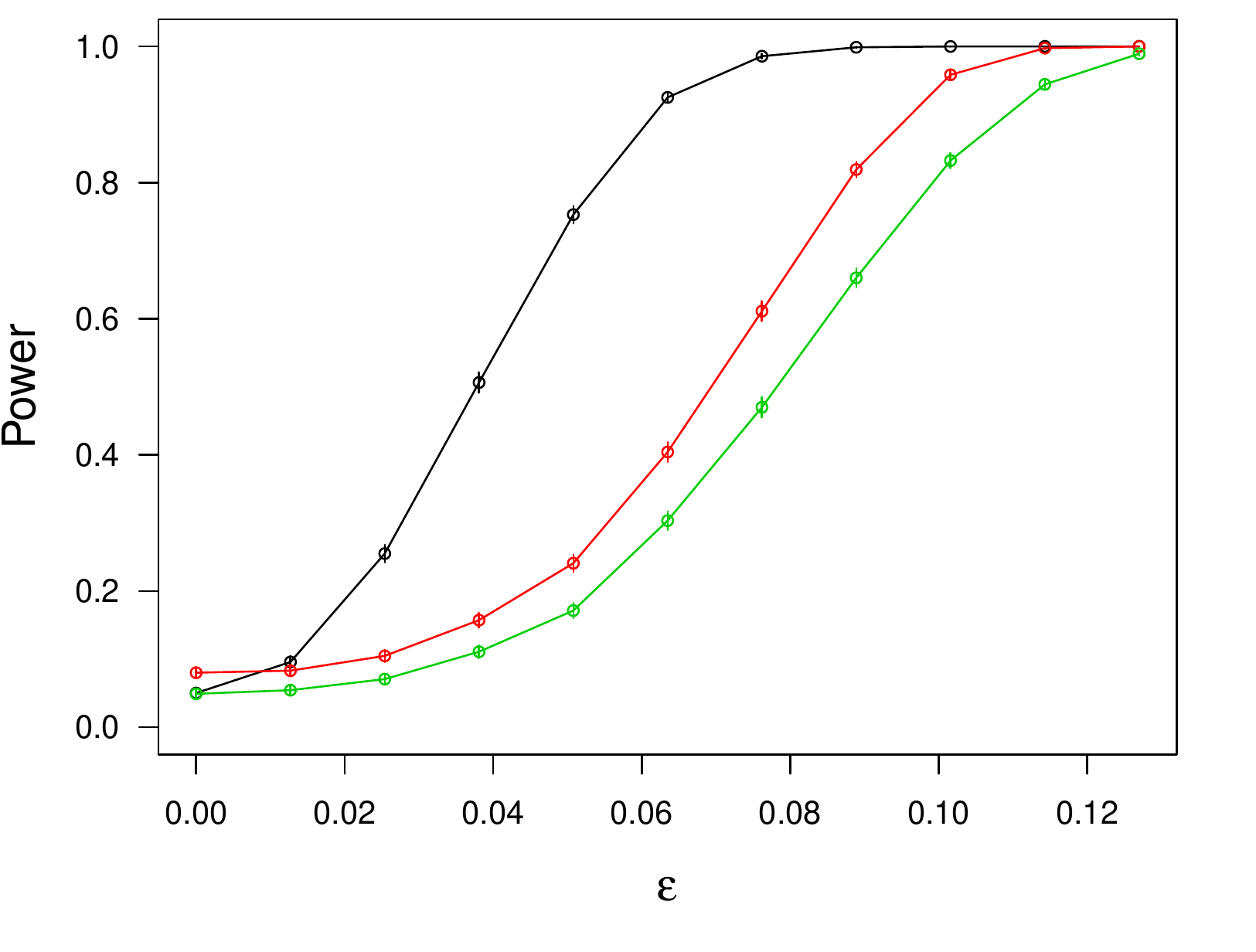}  \includegraphics[width=0.45\textwidth]{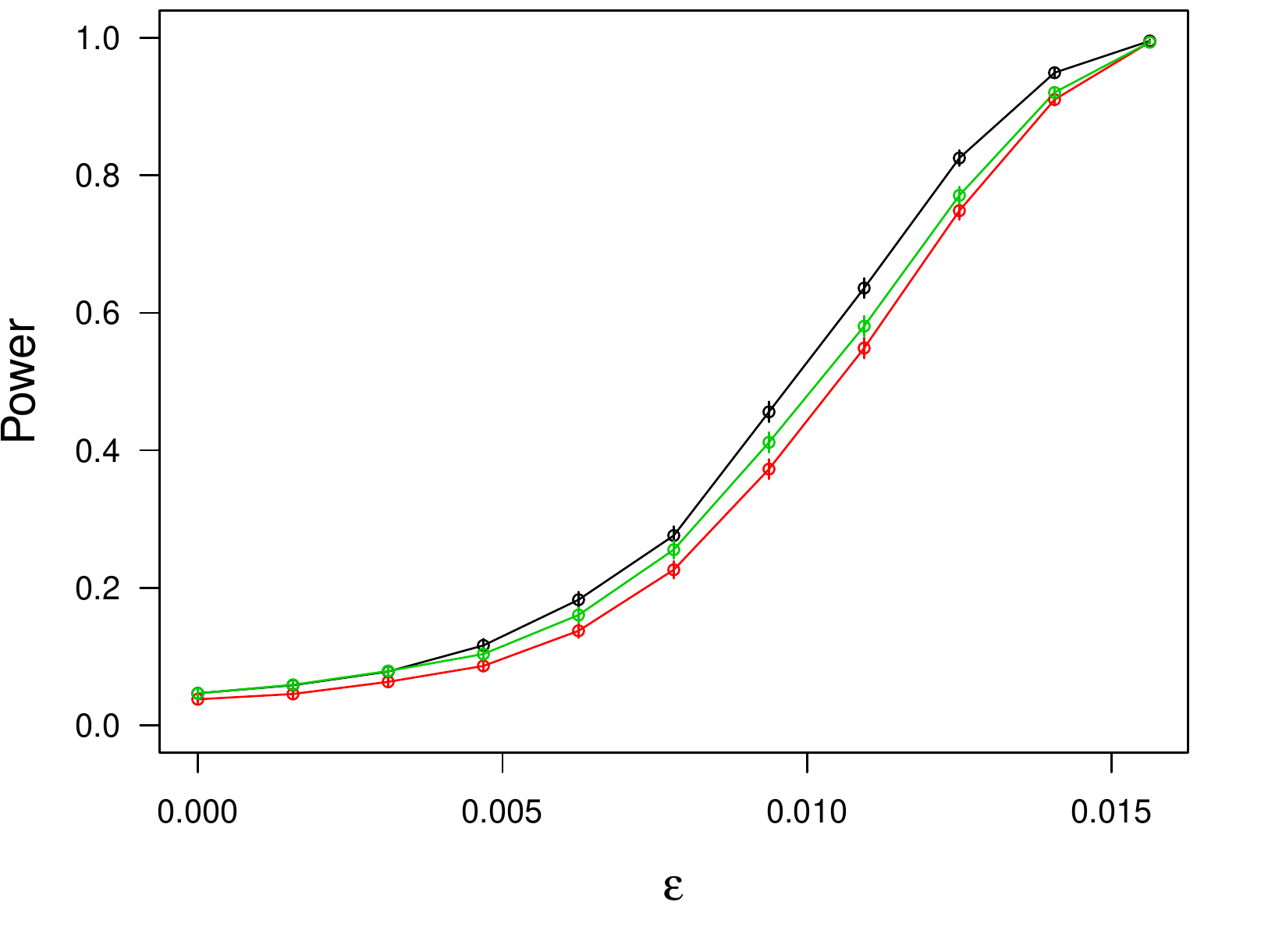}
\caption{\label{Fig:Discrete}Estimated power functions in the two discrete settings for our $U$-statistic permutation test (black), as well as Pearson's chi-squared test with chi-squared quantile (red) and quantile obtained from permutations (green).  Error bars show three standard errors; other parameters: $\alpha = 0.05$, $B = 99$, $n=100$ (left), $n=50$ (right).}
\end{figure}

\subsection{Sobolev example}

In this subsection, we consider a setting originally studied by \citet{Sejdinovic13}.  For $\omega \in \mathbb{N}$ and $(x,y) \in [0,1]^2$, define the density function
\[
	f_\omega(x,y) = 1+ \sin(2\pi \omega x)\sin(2\pi \omega y).
\]
\citet{BS2019} also consider this family of densities, and explain why it becomes increasingly difficult to detect the dependence as $\omega$ increases, despite the fact that the mutual information does not depend on $\omega$.  In fact, we also have $D(f_\omega) = 1/4$ for every $\omega \in \mathbb{N}$, so this measure of dependence does not depend on $\omega$ either.

In Figure~\ref{Fig:Sobolev}, we plot estimates of the power of our USP test, computed over 2000 repetitions with $n=100,200$. 
The choice of $\mathcal{M}$ is made as in Section~\ref{Sec:PowerFunction}, with $M=2,4$.  As alternative approaches, we also study the HSIC test of \citet{Gretton05}, which is implemented in the \texttt{R} package \texttt{dHSIC} \citep{PfisterPeters2017}, the MINTav test of \citet{BS2019}, implemented in the \text{R} package \texttt{IndepTest} \citep{BS2018} with $k \in [5]$, a test based on the empirical copula process described by \citet{Kojadinovic2009} and implemented in the \texttt{R} package \texttt{copula} \citep{HKMY2017} and a test based on distance covariance implemented in the \texttt{R} package \texttt{energy} \citep{RizzoSzekely2017}.  For these comparison methods, we used the default tuning parameter values recommended by the corresponding authors.  The fact that the departures in this example are aligned with a single basis function for each choice of $\omega$ means that the power of our USP test is constant for $\omega \leq M$, and it performs extremely well in these cases.  Once $\omega$ exceeds $M$, the test has no better than nominal power, as expected.  Thus, $M$ determines the number of directions of departure from independence that we can hope to detect with our USP test (we have $4M^2$ coefficients to estimate).  Increasing the value of $M$ would provide non-trivial power for larger values of $\omega$, but would sacrifice some power for smaller values of $\omega$.  
\begin{figure}
(a) \hspace{6cm} (b) \\ 
\includegraphics[width=0.45\textwidth]{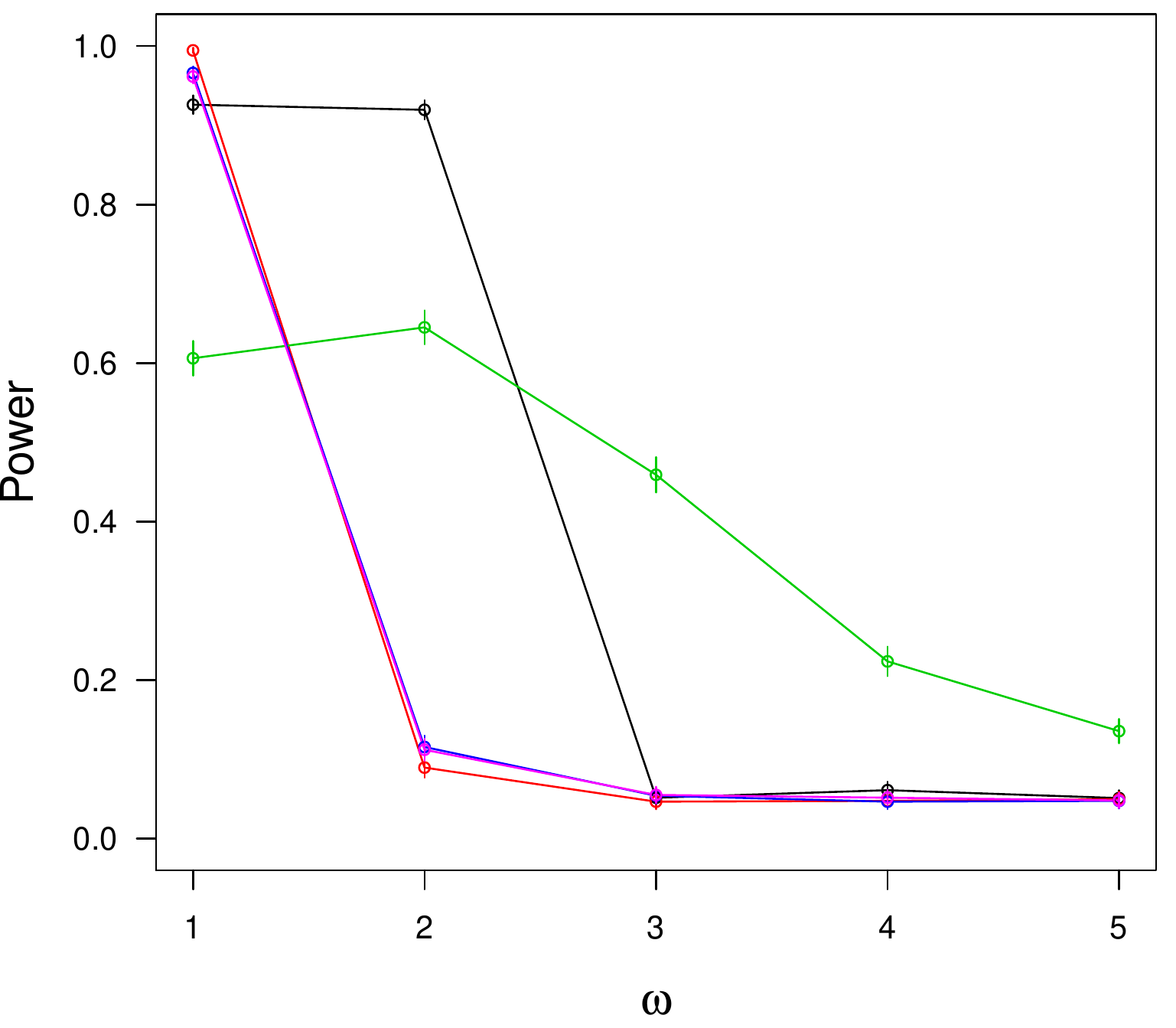}  \includegraphics[width=0.45\textwidth]{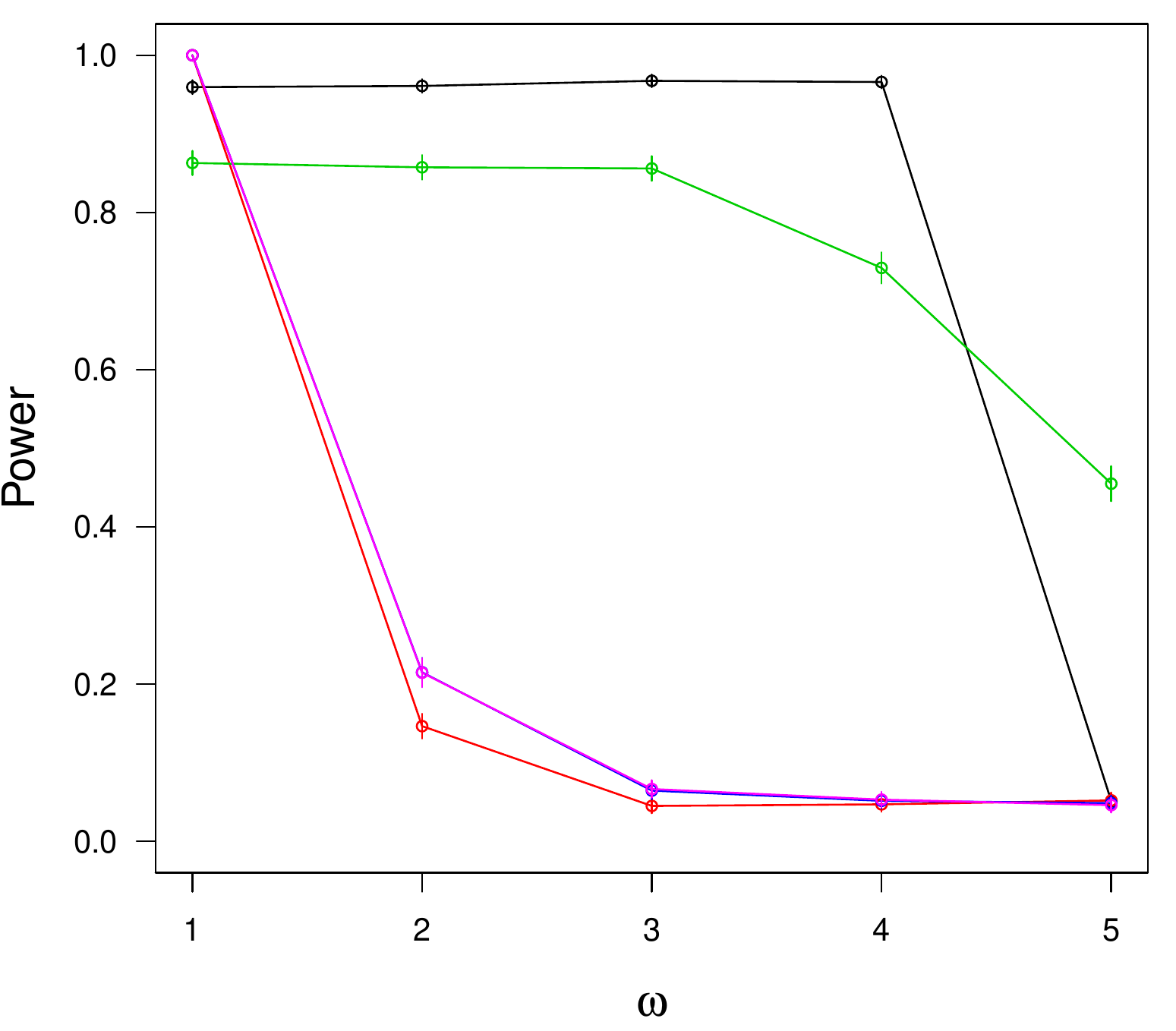}
\caption{\label{Fig:Sobolev}Estimated power functions in the Sobolev example for our $U$-statistic permutation test (black) with $M=2, n=100$ (left) and $M=4, n=200$ (right), HSIC (red), distance covariance (blue), copula (purple) and MINTav (green). Error bars show two standard errors; other parameters: $\alpha = 0.05$, $B = 99$.}
\end{figure}

\subsection{Infinite-dimensional example}

Our final example concerns potentially correlated Brownian motions on $[0,1]$, as an illustration of our USP test applied to functional data.  More precisely, our data come in the form of pairs $(X,Y)$, where $X = (X_t)_{t \in [0,1]}$ is a standard Brownian motion, and where, for some $r \in [0,1]$ and for another standard Brownian motion $Z = (Z_t)_{t \in [0,1]}$ that is independent of $X$, we have that $Y = (Y_t)_{t \in [0,1]}$ is given by
\[
Y_t = r X_t + (1-r^2)^{1/2} Z_t.
\]
Thus, marginally, $Y$ is also distributed as a standard Brownian motion.  

By the Wiener representation of Brownian motion \citep[e.g.,][]{Kahane1997}, we can write
\[
X_t = 2^{1/2} \sum_{\ell=1}^\infty \eta_\ell \frac{\sin \bigl((\ell-1/2)\pi t\bigr)}{(\ell-1/2)\pi},
\]
where $(\eta_\ell)_{\ell=1}^\infty$ is a sequence of independent, standard normal random variables.  For any $W = (W_t)_{t \in [0,1]} \in L^2[0,1]$, we can compute the transformed coefficients
\[
u_\ell(W) := \Phi\biggl(2^{1/2}(\ell-1/2)\pi \int_0^1 W_t \sin \bigl((\ell-1/2)\pi t\bigr) \, dt \biggr)
\]
for $\ell \in \mathbb{N}$.  We can therefore consider testing the independence of the random vectors $\bigl(u_1(X),\ldots,u_L(X)\bigr)$ and $\bigl(u_1(Y),\ldots,u_L(Y)\bigr)$, for some suitably chosen truncation level $L$.  For $\ell, m \in \mathbb{N}$ and $x \in L^2[0,1]$, let $p_{\ell m}^X(x) := 2^{1/2}\cos\bigl(2\pi m u_\ell(x)\bigr)$, and define $p_{\ell m}^Y(\cdot)$ similarly. 
The $U$-statistic kernel in this example can be written as
\begin{align*}
h\bigl((x_1,y_1),\ldots,(x_4,y_4)\bigr) &= \sum_{\ell_1,\ell_2=1}^L \sum_{m_1,m_2=1}^M \Bigl\{p_{\ell_1 m_1}^X(x_1)p_{\ell_2 m_2}^Y(y_1) p_{\ell_1 m_1}^X(x_2)p_{\ell_2 m_2}^Y(y_2) \\
&\hspace{3cm} - 2 p_{\ell_1 m_1}^X(x_1)p_{\ell_2 m_2}^Y(y_1) p_{\ell_1 m_1}^X(x_2)p_{\ell_2 m_2}^Y(y_3) \\
&\hspace{3cm} +  p_{\ell_1 m_1}^X(x_1)p_{\ell_2 m_2}^Y(y_2) p_{\ell_1 m_1}^X(x_3)p_{\ell_2 m_2}^Y(y_4)\Bigr\},
\end{align*}
where $L, M \in \mathbb{N}$.  In Figure~\ref{Fig:InfDim}, we plot the power functions of our USP test, estimated over 2000 repetitions, for three different sample sizes, namely $n \in \{50,100,200\}$, with $L=2$ and $M=1$.  As expected, the power of our test increases with both $r$ and~$n$.
\begin{figure}
\begin{center}
\includegraphics[width=0.7\textwidth]{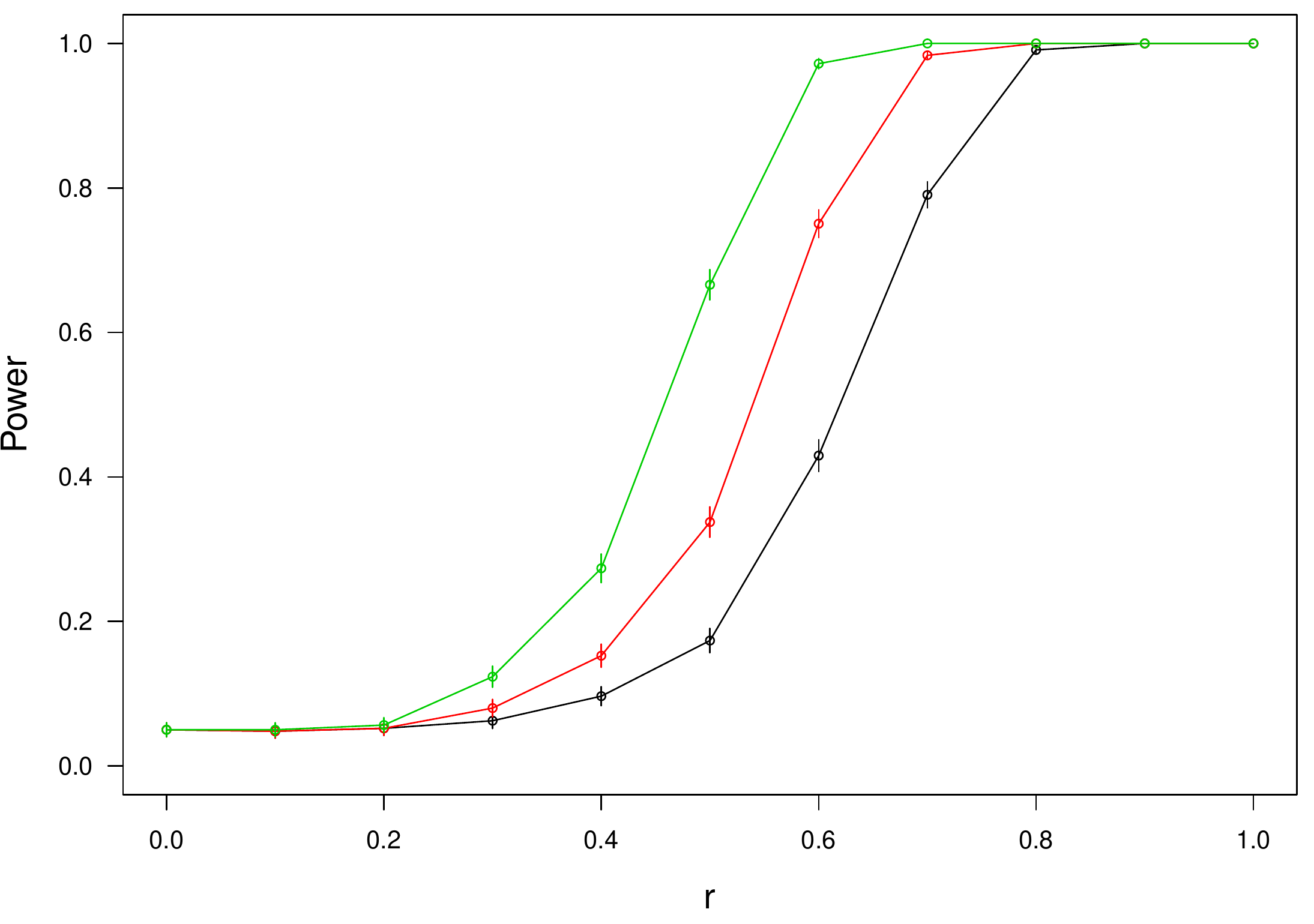} 
\end{center}
\caption{\label{Fig:InfDim}Estimated power functions for testing the independence of two Brownian motions with $n=50$ (black), $n=100$ (red) and $n=200$ (green). Error bars show two standard errors; other parameters: $\alpha = 0.05$, $B = 99$, $L=2$, $M=1$.}
\end{figure}

\section{Discussion and outlook}
\label{Sec:Discussion}

In this paper, we have introduced a new permutation test of independence based on a $U$-statistic estimator of the squared $L^2$-distance between a joint distribution and the product of its marginals.  Our methodology extends naturally to the problem of testing mutual independence of several random elements.  We have further demonstrated its minimax optimality in various settings; to the best of our knowledge, this is the first time that minimax optimality results have been established for such permutation tests.  We conclude by explaining how closely related ideas can be used to provide new goodness-of-fit tests and two-sample tests with desirable properties.

Consider $Z_1,\ldots,Z_n  \stackrel{\mathrm{iid}}{\sim} P \in \mathcal{P}$, where $\mathcal{P}$ is a dominated class of distributions on a separable, $\sigma$-finite measure space $(\mathcal{Z}, \mathcal{C}, \nu)$.  Suppose further that we wish to test $H_0:P = P_0$ against $H_1:P \neq P_0$, where $P_0 \in \mathcal{P}$.  Then, writing $f$ and $f_0$ respectively for the Radon--Nikodym derivatives of $P$ and $P_0$ with respect to $\nu$, we can construct a $U$-statistic estimator of the squared $L^2(\nu)$ distance between $f$ and $f_0$ in a very similar spirit to~\eqref{Eq:Dnhat}.  Since the null hypothesis is simple, there is no need for permutations, and we can obtain a critical value for the test by sampling from $P_0$.  

For two-sample tests, we can let $\mathcal{Y} = \{0,1\}$, so that testing the independence of $X$ and~$Y$ amounts to testing the equality of the distributions $X|\{Y=0\}$ and $X|\{Y=1\}$.  A small observation here is that the sample sizes from each conditional distribution are random (having a binomial distribution), whereas these are often treated as fixed in the usual two-sample testing formulation.  Our methodology and theory apply directly to this problem, therefore further extending its scope.

\section{Proofs of main results}
\label{Sec:Proofs}

\begin{proof}[Proof of Theorem~\ref{Thm:Hardness}]
Since $\psi$ is bounded, we have that $\psi \in L^2\bigl(\bigotimes_{i=1}^n \mu \bigr)$. Given $j \in \mathcal{J}$, $k \in \mathcal{K}$ and $I \subseteq [n]$ we write
\[
	b_{jk}^I := \bigg\langle \psi, \bigotimes_{i=1}^n \{ \mathbbm{1}_{\{i \in I\}} p_{jk} + \mathbbm{1}_{\{i \not\in I\}} p_{j_0 k_0}\} \bigg\rangle_{L^2(\bigotimes_{i=1}^n \mu)}.
\]
Since $r > \underline{\theta}\rho$, we have that $\mathcal{M}_\theta(r/\rho) \neq \emptyset$.  For $(j,k) \in \mathcal{M}_\theta(r/\rho)$ to be chosen later consider $f^* \equiv f_{jk}^* := p_{j_0 k_0} + \rho p_{jk} \in \mathcal{F}$, which satisfies $S_\theta(f^*) = \theta_{jk}^2 \rho^2 \leq r^2$ and $D(f^*) = \rho^2$. Then by Cauchy--Schwarz,
\begin{align}
\label{Eq:HardCalc}
	\mathbb{E}_{f^*}(\psi) = \bigg\langle \psi, \bigotimes_{i=1}^n f^* \bigg\rangle_{L^2(\bigotimes_{i=1}^n \mu)} &= \mathbb{E}_{p_{j_0 k_0}}(\psi) + \sum_{\emptyset \neq I \subseteq [n]}\rho^{|I|} b_{jk}^I \nonumber \\
                                                                                                                            &\leq  \alpha + \{(1+\rho^2)^n -1 \}^{1/2} \biggl\{ \sum_{\emptyset \neq I \subseteq [n]} (b_{jk}^I)^2 \biggr\}^{1/2}.
\end{align}
Now, observe that
\[
	\sum_{(j,k) \in \mathcal{M}_\theta(r/\rho)} \sum_{\emptyset \neq I \subseteq [n]} (b_{jk}^I)^2   \leq \| \psi \|_{L^2(\otimes_{i=1}^n \mu)}^2 = \mathbb{E}_{p_{j_0 k_0}} (\psi^2) \leq \alpha.
\]
Hence, for any $\eta>0$ we may choose $(j,k) \in \mathcal{M}_\theta(r/\rho)$ such that
\[
	\sum_{\emptyset \neq I \subseteq [n]} (b_{jk}^I)^2 \leq \frac{\alpha}{|\mathcal{M}_\theta(r/\rho)|} + \eta.
\]
The first claim of Theorem~\ref{Thm:Hardness} follows from this combined with~\eqref{Eq:HardCalc}.

For the second part, first note the definitions of $\Xi$, $\mathcal{F}_{\xi}(\rho)$ and $\rho^*(n,\alpha,\beta,\xi)$ immediately after \textbf{(A1)}.  For the choice of $j,k$ in the first part of the proof, let $\theta' = (\theta'_{j'k'})_{j' \in \mathcal{J},k' \in \mathcal{K}}$ be given by
\[
\theta'_{j'k'} := \left\{ \begin{array}{ll} 0 & \mbox{ if $j'=j$ and $k'=k$} \\
\infty & \mbox{ otherwise.} \end{array} \right. 
\]
Now $f^* \in \mathcal{F}_{(\theta',r',2)}(\rho)$ for any $r' > 0$.  Applying Theorem~\ref{Thm:UpperBound} with $\mathcal{M} = \{(j,k)\}$ then yields that there exists $C = C(\alpha,\beta) > 0$ such that $\rho^*(n,\alpha,\beta,\xi) \leq C^{1/2}/n^{1/2}$.  In other words, there exists $\psi_{f^*} \in \Psi(\alpha)$ such that $\mathbb{E}_{f^*}(\psi_{f^*}) \geq 1 - \beta$ whenever $n > C/\rho^2$.  Finally, the proof of Theorem~\ref{Thm:UpperBound} reveals that $\psi_{f^*}$ may be taken to be a permutation test (in fact the permutation test described in Section~\ref{Sec:MainResults} with $\mathcal{M} = \{(j,k)\}$), as required.
\end{proof}

\begin{proof}[Proof of Theorem~\ref{Thm:UpperBound}]
Consider the test of Section~\ref{Sec:MainResults}.  Choose $B \geq 2(\frac{1}{\alpha \beta} -1)$, and suppose $f \in \mathcal{F}^*$ were such that
\begin{equation}
\label{Eq:BiasVariance}
	D(f) \geq \max \biggl[ 2 \bigl| \mathbb{E}_f \bigl( \hat{D}_n - \hat{D}_n^{(1)} \bigr) - D(f) \bigr|, \Bigl\{ \frac{8}{\alpha \beta} \mathrm{Var}_f( \hat{D}_n - \hat{D}_n^{(1)} ) \Bigr\}^{1/2} \biggr].
\end{equation}
Then, by two applications of Markov's inequality, we would have that
\begin{align*}
	\mathbb{P}_f(P > \alpha) &= \mathbb{P}_f \biggl( 1+ \sum_{b=1}^B \mathbbm{1}_{\{\hat{D}_n \leq \hat{D}_n^{(b)}\}} > (1+B) \alpha \biggr) \leq \frac{1+ B \mathbb{P}_f(\hat{D}_n \leq \hat{D}_n^{(1)})}{(1+B)\alpha} \\
	& \leq \frac{1}{(1+B)\alpha} \biggl[ 1 + \frac{B \mathrm{Var}_f\bigl( \hat{D}_n - \hat{D}_n^{(1)}\bigr)}{\bigl\{ \mathbb{E}_f\bigl(\hat{D}_n - \hat{D}_n^{(1)}\bigr) \bigr\}^2} \biggr] \leq \frac{1}{(1+B)\alpha} \biggl( 1 + \frac{B \alpha \beta}{2} \biggr) \leq \beta.
\end{align*}
We may think of $\hat{D}_n - \hat{D}_n^{(1)}$ as an estimator of $D(f)$, so that~\eqref{Eq:BiasVariance} ensures that the strength of the dependence $D(f)$ outweighs the bias and standard deviation of the estimator so that we can detect the dependence using our test, up to the given probabilities of error. The remainder of the proof is dedicated to bounding the bias and variance for a given $\xi \in \Xi$, which enables us to choose $\rho$ so that~\eqref{Eq:BiasVariance} holds for all $f \in \mathcal{F}_\xi(\rho)$, and hence ensures that $\rho^*(n,\alpha,\beta,\xi) \leq \rho$.  Henceforth we will write $\Pi$ as shorthand for $\Pi_1$; moreover, for some $\rho > 0$ to be chosen later, we fix $f \in \mathcal{F}_{\xi}(\rho)$ and write $D, a_{jk}, a_{j\bullet}, a_{\bullet k}$ instead of $D(f), a_{jk}(f), a_{j\bullet}(f), a_{\bullet k}(f)$ respectively.

Given $(i_1, i_2) \in \mathcal{I}_2$ write $\sigma_{i_1 i_2} \in \mathcal{S}_n$ for the transposition of $i_1$ and $i_2$, and note that $\Pi \overset{d}{=} \Pi \circ \sigma_{i_1 i_2}$.  Thus $\bigl(\Pi(1),\Pi(2)\bigr) \stackrel{d}{=} \bigl(\Pi(1),\Pi(3)\bigr)$, so for every $(j,k) \in \mathcal{M}$ we have that 
\[
	p_{jk}(X_1,Y_{\Pi(1)})p_{jk}(X_2,Y_{\Pi(2)}) \overset{d}{=} p_{jk}(X_1,Y_{\Pi(1)})p_{jk}(X_2,Y_{\Pi(3)}).
\]
Similarly, $p_{jk}(X_1,Y_{\Pi(1)})p_{jk}(X_2,Y_{\Pi(3)}) \overset{d}{=} p_{jk}(X_1,Y_{\Pi(2)})p_{jk}(X_3,Y_{\Pi(4)})$, so that
\begin{align*}
	\mathbb{E}(\hat{D}_n^{(1)}) &= \sum_{(j,k)\in \mathcal{M}} \mathbb{E} \bigl\{ p_{jk}(X_1,Y_{\Pi(1)})p_{jk}(X_2,Y_{\Pi(2)}) \\
	&\hspace{20pt}- 2p_{jk}(X_1,Y_{\Pi(1)})p_{jk}(X_2,Y_{\Pi(3)}) + p_{jk}(X_1,Y_{\Pi(2)})p_{jk}(X_3,Y_{\Pi(4)}) \bigr\} =0.
\end{align*}
Thus, using our Sobolev smoothness condition to bound the truncation error,
\begin{align}
\label{Eq:Bias}
	\bigl| \mathbb{E}( \hat{D}_n - \hat{D}_n^{(1)})  - D\bigr| &= \bigl| \mathbb{E} (\hat{D}_n) - D \bigr| = \biggl| \sum_{(j,k)\in \mathcal{M}} (a_{jk}-a_{j\bullet}a_{\bullet k})^2 - D \biggr| \nonumber \\
	&= \sum_{(j,k) \in (\mathcal{J} \times \mathcal{K}) \setminus \mathcal{M}} (a_{jk}-a_{j\bullet}a_{\bullet k})^2  \leq \frac{r^2}{\inf\{ \theta_{jk}^2 : (j,k) \not\in \mathcal{M} \}}.
\end{align}

We now turn to bounding $\mathrm{Var}(\hat{D}_n-\hat{D}_n^{(1)})$. First write $\bar{h}$ for the symmetrised version of $h$, given by
\begin{align}
  \label{Eq:barh}
  \bar{h} \bigl( (x_1,&y_1), \ldots , (x_4,y_4) \bigr) := \frac{1}{4!} \sum_{\sigma \in \mathcal{S}_4} h \bigl( (x_{\sigma(1)},y_{\sigma(1)}), \ldots , (x_{\sigma(4)},y_{\sigma(4)}) \bigr).
\end{align}
By, e.g. \citet[Lemma~A, p.~183]{Serfling1980}, we have that
\begin{align}
\label{Eq:StandardVariance}
	\mathrm{Var}(\hat{D}_n) &= \mathrm{Var}\biggl(\frac{1}{4!\binom{n}{4}} \sum_{(i_1,\ldots,i_4) \in \mathcal{I}_4} \bar{h}\bigl((X_{i_1},Y_{i_1}),\ldots,(X_{i_4},Y_{i_4})\bigr)\biggr) \nonumber \\
	&= \binom{n}{4}^{-1} \sum_{c=1}^4 \binom{4}{c} \binom{n-4}{4-c} \zeta_c,
\end{align}
where $\zeta_c:= \mathrm{Var} \bigl( \mathbb{E}\bigl\{\bar{h}\bigl((X_1,Y_1),\ldots,(X_4,Y_4)\bigr) \bigm| (X_1,Y_1), \ldots, (X_c,Y_c)\bigr\}\bigr)$, and moreover $\zeta_1 \leq \zeta_2 \leq \zeta_3 \leq \zeta_4$. For each $j \in \mathcal{J}$ write $\mathcal{K}_j^{\mathcal{M}} := \{k \in \mathcal{K}: (j,k) \in \mathcal{M}\}$ and for each $k \in \mathcal{K}$ write $\mathcal{J}_k^\mathcal{M}:=\{ j \in \mathcal{J} : (j,k) \in \mathcal{M}\}$. Then, using Cauchy--Schwarz,
\begin{align}
\label{Eq:zeta1bound}
	\zeta_1&=\mathrm{Var} \Bigl( \mathbb{E}\bigl\{\bar{h}\bigl((X_1,Y_1),\ldots,(X_4,Y_4)\bigr) \bigm| (X_1,Y_1)\bigr\} \Bigr) \nonumber \\
	&= \frac{1}{4} \mathrm{Var} \biggl( \sum_{(j,k)\in \mathcal{M}} (a_{jk}-a_{j\bullet}a_{\bullet k}) \bigl\{ p_{jk}(X_1,Y_1) - p_j^X(X_1)a_{\bullet k} - a_{j\bullet} p_k^Y(Y_1) \bigr\} \biggr) \nonumber \\
	& \leq \frac{3A}{4} \biggl\{ \biggl\| \sum_{(j,k)\in \mathcal{M}} \!\! (a_{jk}-a_{j\bullet}a_{\bullet k}) p_{jk} \biggr\|_{L^2(\mu)}^2 \!\! + \biggl\| \sum_{(j,k)\in \mathcal{M}} \!\! (a_{jk}-a_{j\bullet}a_{\bullet k}) p_{j}^X a_{\bullet k} \biggr\|_{L^2(\mu_X)}^2 \nonumber \\
	&  \hspace{150pt}+ \biggl\| \sum_{(j,k)\in \mathcal{M}} \!\! (a_{jk}-a_{j\bullet}a_{\bullet k}) a_{j\bullet} p_k^Y \biggr\|_{L^2(\mu_Y)}^2 \biggr\}  \nonumber \\
	& \leq \frac{3A}{4} \biggl[ D  +   \sum_{j \in \mathcal{J}} \biggl\{ \sum_{k \in \mathcal{K}_j^{\mathcal{M}}} \!\! (a_{jk}-a_{j\bullet}a_{\bullet k})a_{\bullet k} \biggr\}^2 \!\! + \!\! \sum_{k \in \mathcal{K}} \biggl\{ \sum_{j \in \mathcal{J}_k^{\mathcal{M}}} \!\! (a_{jk}-a_{j\bullet}a_{\bullet k})a_{j\bullet} \biggr\}^2 \biggr] \nonumber \\
	& \leq \frac{3AD}{4}( 1 + \|f_Y\|_{L^2(\mu_Y)} +  \|f_X\|_{L^2(\mu_X)} ) \leq \frac{9A^2D}{4}.
\end{align}
Observe that we have 
\begin{align*}
	\zeta_4 = \mathrm{Var} \ \bar{h} \bigl((X_1,Y_1), \ldots , (X_4,Y_4) \bigr) \leq \mathrm{Var} \ h \bigl((X_1,Y_1), \ldots , (X_4,Y_4) \bigr).
\end{align*}
One possibility, therefore, is to simply apply the bound $\zeta_4 \leq \|h\|_\infty^2$. On the other hand, by Cauchy--Schwarz, we can say that
\begin{align}
\label{Eq:Zeta3}
  \zeta_4 & 
	\leq A^4 \int_{\mathcal{X} \times \mathcal{Y}} \ldots \int_{\mathcal{X} \times \mathcal{Y}} h^2 \bigl( (x_1,y_1), \ldots,(x_4,y_4) \bigr) \, d\mu(x_1,y_1) \ldots \, d\mu(x_4,y_4) \nonumber \\
	& \leq 18A^4 \int_{\mathcal{X} \times \mathcal{Y}} \int_{\mathcal{X} \times \mathcal{Y}} \biggl\{ \sum_{(j,k)\in \mathcal{M}} p_{jk}(x,y) p_{jk}(x',y') \biggr\}^2 \, d\mu(x,y) \,d\mu(x',y') \nonumber \\
	& \leq 18A^4 |\mathcal{M}|.
\end{align}
We therefore have that
\begin{equation}
\label{Eq:Var1}
	\mathrm{Var}(\hat{D}_n ) \leq \frac{16 \zeta_1}{n} + \frac{72 \zeta_4}{n(n-1)} \leq \frac{36A^2 D}{n} + \frac{72 \min(\|h\|_\infty^2, 18A^4 |\mathcal{M}|)}{n(n-1)}.
\end{equation}

Next, with the same functions $h$ and $\bar{h}$ as above, we may write
\[
	\hat{D}_n^{(1)} = \frac{1}{4! \binom{n}{4}} \sum_{(i_1,\ldots,i_4) \in \mathcal{I}_4 } \bar{h} \bigl( (X_{i_1},Y_{\Pi(i_1)}), \ldots ,(X_{i_4},Y_{\Pi(i_4)}) \bigr).
\]
A simplifying property of $\bar{h}$ is that for every $(x,y) \in \mathcal{X} \times \mathcal{Y}$,
\begin{align}
\label{Eq:Degeneracy}
	\mathbb{E} \bigl\{ \bar{h} \bigl( (x,y), (X_1,Y_2), (X_3,Y_4), (X_5,Y_6) \bigr) \bigr\}= 0.
\end{align}
Since we also have to deal with the uniformly random permutation $\Pi$, we cannot directly appeal to standard $U$-statistic theory for our bounds on $\mathrm{Var}(\hat{D}_n^{(1)})$. However, we can develop an analogue of~\eqref{Eq:StandardVariance} by writing
\begin{align}
  \label{Eq:VarDn1}
	\mathrm{Var}( \hat{D}_n^{(1)}) &= \frac{1}{4! \binom{n}{4}} \sum_{(i_1,\ldots,i_4) \in \mathcal{I}_4} \mathrm{Cov} \Bigl( \bar{h} \bigl( (X_{1},Y_{\Pi(1)}), \ldots , (X_4,Y_{\Pi(4)}) \bigr), \nonumber \\
	&\hspace{2cm}\bar{h} \bigl( (X_{i_1},Y_{\Pi(i_1)}), \ldots , (X_{i_4},Y_{\Pi(i_4)}) \bigr) \Bigr) \nonumber \\
	& =\frac{1}{\binom{n}{4}} \sum_{c=0}^4 \binom{4}{c} \binom{n-4}{4-c} \mathrm{Cov} \Bigl( \bar{h} \bigl( (X_{1},Y_{\Pi(1)}), \ldots , (X_4,Y_{\Pi(4)}) \bigr), \nonumber \\
	&\hspace{2cm} \bar{h} \bigl( (X_{1},Y_{\Pi(1)}), \ldots,  (X_{c},Y_{\Pi(c)}),(X_{5},Y_{\Pi(5)}),\ldots, (X_{8-c},Y_{\Pi(8-c)} ) \bigr) \Bigr) \nonumber \\
	& =: \frac{1}{\binom{n}{4}} \sum_{c=0}^4 \binom{4}{c} \binom{n-4}{4-c} \tilde{\zeta}_c.
\end{align}
For $c=2,3,4$ we will use the crude bound
\begin{align}
\label{Eq:CrudePerm}
	\max(\tilde{\zeta}_2,\tilde{\zeta}_3, \tilde{\zeta}_4) &\leq \max_{\sigma \in \mathcal{S}_n} \mathbb{E} \Bigl\{ h^2 \bigl( (X_{1},Y_{\sigma(1)}), \ldots, (X_4,Y_{\sigma(4)}) \bigr) \Bigr\} \nonumber \\
	& \leq \min( \|h\|_\infty^2, 18A^8 |\mathcal{M}|),
\end{align}
similarly to~\eqref{Eq:Zeta3}. To bound $\tilde{\zeta}_0$ and $\tilde{\zeta}_1$ we must first bound two combinatorial probabilities. First,
\begin{align*}
	\mathbb{P}\bigl( | [7]  \cap \{\Pi(1),\ldots, \Pi(7) \}| \geq 1\bigr) \leq 7 \mathbb{P}\bigl( \Pi(1) \in [7]\bigr) = \frac{49}{n}.
\end{align*}
Now, similarly,
\begin{align}
\label{Eq:Combo2}
	\mathbb{P}\bigl(  | [8]  \cap \{\Pi(1),\ldots,\Pi(8)\}| \geq 2\bigr) &\leq \binom{8}{2} \mathbb{P}\bigl( \Pi(1),\Pi(2) \in [8]\bigr) \nonumber \\
	& = 2\binom{8}{2}^2 \mathbb{P}\bigl(\Pi(1)=1, \Pi(2)=2\bigr) = \frac{1568}{n(n-1)}.
\end{align}
The first of these allows us to use~\eqref{Eq:Degeneracy}, Cauchy--Schwarz and~\eqref{Eq:CrudePerm} to write
\begin{align}
  \label{Eq:zetatilde1}
	\tilde{\zeta}_1 &= \mathrm{Cov} \Bigl( \bar{h} \bigl( (X_{1},Y_{\Pi(1)}), \ldots, (X_{4},Y_{\Pi(4)}) \bigr), \nonumber \\
	&\hspace{3cm}\bar{h} \bigl( (X_{1},Y_{\Pi(1)}), (X_{5},Y_{\Pi(5)}),(X_{6},Y_{\Pi(6)}), (X_{7},Y_{\Pi(7)}) \bigr) \Bigr) \nonumber \\
                        & \leq \mathbb{P}\bigl( [7]  \cap \{\Pi(1),\ldots, \Pi(7)\} \! = \! \emptyset\bigr) \mathbb{E}\bigl\{ \bar{h} \bigl( (X_{1},Y_{8}), (X_{2},Y_{9}),(X_{3},Y_{10}),(X_4,Y_{11}) \bigr) \nonumber \\
  &\hspace{3cm}\times \bar{h} \bigl( (X_{1},Y_{8}), (X_{5},Y_{12}),(X_{6},Y_{13}),(X_7,Y_{14}) \bigr) \bigr\} \nonumber \\
                        & \hspace{3cm} + \frac{49}{n} \max_{\sigma \in \mathcal{S}_n} \mathbb{E} \Bigl\{ h^2 \bigl( (X_{1},Y_{\sigma(1)}), \ldots , (X_4,Y_{\sigma(4)}) \bigr) \Bigr\} \nonumber \\
  &\leq \frac{49}{n} \min( \|h\|_\infty^2, 18A^8 |\mathcal{M}|).
\end{align}
Finally, we may now use~\eqref{Eq:Degeneracy}, Cauchy--Schwarz,~\eqref{Eq:CrudePerm} and~\eqref{Eq:Combo2} to similarly write
\begin{align}
  \label{Eq:zetatilde0}
  \tilde{\zeta}_0= \mathrm{Cov} \Bigl( \bar{h} \bigl( (X_{1},Y_{\Pi(1)}), \ldots ,(X_4,Y_{\Pi(4)}) \bigr), &\bar{h} \bigl( (X_{5},Y_{\Pi(5)}), \ldots ,(X_{8},Y_{\Pi(8)}) \bigr) \Bigr) \nonumber \\
	& \leq \frac{1568}{n(n-1)}\min( \|h\|_\infty^2, 18A^8 |\mathcal{M}|).
\end{align}
From~\eqref{Eq:VarDn1},~\eqref{Eq:CrudePerm},~\eqref{Eq:zetatilde1},~\eqref{Eq:zetatilde0} we have now established that
\begin{align}
\label{Eq:Var2}
	\mathrm{Var}(\hat{D}_n^{(1)}) & \leq  \tilde{\zeta}_0  + \frac{16}{n} \tilde{\zeta}_1 + \frac{72}{n(n-1)} \max(\tilde{\zeta}_2,\tilde{\zeta}_3,\tilde{\zeta}_4) \leq \frac{2424 \min( \|h\|_\infty^2, 18A^8 |\mathcal{M}|)}{n(n-1)}. 
\end{align}
Thus, from~\eqref{Eq:Var1} and~\eqref{Eq:Var2} we deduce that
\begin{equation}
\label{Eq:Variance}
	\mathrm{Var}(\hat{D}_n - \hat{D}_n^{(1)}) \leq \frac{72A^2D}{n} + \frac{4992\min( \|h\|_\infty^2, 18A^8 |\mathcal{M}|)}{n(n-1)}. 
\end{equation}

Now by substituting~\eqref{Eq:Bias} and~\eqref{Eq:Variance} into~\eqref{Eq:BiasVariance} we can see that if
\begin{align}
\label{Eq:FullDetail}
	D(f) & \geq \max \biggl\{ \frac{2r^2}{\inf \{ \theta_{jk}^2 : (j,k) \not\in \mathcal{M}\}}, \frac{1152A^2}{n \alpha \beta}, \frac{283 \min( \|h\|_\infty, 5 A^4 |\mathcal{M}|^{1/2})}{\{n(n-1)\alpha \beta\}^{1/2}} \biggr\},  
\end{align}
then we have controlled the error probabilities as required.  
\end{proof}

\begin{proof}[Proof of Corollary~\ref{Cor:Sobolev}]
There exists $C=C(d_X,d_Y) \in (1,\infty)$ such that for any $T>0$ we have
\begin{align*}
	|\{(j,k) \in \mathcal{J} \times \mathcal{K} :  \theta_{jk} \leq T \}| &= |\{ j \in \mathcal{J}: \|j\|_1^{s_X} \leq T\}|  |\{ k \in \mathcal{K} : \|k\|_1^{s_Y} \leq T \}| \\
	& \leq (T^{1/s_X} + 1)^{d_X} (T^{1/s_Y} + 1)^{d_Y} < C(T \vee 1)^{d/s}.
\end{align*}
From this we can infer that if $m>C$ then $\theta_{\omega(m)} > (m/C)^{s/d}$, and so
\[
	m_0(nr^2) \leq \max\{C, (nr^2)^{2d/(4s+d)}C^{4s/(4s+d)} \} \leq C \{(nr^2) \vee 1\}^{2d/(4s+d)}.
\]
It now follows from~\eqref{Eq:m0bound} that there exists $C=C(d_X,d_Y,\alpha,\beta,A)>0$ such that if $n \geq 16$ and $nr^2 \geq 1$ then
\[
	\rho^*(n,\alpha,\beta,\xi) \leq C\Bigl( \frac{r^d}{n^{2s}} \Bigr)^{1/(4s+d)},
\]
as required.
\end{proof}

\begin{proof}[Proof of Proposition~\ref{Prop:GeneralAdapt}]
By~\eqref{Eq:FullDetail} in the proof of Theorem~\ref{Thm:UpperBound}, we see that we reject $H_0$ with probability at least $1-\beta$, provided that $n \geq 16$ and
\[
	\rho^2 \geq \min_{m \in K_*} \max \biggl\{ \frac{2r^2}{\theta_{\omega(m+1)}^2}, \frac{1152A^2 \gamma}{n \alpha \beta}, \frac{1415 A^4 m^{1/2} \gamma^{1/2}}{ \{n(n-1) \alpha \beta\}^{1/2}} \biggr\}.
\]
Since $m_0(t) \leq t^2/ \theta_0^4 + 1$, there exists $n_0=n_0(R_0,\theta_0) \geq 16$ such that for all $n \geq n_0$ we have $m_0( n r^2/ \log^{1/2} n) \leq 2^\gamma +1$. But then, for $n \geq \max(n_0,e^3)$,
\begin{align*}
	\min_{m \in K_*} \max \biggl\{ &\frac{2r^2}{\theta_{\omega(m+1)}^2}, \frac{1152A^2 \gamma}{n \alpha \beta}, \frac{1415 A^4 m^{1/2} \gamma^{1/2}}{ \{n(n-1) \alpha \beta\}^{1/2}} \biggr\} \\
	& \leq 2^{1/2} \min_{m \in [2^\gamma] \setminus \{1\}} \max \biggl\{ \frac{2r^2}{\theta_{\omega(m+1)}^2}, \frac{1152A^2 \gamma}{n \alpha \beta}, \frac{1415 A^4 m^{1/2} \gamma^{1/2}}{ \{n(n-1) \alpha \beta\}^{1/2}} \biggr\} \\
	& \lesssim_{A,\alpha,\beta} \min_{m \in \{3,4,\ldots, 2^\gamma+1\}} \max \biggl\{ \frac{r^2}{\theta_{\omega(m)}^2}, \frac{\log n }{n}, \frac{m^{1/2} \log^{1/2} n}{n} \biggr\} \\
	&\leq \max \biggl\{ \frac{\log^{1/2}n}{n} m_0^{1/2} \biggl( \frac{n r^2}{ \log^{1/2}n} \biggr), \frac{\log n}{n} \biggr\},
\end{align*}
and the result follows.
\end{proof}

\begin{proof}[Proof of Proposition~\ref{Prop:SobolevAdapt}]
As in the proof of Proposition~\ref{Prop:GeneralAdapt}, by~\eqref{Eq:FullDetail} in the proof of Theorem~\ref{Thm:UpperBound}, we see that we reject $H_0$ with probability at least $1-\beta$ provided that $n \geq 16$ and 
\[
	\rho^2 \geq \min_{\substack{m_X \in K_X\\m_Y \in K_Y}} \max \biggl\{ \frac{2r^2}{m_X^{2s_X} \vee m_Y^{2s_Y} }, \frac{1152A^2 \gamma_X \gamma_Y}{n \alpha \beta}, \frac{1415 A^4 |\mathcal{M}_{m_X,m_Y}|^{1/2} (\gamma_X \gamma_Y)^{1/2} }{ \{n(n-1) \alpha \beta\}^{1/2}} \biggr\}.
\]
Since $|\mathcal{M}_{m_X,m_Y}| \asymp_{d_X,d_Y} m_X^{d_X} m_Y^{d_Y}$, if $(m_X,m_Y)$ were not restricted to lie in $K_X \times K_Y$, then we would maximise the right-hand side here by taking $m_X^{s_X} \asymp m_Y^{s_Y} \asymp_{\alpha,\beta,d_X,d_Y,A} (nr^2 / \log n)^{ \frac{2s}{4s+d}}$. In fact, recalling that $d/s=d_X/s_X+d_Y/s_Y$, we have that
\[
	(nr^2 / \log n)^{ \frac{2s}{s_X(4s+d)}} \lesssim_{R_0,s_X,s_Y,d_X,d_Y} \Bigl( \frac{n}{ \log n} \Bigr)^\frac{2}{4s_X+d_X+d_Y s_X/s_Y} \ll n^{2/d_X} \leq 2^{\gamma_X}.
\]
As in the proof of Proposition~\ref{Prop:GeneralAdapt}, then, we may choose $(m_X,m_Y) \in K_X \times K_Y$ so as to ensure that the separation in~\eqref{Eq:SobolevAdaptBound} suffices to guarantee power at least $1 - \beta$.
\end{proof}

\begin{proof}[Proof of Lemma~\ref{Lemma:LowerBound}]
We will prove that
\[
d_{\mathrm{TV}}^2 \bigl( \mathbb{P}_{p_{j_0k_0}}^{\otimes n}, \mathbb{E} \mathbb{P}_f^{\otimes n} \bigr) \leq \frac{\exp( \frac{n^2}{2} \sum_{j \in \mathcal{J} \setminus \{j_0\},k \in \mathcal{K} \setminus \{k_0\}} a_{jk}^4)}{4\mathbb{P}(p \in \mathcal{F})^2} - \frac{1}{4}
\]
in the case that $n$ is even. If, on the other hand, $n$ is odd then we will use the fact that $d_\mathrm{TV}(\nu_1^{\otimes n},\nu_2^{\otimes n}) \leq d_\mathrm{TV}(\nu_1^{\otimes (n+1)},\nu_2^{\otimes (n+1)})$ for any probability measures $\nu_1,\nu_2$ to complete the proof.

Let $f^{(1)},f^{(2)}$ be independent copies of $f$ and let $p^{(1)},p^{(2)}$ be independent copies of~$p$. Then we have that
\begin{align*}
	\frac{1}{4} &+ d_{\mathrm{TV}}^2\bigl( \mathbb{P}_{p_{j_0k_0}}^{\otimes n}, \mathbb{E} \mathbb{P}_f^{\otimes n} \bigr) \leq \frac{1}{4} + \frac{1}{4}d_{\chi^2}^2 \bigl( \mathbb{P}_{p_{j_0k_0}}^{\otimes n}, \mathbb{E} \mathbb{P}_f^{\otimes n} \bigr) \\
	&= \frac{1}{4}\int_{\mathcal{X} \times \mathcal{Y}} \ldots \int_{\mathcal{X} \times \mathcal{Y}} \bigl( \mathbb{E} \{f(x_1,y_1) \ldots f(x_n,y_n)\} \bigr)^2 \,d\mu(x_n,y_n) \ldots \,d \mu(x_1,y_1) \\
	&  = \frac{\mathbb{E}\bigl\{ \langle p^{(1)}, p^{(2)} \rangle_{L^2(\mu)}^n \mathbbm{1}_{\{p^{(1)},p^{(2)} \in \mathcal{F}\}}\bigr\}}{4\mathbb{P}(p^{(1)},p^{(2)} \in \mathcal{F})} \leq \frac{\mathbb{E}\bigl\{ \langle p^{(1)}, p^{(2)} \rangle_{L^2(\mu)}^n\bigr\}}{4\mathbb{P}(p \in \mathcal{F})^2},
\end{align*}
and all that remains is to bound the numerator in this final expression. Let $(\xi_{jk}^{(1)}),(\xi_{jk}^{(2)})$ be independent copies of $(\xi_{jk})$ and write
\[
  Y := \sum_{j \in \mathcal{J} \setminus \{j_0\}, k \in \mathcal{K} \setminus \{k_0\}} a_{jk}^2 \xi_{jk}^{(1)} \xi_{jk}^{(2)} \overset{d}{=} \sum_{j \in \mathcal{J} \setminus \{j_0\}, k \in \mathcal{K} \setminus \{k_0\}} a_{jk}^2 \xi_{jk}.
\]
The random variable $Y$ has a distribution that is symmetric about the origin, so for odd $m$ we have $\mathbb{E}(Y^m)=0$. For $m,r \in \mathbb{N}$ with $r \leq m$ write $A_{m,r}:=\{ \alpha=(\alpha_1,\ldots,\alpha_r) \in \mathbb{N}^r : \alpha_1+\ldots+\alpha_r =m\}$ and $(2m-1)!!=(2m-1)(2m-3)\ldots3 = \frac{(2m)!}{m! 2^m}$ for the double factorial.  It is also convenient to define the multinomial coefficient: for $N \in \mathbb{N}$ and $m_1,\ldots,m_r \in \mathbb{N}_0$ with $m_1+\ldots+m_r = N$, we set
\[
\binom{N}{m_1,m_2,\ldots,m_r} := \frac{N!}{m_1!m_2!\ldots m_r!}.
\]
  Then, for every $m \in \{0,1,\ldots,n/2\}$, we have 
\begin{align*}
\mathbb{E}(Y^{2m}) &= \sum_{\substack{j_1,\ldots,j_m \in \mathcal{J} \setminus \{j_0\} \\ k_1,\ldots,k_{2m} \in \mathcal{K} \setminus \{k_0\}}} a_{j_1k_1}^2 \ldots a_{j_{2m}k_{2m}}^2 \mathbb{E}( \xi_{j_1k_1} \ldots \xi_{j_{2m} k_{2m}} ) \\
	& = \sum_{r=1}^m \sum_{\alpha \in A_{m,r}} \sum_{\substack{(j_1,k_1),\ldots,(j_r,k_r) \\ \text{distinct}}}  a_{j_1 k_1}^{4 \alpha_1} \ldots a_{j_rk_r}^{4 \alpha_r} \times \frac{1}{r!}\binom{2m}{2\alpha_1,2\alpha_2,\ldots,2\alpha_r} \\
  &= \sum_{r=1}^m \sum_{\alpha \in A_{m,r}} \sum_{\substack{(j_1,k_1),\ldots,(j_r,k_r) \\ \text{distinct}}}  a_{j_1 k_1}^{4 \alpha_1} \ldots a_{j_rk_r}^{4 \alpha_r} \times \frac{(2m-1)!!  \binom{m}{\alpha_1,\ldots,\alpha_r}}{r! (2\alpha_1-1)!! \ldots (2 \alpha_r-1)!!}\\
  &\leq \sum_{r=1}^m \sum_{\alpha \in A_{m,r}} \sum_{\substack{(j_1,k_1),\ldots,(j_r,k_r) \\ \text{distinct}}}  a_{j_1 k_1}^{4 \alpha_1} \ldots a_{j_rk_r}^{4 \alpha_r} \times \frac{(2m-1)!! \binom{m}{\alpha_1,\ldots,\alpha_r}}{r!}  \\
	& = (2m-1)!! \biggl( \sum_{j \in \mathcal{J} \setminus \{j_0\}, k \in \mathcal{K} \setminus \{k_0\}} a_{jk}^4 \biggr)^m. 
\end{align*}
It therefore follows that 
\begin{align*}
	&\mathbb{E}\bigl\{ \langle p^{(1)}, p^{(2)} \rangle_{L^2(\mu)}^n \bigr\} = \mathbb{E}\bigl\{ (1+Y)^n\bigr\} = \sum_{m=0}^{n/2} \binom{n}{2m} \mathbb{E}(Y^{2m}) \\
                                                                                &\leq \sum_{m=0}^{n/2} \frac{1}{m!} \biggl( \frac{n^2}{2} \sum_{j \in \mathcal{J} \setminus \{j_0\}, k \in \mathcal{K} \setminus \{k_0\}} a_{jk}^4 \biggr)^m \leq \exp \biggl(  \frac{n^2}{2} \sum_{j \in \mathcal{J} \setminus \{j_0\}, k \in \mathcal{K} \setminus \{k_0\}} a_{jk}^4 \biggr),
\end{align*}
as required.
\end{proof}

\begin{proof}[Proof of Theorem~\ref{Thm:LowerBound1}]
For $m \in \mathbb{N}$, set
\[
	c_m :=\min\biggl( \frac{r^2}{\theta_{\omega(m)}^2}, \frac{(2m)^{1/2}}{n+1} \log^{1/2}\bigl(1+(1-\gamma)^2\bigr), \frac{(A-1)^2 \wedge 1}{m \bar{p}^2} \biggr)
\]
and
\[
	a_{\omega(\ell)} := \left\{ \begin{array}{ll} c_m^{1/2}/m^{1/2} & \mbox{for $\ell \in [m]$} \\ 0 & \mbox{otherwise.} \end{array} \right. 
\]
Then, with the convention that $\infty.0 = 0$, we have
\begin{equation}
\label{Eq:thetaa}  
	\sum_{j \in \mathcal{J} \setminus \{j_0\},k \in \mathcal{K} \setminus \{k_0\}} \theta_{jk}^2 a_{jk}^2 = \frac{c_m}{m} \sum_{\ell=1}^m \theta_{\omega(\ell)}^2 \leq c_m \theta_{\omega(m)}^2 \leq r^2.
\end{equation}
Moreover,
\[
	\sum_{j\in \mathcal{J} \setminus \{j_0\},k \in \mathcal{K} \setminus \{k_0\}} a_{jk}^4 = \frac{c_m^2}{m} \leq \frac{2}{(n+1)^2}\log\bigl( 1+ (1-\gamma)^2\bigr),
\]
and
\begin{equation}
  \label{Eq:apbar}
	\sum_{j\in \mathcal{J} \setminus \{j_0\},k \in \mathcal{K} \setminus \{k_0\}} a_{jk} \|p_{jk}\|_\infty = m^{1/2}c_m^{1/2} \bar{p} \leq (A-1) \wedge 1.
      \end{equation}
      Now, writing $\rho = \bigl\{\sum_{\mathcal{J} \setminus \{j_0\},k \in \mathcal{K} \setminus \{k_0\}} a_{jk}^2\bigr\}^{1/2}$, observe that the random element $p$ of $L^2(\mu)$ defined in Lemma~\ref{Lemma:LowerBound} has $D(p) = \rho^2$ with probability one.  Furthermore, from~\eqref{Eq:thetaa} and~\eqref{Eq:apbar}, we have with probability one that $p \in \mathcal{F}_{\xi}(\rho)$.  Since only finitely many elements of the set $\bigl\{a_{jk}:j \in \mathcal{J} \setminus \{j_0\},k \in \mathcal{K} \setminus \{k_0\}\bigr\}$ are non-zero, $\{p \in \mathcal{F}\}$ is an event, so by Lemma~\ref{Lemma:LowerBound} and the discussion immediately following it, we have
\[
	\tilde{\rho}(n,\gamma,\xi)^2 \geq \sup_{m \in \mathbb{N}} \sum_{j\in \mathcal{J} \setminus \{j_0\},k \in \mathcal{K} \setminus \{k_0\}} a_{jk}^2 = \sup_{m \in \mathbb{N}} c_m,
\]
and the result follows.
%
\end{proof}

\begin{proof}[Proof of Proposition~\ref{Thm:FourierLowerBound}]
For $m =\lceil n r^2 \rceil^{2d/(4s+d)}$, we set
\[
	d_m=\min\biggl( \frac{r^2}{\theta_{\omega(m)}^2}, \frac{(2m)^{1/2}}{n+1}\log^{1/2}\bigl(1+(1-\gamma)^2\bigr)\biggr) \asymp_{s_X,s_Y,d_X,d_Y,\gamma} \Bigl( \frac{r^d}{n^{2s}} \Bigr)^{2/(4s+d)}
\]
and
\[
	a_{\omega(\ell)}:= \left\{ \begin{array}{ll} d_m^{1/2}/m^{1/2} & \mbox{for $\ell \in [m]$} \\ 0 & \mbox{otherwise.} \end{array} \right. 
\]
The rest of this proof is dedicated to showing that, for the $p$ constructed in the statement of Lemma~\ref{Lemma:LowerBound}, we have
\[
	\mathbb{P}(p \not\in \mathcal{F}) = \mathbb{P}\biggl( \essinf_{x \in \mathcal{X},y \in \mathcal{Y}} p(x,y) < 0\biggr) < 1 - \sqrt{\frac{1+(1-\gamma)^2}{1 + 4(1-\gamma)^2}},
\]
from which the result will follow from Lemma~\ref{Lemma:LowerBound}. We define the random function
\[
	F(x,y):= 1- p(x,y) = - \sum_{j \in \mathcal{J} \setminus \{j_0\}, k \in \mathcal{K} \setminus \{k_0\}} a_{jk} \xi_{jk} p_{jk}(x,y)
\]
and aim to bound $\mathbb{P}\bigl( \esssup_{x\in \mathcal{X},y \in \mathcal{Y}} F(x,y) > 1\bigr)$. The space $\mathcal{X} \times \mathcal{Y}$ can be equipped with the pseudo-metric
\[
	\tau\bigl((x,y), (x',y') \bigr) := \biggl[\sum_{j \in \mathcal{J} \setminus \{j_0\}, k \in \mathcal{K} \setminus \{k_0\}} a_{jk}^2 \{p_{jk}(x,y) - p_{jk}(x',y')\}^2\biggr]^{1/2},
\]
which satisfies 
\[
\delta := \sup_{x\in \mathcal{X},y \in \mathcal{Y}} \tau\bigl((x,y),(x_0,y_0)\bigr) \leq 4 \biggl\{\sum_{j \in \mathcal{J} \setminus \{j_0\}, k \in \mathcal{K} \setminus \{k_0\}} a_{jk}^2\biggr\}^{1/2} = 4 d_m^{1/2}
\]
for any $(x_0,y_0) \in \mathcal{X} \times \mathcal{Y}$.  Now, for $m = (m_1,\ldots,m_{d_X}) \in \mathbb{N}_0^{d_X}$ and $x = (x_1,\ldots,x_{d_X}) \in \mathcal{X}$, we write $\langle m,x \rangle_{\mathcal{X}} := \sum_{\ell=1}^{d_X} m_\ell x_\ell$; similarly, for $m = (m_1,\ldots,m_{d_Y}) \in \mathbb{N}_0^{d_Y}$ and $y = (y_1,\ldots,y_{d_Y}) \in \mathcal{Y}$, we write $\langle m,y \rangle_{\mathcal{Y}} := \sum_{\ell=1}^{d_Y} m_\ell y_\ell$.   Then, for any $x,x' \in \mathcal{X}$ and $y,y' \in \mathcal{Y}$,
\begin{align}
\label{Eq:SmoothMetric}
	\tau\bigl((x,y),&(x',y')\bigr)^2 \leq 4\sum_{\substack{(a_X,m_X) \in \mathcal{J} \setminus \{j_0\} \\ (a_Y,m_Y) \in \mathcal{K} \setminus \{k_0\}}} \! a_{jk}^2 \bigl\{\bigl|e^{-2\pi i \langle m_X,x-x' \rangle_{\mathcal{X}}} - 1\bigr| + \bigl|e^{-2\pi i \langle m_Y,y-y' \rangle_{\mathcal{Y}}} - 1\bigr|\bigr\}^2 \nonumber \\ 
	&\leq 32 \pi^2 \sum_{\substack{(a_X,m_X) \in \mathcal{J} \setminus \{j_0\} \\ (a_Y,m_Y) \in \mathcal{K} \setminus \{k_0\}}} \hspace{-20pt} a_{jk}^2 \bigl( 1 \wedge \langle m_X, x-x' \rangle_{\mathcal{X}}^2 + 1 \wedge \langle m_Y, y-y' \rangle_{\mathcal{Y}}^2 \bigr) \nonumber \\
	& \leq 32 \pi^2 \sum_{\substack{(a_X,m_X) \in \mathcal{J} \setminus \{j_0\} \\ (a_Y,m_Y) \in \mathcal{K} \setminus \{k_0\}}}\hspace{-20pt}  a_{jk}^2 \bigl\{ (\|m_X\|_1 \|x-x'\|_\infty)^{2(s_X \wedge 1)} + (\|m_Y\|_1 \|y-y'\|_\infty)^{2(s_Y \wedge 1)} \bigr\} \nonumber \\
	& \leq 64 \pi^2 r^2 \max\bigl\{  \|x-x'\|_\infty^{2(s_X \wedge 1)}, \|y-y'\|_\infty^{2(s_Y \wedge 1)}\bigr\}.
\end{align}
For $u,v > 0$, let $H_{\infty}(u,\mathcal{X})$ and $H_\infty(v,\mathcal{Y})$ be the $u$- and $v$-metric entropies of $\mathcal{X}$ and $\mathcal{Y}$, respectively, with respect to the appropriate supremum metric; thus, for example, there exists $\mathcal{X}_N := \{x_1,\ldots,x_N\}$, where $\log N = H(u,\mathcal{X})$, such that given any $x \in \mathcal{X}$, there exists $x_{j^*} \in \mathcal{X}_N$ with $\|x-x_{j^*}\|_\infty \leq u$.  It follows from~\eqref{Eq:SmoothMetric} that, if $H(w,\mathcal{X} \times \mathcal{Y})$ is the $w$-metric entropy of $(\mathcal{X} \times \mathcal{Y},\tau)$ in the metric $\tau$, then
\begin{align*}
	H(w, \mathcal{X} \times \mathcal{Y}) &\leq H_\infty \Bigl( \Bigl( \frac{w}{8 \pi r}\Bigr)^{1/(s_X \wedge 1)}, \mathcal{X} \Bigr) + H_\infty \Bigl( \Bigl( \frac{w}{8 \pi r}\Bigr)^{1/(s_Y \wedge 1)}, \mathcal{Y} \Bigr) \\
	& \leq d_X \log\biggl( 1 + \Bigl( \frac{8 \pi r}{w} \Bigr)^{1/(s_X \wedge 1)} \biggr) + d_Y \log\biggl( 1 + \Bigl( \frac{8 \pi r}{w} \Bigr)^{1/(s_Y \wedge 1)} \biggr) \\
	& \leq \Bigl( \frac{d_X}{s_X \wedge 1} + \frac{d_Y}{s_Y \wedge 1} \Bigr) \log (1+8 \pi r/w).
\end{align*}
This choice of metric allows us to write, for any $\lambda \in \mathbb{R}, x,x' \in \mathcal{X}$ and $y,y' \in \mathcal{Y}$, 
\begin{align}
  \label{Eq:cgf}
	\log \mathbb{E}  &e^{\lambda \{F(x,y) - F(x',y')\}}  =\sum_{ \substack{ j \in \mathcal{J} \setminus \{j_0\} \\ k \in \mathcal{K} \setminus \{k_0\}}} \log \cosh \bigl( \lambda a_{jk} \{p_{jk}(x,y)-p_{jk}(x',y')\} \bigr) \nonumber \\
	& \hspace{50pt} \leq \frac{\lambda^2}{2} \sum_{ \substack{ j \in \mathcal{J} \setminus \{j_0\} \\ k \in \mathcal{K} \setminus \{k_0\}}} a_{jk}^2 \{p_{jk}(x,y)-p_{jk}(x',y')\}^2 = \frac{\lambda^2}{2} \tau\bigl((x,y), (x',y') \bigr)^2.
\end{align}
We now apply a chaining argument.  For each $t \in \mathbb{N}$, let $\delta_t := \delta 2^{-t}$, and let $\mathcal{Z}_t$ denote a $\delta_t$-net of $\mathcal{X} \times \mathcal{Y}$ with respect to the pseudo-metric $\tau$.  Let $z_0 = (x_0,y_0)$ be an arbitrary element of $\mathcal{X} \times \mathcal{Y}$ and $\mathcal{Z}_0 := \{z_0\}$.  Then, for each $t \in \mathbb{N}_0$, we can define a map $\Pi_t: \mathcal{X} \times \mathcal{Y} \rightarrow \mathcal{Z}_t$ such that $\tau\bigl(z,\Pi_t(z)\bigr) \leq \delta_t$.  Noting that $\mathbb{E}F(x_0,y_0) = 0$ and writing $F_t := F \circ \Pi_t$, we have for every $T \in \mathbb{N}$ that
\begin{align*}
  \mathbb{E}  \biggl(\esssup_{x \in \mathcal{X},y \in \mathcal{Y}} F(x,y)\biggr) &\leq \mathbb{E}\biggl(\esssup_{x \in \mathcal{X},y \in \mathcal{Y}} F_T(x,y) + \esssup_{x \in \mathcal{X},y \in \mathcal{Y}} |F(x,y) - F_T(x,y)|\biggr) \\
                                                                            &\leq \sum_{t=1}^T \mathbb{E}\biggl[\esssup_{x \in \mathcal{X},y \in \mathcal{Y}} \bigl\{F_t(x,y) - F_{t-1}(x,y)\bigr\}\biggr] \\
                                                                            &\hspace{2cm}+ \sum_{j \in \mathcal{J} \setminus \{j_0\}, k \in \mathcal{K} \setminus \{k_0\}} a_{jk}\bigl|p_{jk}(x,y) - p_{jk}\bigl(\Pi_T(x,y)\bigr)\bigr|.
\end{align*}
Now $\tau\bigl(\Pi_t(x,y),\Pi_{t-1}(x,y)\bigr) \leq 3\delta_t$ for all $x \in \mathcal{X}, y \in \mathcal{Y}$ and $t \in \mathbb{N}$.  Hence, by~\eqref{Eq:cgf} and a standard sub-Gaussian maximal inequality \citep[e.g.][Theorem~2.5]{BGM2013},
\begin{align*}
  \mathbb{E}\biggl(\esssup_{x \in \mathcal{X},y \in \mathcal{Y}} F(x,y)\biggr) &\leq 6\sum_{t=1}^T \delta_tH^{1/2}(\delta_t,\mathcal{X} \times \mathcal{Y}) + m\delta_T \\
  &\leq 12\int_0^{\delta/2} H^{1/2}(u,\mathcal{X} \times \mathcal{Y}) \, du + m\delta_T.
\end{align*}
Since this bound holds for every $T \in \mathbb{N}$, we conclude that
\begin{align*}
  \mathbb{E}\biggl(  \esssup_{x \in \mathcal{X},y \in \mathcal{Y}}  F(x,y)\biggr)& \leq 12\int_0^{\delta/2} H^{1/2}(u,\mathcal{X} \times \mathcal{Y}) \, du \\
	&\leq 96 \pi \Bigl( \frac{d_X}{s_X \wedge 1} + \frac{d_Y}{s_Y \wedge 1} \Bigr)^{1/2} r \int_0^{\frac{d_m^{1/2}}{2 \pi r}} \log^{1/2}(1+1/v) \,dv \\
                                                                     & \leq 24 d_m^{1/2} \Bigl( \frac{d_X}{s_X \wedge 1} \! + \! \frac{d_Y}{s_Y \wedge 1} \Bigr)^{1/2}\biggl\{\sqrt{\pi} + 2\sqrt{\log 2} + 2\sqrt{\log \biggl(\frac{4\pi r}{d_m^{1/2}}\biggr)}\biggr\}. 
\end{align*}
Now with $\zeta=(s_X,s_Y,d_X,d_Y,\gamma)$ we have $d_m^{1/2} \asymp_\zeta (r^d/n^{2s} )^{1/(4s+d)}$, so that $r/d_m^{1/2} \asymp_\zeta (nr^2)^{2s/(4s+d)}$ and hence there exists $c_1=c_1(\zeta) \in (0,\infty)$ such that if $(r^d/n^{2s})^{1/(4s+d)} \leq c_1/ \log^{1/2}(nr^2)$, then $\mathbb{E} \esssup_{x\in \mathcal{X},y \in \mathcal{Y}} F(x,y) \leq 1/2$.

Now, by e.g.~\citet[][Theorem~12.1]{BGM2013}, the random variable $\sup_{x\in \mathcal{X},y \in \mathcal{Y}}F(x,y)$ is sub-Gaussian with variance proxy
\[
  \sum_{j\in\mathcal{J} \setminus \{j_0\},k \in \mathcal{K}\setminus\{k_0\}} a_{jk}^2 \|p_{jk}\|_\infty^2 \leq 2d_m.
\]
By reducing $c_1 = c_1(\zeta) > 0$ if necessary, and since $nr^2 \geq 2$, we may assume that
\[
  d_m < -\frac{1}{16} \log\biggl(1 - \sqrt{\frac{1+(1-\gamma)^2}{1 + 4(1-\gamma)^2}}\biggr).
\]
Hence, by a standard sub-Gaussian tail bound \citep[e.g.][p.~25]{BGM2013}
\begin{align*}
  \mathbb{P}(p \notin \mathcal{F})  &\leq \mathbb{P}\Bigl( \esssup_{x\in \mathcal{X},y\in \mathcal{Y}} F(x,y) - \mathbb{E} \esssup_{x\in \mathcal{X},y\in \mathcal{Y}} F(x,y) > 1/2\Bigr) \\
  &\leq e^{-1/(16d_m)} < 1 - \sqrt{\frac{1+(1-\gamma)^2}{1 + 4(1-\gamma)^2}},
\end{align*}
as required.
\end{proof}

\medskip

\textbf{Acknowledgements:} We are very grateful to the anonymous reviewers, whose constructive feedback helped to improve the paper.  We would also like to thank Ilmun Kim for bringing the work of \citet{Song2012} to our attention; this inspired the computational improvements discussed in Section~\ref{Sec:CompTrick}.

\newpage

\setcounter{section}{0}
\setcounter{equation}{0}
\setcounter{thm}{0}
\def\theequation{S\arabic{equation}}
\def\thesection{S.\arabic{section}}
\def\thethm{S\arabic{thm}}
\def\thedefn{S\arabic{thm}}
\def\theexample{S\arabic{thm}}
\def\theremark{S\arabic{thm}}

\begin{frontmatter}

\title{Supplementary material for `Optimal rates for independence testing via $U$-statistic permutation tests'}
\runtitle{Independence testing via permutation tests}

\begin{aug}
\author[A]{\fnms{Thomas} B. \snm{Berrett}\thanksref{t1}\ead[label=e1]{tom.berrett@warwick.ac.uk}}
\author[B]{\fnms{Ioannis} \snm{Kontoyiannis}\ead[label=e2]{yiannis@maths.cam.ac.uk}} \\
\and
\author[C]{\fnms{Richard} J. \snm{Samworth}\thanksref{t2}\ead[label=e3]{r.samworth@statslab.cam.ac.uk}}
\thankstext{t1}{Financial support from the French National Research Agency (ANR) under the grants Labex Ecodec (ANR-11-LABEX-0047 and ANR-17-CE40-0003.}
\thankstext{t2}{Research supported by Engineering and Physical Sciences Reseach Council (EPSRC) Programme grant EP/N031938/1 and EPSRC Fellowship EP/P031447/1.}

\runauthor{T. B. Berrett, I. Kontoyiannis and R. J. Samworth}


  \address[A]{Department of Statistics, University of Warwick, Coventry, CV4 7AL, United Kingdom\\ 
          \printead{e1}}

\address[B]{Statistical Laboratory, Centre for Mathematical Sciences, Wilberforce Road, Cambridge, CB3 0WB, United Kingdom\\ \printead{e2}, \printead{e3}}
        




\end{aug}







\end{frontmatter}

\section{Remaining proof from Section~\ref{Sec:MainResults}}

\begin{proof}[Proof of Corollary~\ref{Cor:InfDim}]
For $t \geq1$, we have that
\[
	M_{s_X,s_Y}(t):=| \{(j,k) \in \mathcal{J} \times \mathcal{K} : \theta_{jk} \leq t \}| = 4 M(s_X^{-2} \log^2 t) M(s_Y^{-2} \log^2 t).
\]
Thus,
\begin{align*}
	m_0(t) > m \iff m^{1/2} \theta_{\omega(m)}^2 \leq t  \iff M_{s_X,s_Y}(t^{1/2}/m^{1/4}) \geq m.
\end{align*}
If follows that $m_0(t) = \min\{ m \in \mathbb{N} : M_{s_X,s_Y}(t^{1/2}/m^{1/4}) < m \}$, and the result then follows from Theorem~\ref{Thm:UpperBound} in the main text.

We now turn to the proof of the second part of the result. For $t \geq 16$ set $L:=\lfloor t^{1/3} \rfloor$. We then have
\begin{align*}
	&\biggl| \sum_{\ell=1}^\infty  \log( 1 + \lfloor t/\ell^2 \rfloor) - c_0 t^\frac{1}{2} \biggr| \\
	& = \biggl| \sum_{\ell=1}^{\lfloor (t/L)^\frac{1}{2} \rfloor} \log(1+\lfloor t/\ell^2 \rfloor ) \! + \! \sum_{u=1}^{L-1} \Bigl( \Bigl\lfloor \Bigl( \frac{t}{u} \Bigr)^\frac{1}{2} \Bigr\rfloor \! -  \! \Bigl\lfloor \Bigl( \frac{t}{u+1} \Bigr)^\frac{1}{2} \Bigr\rfloor \Bigr) \log ( 1+u) - c_0 t^\frac{1}{2} \biggr| \\
	& \leq \frac{t^\frac{1}{2} \log(1+t)}{L^\frac{1}{2}} + 2L \log L + \biggl| \sum_{u=1}^{L-1} \Bigl\{ \Bigl( \frac{t}{u} \Bigr)^\frac{1}{2} - \Bigl( \frac{t}{u+1} \Bigr)^\frac{1}{2} \Bigr\} \log(1+u) - c_0 t^\frac{1}{2} \biggr| \\
	& = \frac{t^\frac{1}{2} \log(1+t)}{L^\frac{1}{2}} + 2 L \log L + t^\frac{1}{2} \sum_{u=L}^\infty \{u^{-1/2} - (u+1)^{-1/2}\} \log(1+u).
\end{align*}
Thus, given $\delta>0$, there exists $t_0=t_0(\delta)>0$ such that 
\[
	\bigl| t^{-1/2} \log\bigl( 1 + M(t) \bigr) - c_0 \bigr| \leq \delta
\]
whenever $t \geq t_0$.

Let $\epsilon \in (0,4s)$ be given, and consider $m=t^\frac{2c_0 + \epsilon}{2s+c_0}$ for $t$ sufficiently large that $t m^{-1/2} = t^\frac{2s-\epsilon/2}{2s+c_0} \geq \exp(2(s_X \vee s_Y)t_0(\epsilon/2)^{1/2})$. For such $t$ we have that
\begin{align*}
	&\frac{4}{m} M \biggl( \frac{\log^2(m^{-1/2}t)}{4s_X^2} \biggr) M \biggl( \frac{\log^2(m^{-1/2}t)}{4s_Y^2} \biggr) \\
	&\leq \frac{4}{m} \Bigl\{ \Bigl( \frac{t}{m^{1/2}} \Bigr)^\frac{c_0+\epsilon/2}{2s_X} -1 \Bigr\}\Bigl\{ \Bigl( \frac{t}{m^{1/2}} \Bigr)^\frac{c_0+\epsilon/2}{2s_Y} -1 \Bigr\} \leq \frac{4}{m} \Bigl( \frac{t}{m^{1/2}} \Bigr)^\frac{c_0+\epsilon/2}{s} <1
\end{align*}
for $t$ sufficiently large. Thus, when $t$ is large enough we have $m_{0,s_X,s_Y}(t) \leq t^{\frac{2c_0+\epsilon}{2s+c_0}}$. An analogous argument shows that we also have $m_{0,s_X,s_Y}(t)  \geq t^{\frac{2c_0-\epsilon}{2s+c_0}}$ when $t$ is sufficiently large, as required.
\end{proof}

\section{Remaining proof from Section~\ref{Sec:LowerBounds}}

\begin{proof}[Proof of Proposition~\ref{Prop:InfDimLowerBound}]
Take
\begin{align*}
  m:=m_0(nr^2)-1 = m_{0,s_X,s_Y}(nr^2) - 1.
\end{align*}
For this choice of $m$, set
\[
	d_m:= \frac{(2m)^{1/2}}{n+1} \log^{1/2}\bigl(1+(1-\gamma)^2/4\bigr),
\]
which by construction is bounded above by $r^2/\theta_{\omega(m)}^2$.  Further, set
\[
	a_{\omega(\ell)}:= \left\{ \begin{array}{ll} d_m^{1/2}/m^{1/2} & \text{for } \ell \in [m] \\ 0 & \text{otherwise.} \end{array} \right. 
\]
As in the proof of Proposition~\ref{Thm:FourierLowerBound} in the main text, our aim now is to give an upper bound for $\mathbb{P}\bigl( \esssup_{x\in \mathcal{X},y \in \mathcal{Y}} F(x,y) > 1\bigr)$, where
\[
	F(x,y) := - \sum_{j \in \mathcal{J} \setminus \{j_0\}, k \in \mathcal{K} \setminus \{k_0\}} a_{jk} \xi_{jk} p_{jk}(x,y).
\]
Again we take
\[
		\tau\bigl((x,y), (x',y') \bigr)^2 := \sum_{j \in \mathcal{J} \setminus \{j_0\}, k \in \mathcal{K} \setminus \{k_0\}} a_{jk}^2 \{p_{jk}(x,y) - p_{jk}(x',y')\}^2.
\]
The main difference with the proof of Proposition~\ref{Thm:FourierLowerBound} in the main text is in how we bound the metric entropy of the space $(\mathcal{X} \times \mathcal{Y}, \tau)$. For $m \in \mathbb{N}_0^{< \infty}$ and $x,x' \in \mathcal{X}$ we have
\begin{align*}
	\biggl| \prod_{\ell=1}^\infty e^{-2 \pi i m_\ell x_\ell} - \prod_{\ell=1}^\infty e^{-2 \pi i m_\ell x_{\ell}'} \biggr| \leq 2 \pi \sum_{\ell=1}^\infty |m_\ell| |x_\ell-x_\ell'| \leq 2 \pi |m| \sum_{\ell=1}^\infty \frac{|x_\ell-x_\ell'|}{\ell^2}.
\end{align*}
Thus, if we define the norm $\|\cdot\|$ on $\mathcal{X}=\mathcal{Y}=[0,1]^{\mathbb{N}}$ by $\|(x_1,x_2,\ldots)\|=\sum_{\ell=1}^\infty \ell^{-2} | x_\ell|$, then we may write, for any $x,x' \in \mathcal{X}$ and $y,y' \in \mathcal{Y}$ that
\begin{align*}
	\tau\bigl((x,y), (x',y') \bigr) &\leq 4\pi \biggl[ \sum_{j \in \mathcal{J} \setminus \{j_0\}, k \in \mathcal{K} \setminus \{k_0\}} a_{jk}^2 \bigl\{|j| \|x-x'\| +  |k| \|y-y'\| \bigr\}^2 \biggr]^{1/2} \\	
	& \leq 8 \pi ( \|x-x'\| \vee \|y-y'\|) \biggl[\sum_{j \in \mathcal{J} \setminus \{j_0\}, k \in \mathcal{K} \setminus \{k_0\}} a_{jk}^2 ( |j|^2 \vee |k|^2 ) \biggr]^{1/2} \\
  &\leq \frac{8 \pi r}{(s_X \wedge s_Y)^2} ( \|x-x'\| \vee \|y-y'\|).
\end{align*}
Here, in the final inequality, we have used the fact that $x^2 \leq e^x$ for $x \geq 0$.  We now bound the metric entropy of the space $(\mathcal{X}, \|\cdot \|)$. For $L \in \mathbb{N}$ define
\[
	\mathcal{X}^{(L)}:= \bigl\{ x = (x_1,x_2,\ldots) \in \mathcal{X}: x_\ell=0 \text{ for all } \ell \geq L+1 \bigr\}.
\]
For any $x = (x_1,x_2,\ldots) \in \mathcal{X}, L \in \mathbb{N}$ we have $\|x-(x_1,\ldots,x_L,0,0,\ldots)\| \leq \sum_{\ell=L+1}^\infty \ell^{-2} \leq 1/L$. Given $\delta >0$, set $L=\lfloor 2/ \delta \rfloor$ and define
\[
  \mathcal{I} := \{1,\ldots,\lfloor (\delta/4)^{-2} \rfloor\} \times \{1,\ldots,\lfloor (\delta/4)^{-2}/4 \rfloor\} \times \ldots \times \{1,\ldots,  \lfloor  (\delta/4)^{-2} /L^2 \rfloor \}.
\]
For every $i \in \mathcal{I}$, set
\[
	x^{(i)} := \bigl( i_1 (\delta/4)^2, 4i_2 (\delta/4)^2, \ldots, L^2 i_L (\delta/4)^2,0,0,\ldots\bigr) \in \mathcal{X}^{(L)}.
\]
We now show that the family $\{x^{(i)}\}_{i \in \mathcal{I}}$ is a $\delta$-covering set of $\mathcal{X}$. Let $x = (x_1,x_2,\ldots) \in \mathcal{X}$ be given, and for each $\ell=1,\ldots,L$, define the quantity $i_\ell^*:=\mathrm{argmin}_{i \in \{1,\ldots,\lfloor (\ell\delta/4)^{-2} \rfloor \}} |x_\ell - \ell^2 i (\delta/4)^2|$. Then, when $\delta \leq 1/2$ so that $ L \geq 8/(5 \delta)$ we have
\begin{align*}
	\|x-x^{(i^*)}\| &\leq 1/L + \|(x_1,\ldots,x_L,0,0,\ldots) - x^{(i^*)} \| \\
	& \leq 1/L + L \delta^2 /16 \leq 5 \delta /8 + \delta/8 < \delta.
\end{align*}
Hence,
\[
	H(\delta, \mathcal{X}) \leq \sum_{\ell=1}^L \log \bigl(1 + (\ell \delta /4)^{-2} \bigr) \leq \int_0^{2/\delta} \log \bigl(1+ (x \delta /4)^{-2} \bigr) \,dx \leq \frac{8}{\delta}. 
\]
It follows that
\begin{align*}
	H(u,  \mathcal{X} \times \mathcal{Y}) \leq 2 H \Bigl( \frac{u(s_X \wedge s_Y)^2}{8 \pi r}, \mathcal{X} \Bigr) \leq \frac{64 \pi r}{ u (s_X \wedge s_Y)^2}.
\end{align*}
Write $\delta = 2^{3/2} d_m^{1/2}$ as in the proof of Proposition~\ref{Thm:FourierLowerBound} in the main text. Then after reducing $\epsilon > 0$ from the statement of Proposition~\ref{Prop:InfDimLowerBound} in the main text if necessary so that $\epsilon \in \bigl(0,4s/(s+c_0)\bigr)$, by the second part of Corollary~\ref{Cor:InfDim} in the main text, there exists $C'=C'\bigl(s_X,s_Y,\epsilon(s+c_0)\bigr)$ such that when $nr^2 \geq C'$ we have
\begin{align*}
	(r \delta)^2 &= 8d_m r^2 \leq \frac{8 r^2 m^{1/2}}{n} \leq \frac{8r^2}{n} (nr^2)^\frac{c_0+\epsilon(s+c_0)/2}{2s+c_0} \\
	&= 8 r^\frac{(4+\epsilon)(s+c_0)}{2s+c_0} n^{-\frac{2s-\epsilon(s+c_0)/2}{2s+c_0}} \leq 8 n^{ \frac{1}{2} (\frac{s}{s+c_0} - \epsilon) \frac{(4+\epsilon)(s+c_0)}{2s+c_0} -\frac{2s-\epsilon(s+c_0)/2}{2s+c_0}} \!\! \leq 8 n^{-\frac{\epsilon(s+c_0)}{2s+c_0}}.
\end{align*}
Hence, by the chaining argument in the proof of Proposition~\ref{Thm:FourierLowerBound} in the main text and by increasing $C' = C'\bigl(s_X,s_Y,\epsilon\bigr) > 0$ if necessary, we have when $\min(n,nr^2) \geq C'$ that
\begin{align*}
  \mathbb{E} \esssup_{x\in \mathcal{X},y\in \mathcal{Y}} F(x,y) &\leq \frac{12 \cdot 8 \pi^{1/2}}{s_X \wedge s_Y} \int_0^{\delta/2} \Bigl(\frac{r}{u} \Bigr)^{1/2} \,du= \frac{96 \cdot 2^{1/2} \pi^{1/2}}{s_X \wedge s_Y} (r\delta)^{1/2} \leq \frac{1}{2}.
\end{align*}
From the second part of Corollary~\ref{Cor:InfDim} in the main text, we see that by still further increasing $C' = C'\bigl(s_X,s_Y,\epsilon\bigr) > 0$ if necessary, we may assume that when $\min(n,nr^2) \geq C'$ we have $d_m < -\frac{1}{16} \log\bigl( 1 - \sqrt{\frac{1+(1-\gamma)^2}{1 + 4(1-\gamma)^2}}\bigr)$. Hence, as at the end of the proof of Proposition~\ref{Thm:FourierLowerBound} in the main text, we have
\begin{align*}
  \mathbb{P}\Bigl( \esssup_{x\in \mathcal{X},y\in \mathcal{Y}} F(x,y) > 1\Bigr) &\leq \mathbb{P}\Bigl( \esssup_{x\in \mathcal{X},y\in \mathcal{Y}}F(x,y) - \mathbb{E} \esssup_{x\in \mathcal{X},y\in \mathcal{Y}} F(x,y) \geq 1/2\Bigr) \\
  &\leq e^{-1/(16d_m)} < 1 - \sqrt{\frac{1+(1-\gamma)^2}{1 + 4(1-\gamma)^2}}.
\end{align*}
Finally, by Lemma~\ref{Lemma:LowerBound} in the main text, we conclude that, when $nr^2 \geq C'=C'(s_X,s_Y,\epsilon)$,
\begin{align*}
  \rho^*(n,\alpha,\beta,\xi)^2 \geq d_m &= \frac{(2m)^{1/2}}{n+1} \log^{1/2}\biggl(1+\frac{(1-\gamma)^2}{4}\biggr) \\
  &\geq \frac{m_0(nr^2)^{1/2}}{n+1} \log^{1/2}\biggl(1+\frac{(1-\gamma)^2}{4}\biggr),
\end{align*}
as required.
\end{proof}

\section{Remaining proofs from Section~\ref{Sec:PowerFunction}}

\begin{proof}[Proof of Theorem~\ref{Thm:PowerFunction}]
As in the proof of Theorem~\ref{Thm:UpperBound}, we write $D, a_{jk}, a_{j\bullet}, a_{\bullet k}$ in place of $D(f), a_{jk}(f), a_{j\bullet}(f), a_{\bullet k}(f)$ respectively.

The first step of the proof is to show that $\hat{D}_n,\hat{D}_n^{(1)},\ldots,\hat{D}_n^{(B)}$ can be approximated by appropriate second-order $U$-statistics that are degenerate.  We then apply Lemma~\ref{Lemma:JointNormality} to establish that they can be jointly be approximated in the $d_{\mathcal{G}_B}$ metric by a multivariate Gaussian distribution.

For $j \in \mathcal{J}, k \in \mathcal{K}, x,x_1,x_2 \in \mathcal{X}$ and $y,y_1,y_2 \in \mathcal{Y}$, it will be convenient to write $\tilde{p}_{jk}(x,y):=p_{jk}(x,y)-a_{j\bullet}p_k^Y(y)-a_{\bullet k}p_j^X(x) - a_{jk} + 2a_{j\bullet}a_{\bullet k}$ and
\begin{align*}
	h_2\bigl( (x_1,y_1),(x_2,y_2) \bigr):= \sum_{(j,k)\in \mathcal{M}} \tilde{p}_{jk}(x_1,y_1) \tilde{p}_{jk}(x_2,y_2).
\end{align*}
Recalling the definition of the symmetrised version $\bar{h}$ of $h$ in~\eqref{Eq:barh}, we begin by calculating
\begin{align}
\label{Eq:SecondOrderU}
	 \bar{g}&\bigl((x_1,y_1)(x_2,y_2) \bigr) := \mathbb{E} \bigl\{ \bar{h} \bigl((x_1,y_1),(x_2,y_2), (X_1,Y_1), (X_2,Y_2) \bigr) \bigr\} \nonumber \\
	&= \frac{1}{6} \sum_{(j,k)\in \mathcal{M}} \bigl\{p_{jk}(x_1,y_1)-a_{j\bullet}p_k^Y(y_1)-a_{\bullet k}p_j^X(x_1) - a_{jk} + 2a_{j\bullet}a_{\bullet k} \bigr\} \nonumber \\
	& \hspace{50pt} \times \bigl\{p_{jk}(x_2,y_2)-a_{j\bullet}p_k^Y(y_2)-a_{\bullet k}p_j^X(x_2) - a_{jk} + 2a_{j\bullet}a_{\bullet k} \bigr\} \nonumber \\
	& \hspace{10pt} + \frac{1}{6} \! \! \sum_{(j,k)\in \mathcal{M}} \! \! \! \! (a_{jk} -a_{j\bullet}a_{\bullet k})\bigl\{ 3p_{jk}(x_1,y_1) \!+\! 3 p_{jk}(x_2,y_2) \!-\! p_{jk}(x_1,y_2) \!-\! p_{jk}(x_2,y_1) \nonumber \\ 
	& \hspace{40pt}- 2p_j^X(x_1)a_{\bullet k} - 2p_j^X(x_2)a_{\bullet k} - 2p_k^Y(y_1)a_{j\bullet} - 2 p_k^Y(y_2)a_{j\bullet} + 4a_{j\bullet}a_{\bullet k} \bigr\} \nonumber \\
	&= \frac{1}{6} h_2 \bigl( (x_1,y_1), (x_2,y_2) \bigr) + \frac{2}{3} \sum_{(j,k) \in \mathcal{M}} (a_{jk} - a_{j \bullet} a_{\bullet k})^2  \nonumber \\
	& \hspace{10pt}+ \frac{1}{6} \!\!\! \sum_{(j,k) \in \mathcal{M}} \!\!\! (a_{jk} \!-\! a_{j \bullet} a_{\bullet k}) \bigl\{ 3 \tilde{p}_{jk}(x_1,y_1) \!+\! 3 \tilde{p}_{jk}(x_2,y_2) \!-\! \tilde{p}_{jk}(x_1,y_2)\! -\! \tilde{p}_{jk}(x_2,y_1) \bigr\}.
\end{align}
Moreover, we let
\begin{align*}
	\tilde{g}\bigl( &(x_1,y_1), \ldots, (x_4,y_4) \bigr) \\
	&:= \bar{h}  \bigl( (x_1,y_1), \ldots , (x_4,y_4) \bigr) - \frac{1}{4!} \sum_{\sigma \in \mathcal{S}_4} h_2 \bigl( (x_{\sigma(1)},y_{\sigma(1)}),(x_{\sigma(2)},y_{\sigma(2)}) \bigr).
\end{align*}
Then, by \citet[][Lemma~A, p.~183]{Serfling1980}, and calculations very similar to those in~\eqref{Eq:zeta1bound} and~\eqref{Eq:Zeta3} in the main text, we have that
\begin{align}
\label{Eq:DnhatApprox}
	\mathrm{Var} \biggl( & \hat{D}_n - \frac{1}{\binom{n}{2}} \sum_{i_1 < i_2} h_2 \bigl((X_{i_1},Y_{i_1}), (X_{i_2},Y_{i_2}) \bigr) \biggr) \nonumber \\
	&= \mathrm{Var} \biggl( \frac{1}{\binom{n}{4}} \sum_{i_1< i_2<i_3<i_4} \tilde{g} \bigl((X_{i_1},Y_{i_1}), \ldots ,(X_{i_4},Y_{i_4}) \bigr) \biggr) \nonumber \\
	& \lesssim \frac{1}{n} \mathrm{Var} \Bigl( \mathbb{E}\bigl\{ \tilde{g} \bigl((X_1,Y_1), \ldots,(X_4,Y_4) \bigr) | X_1,Y_1 \bigr\} \Bigr) \nonumber \\
	& \hspace{50pt} + \frac{1}{n^2}  \mathrm{Var} \Bigl( \mathbb{E}\bigl\{ \tilde{g} \bigl((X_1,Y_1), \ldots,(X_4,Y_4) \bigr) | X_1,Y_1,X_2,Y_2 \bigr\} \Bigr) \nonumber \\
	& \hspace{50pt} + \frac{1}{n^3} \mathrm{Var} \Bigl( \tilde{g} \bigl((X_1,Y_1), \ldots,(X_4,Y_4) \bigr) \Bigr)\nonumber \\
	&= \frac{1}{n}  \mathrm{Var} \biggl( \frac{1}{2} \sum_{(j,k) \in \mathcal{M}} (a_{jk} - a_{j \bullet} a_{\bullet k}) \tilde{p}_{jk}(X_1,Y_1) \biggr) \nonumber \\
	& \hspace{50pt} + \frac{1}{n^2} \mathrm{Var} \biggl( \bar{g} \bigl( (X_1,Y_1), (X_2,Y_2) \bigr) - \frac{1}{6} h_2 \bigl( (X_1,Y_1), (X_2,Y_2) \bigr) \biggr) \nonumber \\
	& \hspace{50pt} +  \frac{1}{n^3} \mathrm{Var} \Bigl( \tilde{g} \bigl((X_1,Y_1), \ldots,(X_4,Y_4) \bigr) \Bigr) \nonumber \\
	& \lesssim \frac{A^2}{n} \sum_{(j,k) \in \mathcal{M}} (a_{jk} - a_{j \bullet}a_{\bullet k})^2 +  \frac{A^4M^2}{n^3} \lesssim \frac{A^2 M \Delta_f}{n^2} + \frac{A^4 M^2}{n^3}.
\end{align}
Having bounded the difference between $\hat{D}_n$ and an appropriate second-order degenerate $U$-statistic, we now approximate the second moment of $h_2\bigl((X_1,Y_1),(X_2,Y_2)\bigr)$, so that we may standardise this $U$-statistic.  For this we will first recall a basic fact about the trigonometric basis.  Extending the definitions of $p_{a,m}^X = p_{a,m}^Y$ in~\eqref{Eq:Fourier} in the main text to hold for all $a,m \in \mathbb{Z}$, we have for $a_1,a_2 \in \{0,1\}$, $m_1,m_2 \in \mathbb{Z}$, $Z \in \{X,Y\}$ and $x \in [0,1]$ that
\begin{align*}
  p_{a,m_1}^Z(x)p_{a,m_2}^Z(x) = \frac{1}{\sqrt{2}}\bigl\{p_{a_1+a_2,m_1+m_2}^Z(x) +p_{a_1-a_2,m_1-m_2}^Z(x)\bigr\}. 
\end{align*}
Henceforth, when there is no confusion, we will write $(X,Y)$ for a random variable with density $f$. We can now see that
\begin{align*}
	&\sum_{(j,k) \in \mathcal{M}} \sum_{(j',k') \in \mathcal{M}} \mathbb{E}^2\bigl\{p_{jk}(X,Y) p_{j'k'}(X,Y) \bigr\}  \\
	&= \sum_{m_1,\ldots,m_4=1}^M \sum_{a_1,\ldots,a_4=0}^1 \mathbb{E}^2\bigl\{p_{a_1,m_1}^X(X) p_{a_2,m_2}^X(X) p_{a_3,m_3}^Y(Y) p_{a_4,m_4}^Y(Y)\bigr\}  \\
	&= \sum_{m_1,\ldots,m_4=1}^M \sum_{a_1,\ldots,a_4=0}^1 \mathbb{E}^2\bigl\{ p_{a_1,m_1+(-1)^{a_2}m_2}^X(X) p_{a_3,m_3+(-1)^{a_4}m_4}^Y(Y)\bigr\}  \\
	&= \sum_{m_1,m_2=0}^{2M} \sum_{a_1,a_2=0}^1 (2M+1-m_1)(2M+1-m_2) a_{(a_1,m_1)(a_2,m_2)}^2 + O(MA)\\
	& = \sigma_{M,X}^2 \sigma_{M,Y}^2 + O(MA + M^2A^2D^{1/2}).
\end{align*}
Similarly, we have for example that
\begin{align*}
	\sum_{(j,k) \in \mathcal{M}} &\sum_{(j',k') \in \mathcal{M}}  \mathbb{E}^2\bigl\{ p_{jk}(X,Y) p_{j'}^X(X)a_{\bullet k'}\bigr\} \\
	& \leq A \sum_{m_1,m_2,m_3=1}^M \sum_{a_1,a_2,a_3=0}^1 \mathbb{E}^2\bigl\{ p_{a_1,m_1+(-1)^{a_2}m_2}^X(X) p_{a_3,m_3}^Y(Y)\bigr\} \\
	& \leq 2MA \sum_{m_1=1}^{2M} \sum_{m_2=1}^M \sum_{a_1,a_2=0}^1 \mathbb{E}^2\bigl\{p_{a_1,m_1}^X(X)p_{a_2,m_2}^Y(Y)\bigr\} \leq 2M A^2
\end{align*}
and $\sum_{(j,k),(j'k') \in \mathcal{M}} a_{j\bullet}^2 a_{\bullet k}^2 a_{j'\bullet}^2 a_{\bullet k'}^2 \leq A^4$. Hence, we can check that
\begin{align}
  \label{Eq:h2SecondMoment}
	\mathbb{E}\bigl\{ h_2^2 \bigl( (X_1,Y_1), (X_2,Y_2) \bigr) \bigr\} &= \sum_{(j,k) \in \mathcal{M}}  \sum_{(j',k') \in \mathcal{M}} \mathbb{E}^2\bigl\{ \tilde{p}_{jk}(X,Y) \tilde{p}_{j'k'}(X,Y) \bigr\} \nonumber \\
  &= \sigma_{M,X}^2\sigma_{M,Y}^2 + O(MA^4 + M^2A^2D^{1/2}).
\end{align}
Our calculations so far allow us to bound two of the terms that will appear in the bound when we apply Lemma~\ref{Lemma:JointNormality} below in our context, namely $\frac{\mathbb{E}\{h_2^4((X_1,Y_1),(X_2,Y_2))\}}{n \sigma_{M,X}^4 \sigma_{M,Y}^4}$ and $\frac{\mathbb{E}\{g_2^2((X_1,Y_1),(X_2,Y_2)) \}}{\sigma_{M,X}^4\sigma_{M,Y}^4}$, where
\begin{align*}
	g_2\bigl( &(x_1,y_1), (x_2,y_2) \bigr) := \mathbb{E}\bigl\{ h_2 \bigl((x_1,y_1),(X,Y) \bigr) h_2 \bigl((x_2,y_2),(X,Y) \bigr) \bigr\} \\
	&= \sum_{(j,k) \in \mathcal{M}} \sum_{(j',k') \in \mathcal{M}} \tilde{p}_{jk}(x_1,y_1) \tilde{p}_{j'k'}(x_2,y_2) \mathbb{E}\bigl\{\tilde{p}_{jk}(X,Y) \tilde{p}_{j'k'}(X,Y)\bigr\}.
\end{align*}
Specifically, we note that
\begin{align}
\label{Eq:g2expansion}
	& \mathbb{E}\bigl\{ g_2^2 \bigl((X_1,Y_1),(X_2,Y_2) \bigr) \bigr\} \nonumber \\
	&= \sum_{(j_1,k_1), \ldots, (j_4,k_4) \in \mathcal{M}} \mathbb{E}\bigl\{ \tilde{p}_{j_1k_1}(X,Y) \tilde{p}_{j_2k_2}(X,Y)\bigr\} \mathbb{E}\bigl\{ \tilde{p}_{j_1k_1}(X,Y) \tilde{p}_{j_3k_3}(X,Y)\bigr\} \nonumber \\
	& \hspace{50pt} \times \mathbb{E}\bigl\{ \tilde{p}_{j_2k_2}(X,Y) \tilde{p}_{j_4k_4}(X,Y)\bigr\} \mathbb{E}\bigl\{ \tilde{p}_{j_3k_3}(X,Y) \tilde{p}_{j_4k_4}(X,Y)\bigr\}.
\end{align}
Now, for $a \in \{0,1\}$, $m_1,m_2 \in \{1,\ldots,M\}$ and $Z \in \{X,Y\}$, define the shorthand
\[
  q_{a,m_1,m_2}^Z := p_{a,m_1+m_2}^Z - p_{a,m_1-m_2}^Z,
\]
and write $\mathcal{M}_2:= \bigl(\{0,1\} \times \{-2M,-2M+1,\ldots,2M\}\bigr)^2$.  Then
\begin{align*}
	&\biggl| \sum_{(j_1,k_1), \ldots, (j_4,k_4) \in \mathcal{M}} \mathbb{E}\{ p_{j_1k_1}(X,Y) p_{j_2k_2}(X,Y)\} \mathbb{E}\{ p_{j_1k_1}(X,Y) p_{j_3k_3}(X,Y)\} \\
	& \hspace{75pt} \times \mathbb{E}\{ p_{j_2k_2}(X,Y) p_{j_4k_4}(X,Y)\} \mathbb{E}\{p_{j_3k_3}(X,Y) p_{j_4k_4}(X,Y)\} \biggr| \\
	& \lesssim \!\!\!\! \sum_{\substack{m_1,\ldots,m_4=1 \\ n_1,\ldots,n_4=1}}^M \sum_{\substack{a_1,\ldots,a_4=0 \\ b_1,\ldots,b_4=0}}^1 \!\!\!\!\!\! \bigl| \mathbb{E}\{ q_{a_1,m_1,m_2}^X(X)q_{b_1,n_1,n_2}^Y(Y)\}\mathbb{E}\{ q_{a_2,m_1,m_3}^X(X)q_{b_2,n_1,n_3}^Y(Y)\} \\
	& \hspace{75pt} \times  \mathbb{E}\{ q_{a_3,m_2,m_4}^X(X)q_{b_3,n_2,n_4}^Y(Y)\} \mathbb{E}\{ q_{a_4,m_3,m_4}^X(X)q_{b_4,n_3,n_4}^Y(Y)\} \bigr| \\
	& \lesssim \! A \!\!\!\!\!\!\!\! \sum_{\substack{m_1,m_2,m_3=1 \\ n_1,n_2,n_3=1}}^M \sum_{\substack{a_1,a_2=0 \\ b_1,b_2=0}}^1 \!\!\! \bigl| \mathbb{E}\{q_{a_1,m_1,m_2}^X(X)q_{b_1,n_1,n_2}^Y(Y)\}\mathbb{E}\{ q_{a_2,m_1,m_3}^X(X)q_{b_2,n_1,n_3}^Y(Y)\} \bigr| \\
	& \lesssim M^2 A \!\!\!\! \sum_{\substack{m_1,m_2=-2M \\ n_1,n_2=-2M}}^{2M} \sum_{\substack{a_1,a_2=0 \\ b_1,b_2=0}}^1 \bigl| \mathbb{E}\{p_{a_1,m_1}^X(X) p_{b_1,n_1}^Y(Y)\} \mathbb{E}\{p_{a_2,m_2}^X(X_1) p_{b_2,n_2}^Y(Y)\} \bigr| \\
	& = M^2 A \biggl( \sum_{(j,k) \in \mathcal{M}_2} |a_{jk}| \biggr)^2 \leq M^2 A \biggl(  \sum_{(j,k) \in \mathcal{M}_2} |a_{j\bullet}a_{\bullet k}| + 16M D^{1/2} \biggr)^2.
\end{align*}
The other terms in the expansion of~\eqref{Eq:g2expansion} can be bounded similarly: in particular, by Parseval's identity,
\[
\sum_{(j_1,k_1), \ldots, (j_4,k_4) \in \mathcal{M}} \!\!\!\!\!\!\!\!\!\!\! a_{j_1 \bullet}^2 a_{\bullet k_1}^2 a_{j_2 \bullet}^2 a_{\bullet k_2}^2 a_{j_3 \bullet}^2 a_{\bullet k_3}^2 a_{j_4 \bullet}^2 a_{\bullet k_4}^2 \leq \|f_X\|_{L^2(\mu_X)}^4\|f_Y\|_{L^2(\mu_Y)}^4 \leq A^8. 
\]
It follows that
\begin{equation}
\label{Eq:SecondMomentg}
	\mathbb{E}\bigl\{ g_2^2 \bigl((X_1,Y_1),(X_2,Y_2) \bigr) \bigr\} \lesssim A^8 M^2 \biggl(  \sum_{(j,k) \in \mathcal{M}_2} |a_{j\bullet}a_{\bullet k}| \biggr)^2	 + A^8 M^4 D.
\end{equation}
The fourth moment of $h_2$ can similarly be bounded by writing
\begin{align}
\label{Eq:FourthMomenth}
  \mathbb{E}\bigl\{ &h_2^4 \bigl( (X_1,Y_1), (X_2,Y_2) \bigr) \bigr\} \leq A^2 \int_0^1 \int_0^1 h_2^4\bigl((x_1,y_1),(x_2,y_2)\bigr) \, dx \, dy \nonumber \\
	&= A^2 \sum_{(j_1,k_1),\ldots,(j_4,k_4) \in \mathcal{M}} \biggl( \int_0^1 \int_0^1 \tilde{p}_{j_1k_1}(x,y) \ldots \tilde{p}_{j_4k_4}(x,y) \,dx \,dy \biggr)^2 \nonumber \\
	& \lesssim A^{10} \sum_{j_1,k_1, \ldots,j_4,k_4=- M}^M \mathbbm{1}_{\{j_1+j_2+j_3+j_4=0\}} \mathbbm{1}_{\{k_1+k_2+k_3+k_4=0\}} \lesssim M^6 A^{10}.
\end{align}

The next step is to study the behaviour of the permuted statistics $\hat{D}_n^{(1)},\ldots,\hat{D}_n^{(B)}$. Many of the calculations are very similar to those already carried out in our study of $\hat{D}_n$. First, for $(j,k) \in \mathcal{J}$, define
\[
	\tilde{p}_{jk}^{(1)}(x,y):= p_{jk}(x,y)-a_{j\bullet}p_k^Y(y) - a_{\bullet k}p_j^X(x) + a_{j\bullet}a_{\bullet k}.
\]
Further, define
\[
	h_2^{(1)} \bigl((x_1,y_1),(x_2,y_2) \bigr) := \sum_{(j,k) \in \mathcal{M}} \tilde{p}_{jk}^{(1)}(x_1,y_1) \tilde{p}_{jk}^{(1)}(x_2,y_2).
\]
and
\begin{align*}
	&\tilde{g}^{(1)}\bigl( (x_1,y_1),\ldots,(x_4,y_4) \bigr) \\
	& := \bar{h} \bigl( (x_1,y_1),\ldots,(x_4,y_4) \bigr) - \frac{1}{4!} \sum_{\sigma \in \mathcal{S}_4} h_2^{(1)} \bigl( (x_{\sigma(1)},y_{\sigma(1)}),(x_{\sigma(2)},y_{\sigma(2)}) \bigr).
\end{align*}
To approximate $\hat{D}_n^{(1)}$ by a second-order $U$-statistic we first look at
\begin{align*}
	\zeta_c^{(1)}&:= \mathbb{E} \Bigl[ \tilde{g}^{(1)} \bigl( (X_1,Y_{\Pi(1)}), \ldots, (X_4,Y_{\Pi(4)}) \bigr) \\
	& \hspace{19pt} \times \tilde{g}^{(1)} \bigl( (X_1,Y_{\Pi(1)}), \ldots, (X_c,Y_{\Pi(c)}), (X_5,Y_{\Pi(5)}), \ldots, (X_{8-c},Y_{\Pi(8-c)}) \bigr) \Bigr]
\end{align*}
for $c=0,1,2,3,4$. First, as in~\eqref{Eq:CrudePerm} in the main text, for $c=3,4$ we have the simple bound that $\max(| \zeta_3^{(1)}|, |\zeta_4^{(1)}|) \lesssim A^8 M^2$. By the construction of $\tilde{g}^{(1)}$ and the null hypothesis version of~\eqref{Eq:SecondOrderU}, we have that
\[
	\mathbb{E} \tilde{g}^{(1)} \bigl( (x_1,y_1),(x_2,y_2), (X_1,Y_2),(X_3,Y_4) \bigr) =0
\]
for all $x_1,y_1,x_2,y_2 \in [0,1]$. It now follows by calculations similar to those in~\eqref{Eq:zetatilde1} and~\eqref{Eq:zetatilde0} in the main text that $|\zeta_c^{(1)}| \lesssim A^8 M^2 n^{c-3}$ for $c=0,1,2$. Hence,
\begin{align}
\label{Eq:DnPermApprox}
	\mathbb{E} \biggl[ \biggl\{ \hat{D}_n^{(1)}& - \binom{n}{2}^{-1} \sum_{i_1<i_2} h_2^{(1)} \bigl( (X_{i_1},Y_{\Pi(i_1)}),(X_{i_2},Y_{\Pi(i_2)}) \bigr) \biggr\}^2 \biggr] \nonumber \\
	&=\mathbb{E} \biggl[ \biggl\{ \binom{n}{4}^{-1} \sum_{i_1< \ldots <i_4} \tilde{g}^{(1)} \bigl((X_{i_1},Y_{\Pi(i_1)}), \ldots, (X_{i_4},Y_{\Pi(i_4)}) \bigr) \biggr\}^2 \biggr] \nonumber \\
	&= \binom{n}{4}^{-1} \sum_{c=0}^4 \binom{4}{c} \binom{n-4}{4-c} \zeta_c^{(1)} \lesssim \frac{A^8 M^2}{n^3}.
\end{align}
As in the null hypothesis version of~\eqref{Eq:h2SecondMoment}, we have that 
\begin{equation}
\label{Eq:AsymptoticPermVariance}
\mathbb{E}\bigl\{ h_2^{(1)} \bigl((X_1,Y_2),(X_3,Y_4) \bigr)^2\bigr\} = \sigma_{M,X}^2\sigma_{M,Y}^2 + O(MA^4).
\end{equation}
Now we bound
\begin{align}
\label{Eq:PermApproxIndep}
	\mathbb{E} \bigl| \mathbb{E} \bigl\{ h_2^{(1)} &\bigl( (X_1,Y_2),(X_3,Y_1) \bigr) | Y_2,X_3 \bigr\} \bigr| \nonumber \\
	& = \mathbb{E} \biggl| \sum_{(j,k) \in \mathcal{M}} (a_{jk}-a_{j\bullet}a_{\bullet k}) \tilde{p}_{jk} (X_3,Y_2) \biggr| \lesssim A D^{1/2},
\end{align}
where the final bound follows from Cauchy--Schwarz and~\eqref{Eq:zeta1bound} in the main text. Define
\[
	g_2^{(1)} \bigl((x_1,y_1),(x_2,y_2) \bigr):= \mathbb{E} \bigl\{ h_2^{(1)} \bigl((x_1,y_1),(X_1,Y_2) \bigr) h_2^{(1)} \bigl( (x_2,y_2), (X_1,Y_2) \bigr) \bigr\}.
\]
As in~\eqref{Eq:SecondMomentg} and~\eqref{Eq:FourthMomenth} we can also see that
\begin{align}
\label{Eq:PermFourthMoments}
	\mathbb{E}\bigl\{ g_2^{(1)} \bigl((X_1,Y_1),(X_2,Y_2) \bigr)^2 \bigr\} &\lesssim A^8 M^2 \biggl(  \sum_{(j,k) \in \mathcal{M}_2} |a_{j\bullet}a_{\bullet k}| \biggr)^2	 \nonumber \\
	\max_{\sigma \in \mathcal{S}_4} \mathbb{E} \bigl\{ h_2^{(1)} \bigl((X_1,Y_{\sigma(1)}),(X_2,Y_{\sigma(2)}) \bigr)^4 \bigr\} &\lesssim M^6 A^{10}.
\end{align}
Using the inequality $d_{\mathcal{G}_B}^2(W,U) \leq \mathbb{E} (\|W-U\|^2)$ for $d_{\mathcal{G}_B}$ defined just above Lemma~\ref{Lemma:JointNormality}, it now follows from~\eqref{Eq:DnhatApprox},~\eqref{Eq:h2SecondMoment},~\eqref{Eq:SecondMomentg},~\eqref{Eq:FourthMomenth},~\eqref{Eq:DnPermApprox},~\eqref{Eq:AsymptoticPermVariance},~\eqref{Eq:PermApproxIndep},~\eqref{Eq:PermFourthMoments}, Lemma~\ref{Lemma:JointNormality} and the facts that $\sigma_M^2:= \sigma_{M,X}^2 \sigma_{M,Y}^2 \gtrsim M^2$ and $\sigma_M^2 \lesssim A^2M^2$ that
\begin{align}
\label{Eq:dGdistance}
	d_{\mathcal{G}_B}&\biggl( \frac{\binom{n}{2}^{1/2}}{\sigma_M} \bigl( \hat{D}_n - \mathbb{E} \hat{D}_n, \hat{D}_n^{(1)}, \ldots, \hat{D}_n^{(B)} \bigr), \bigl( Z_0, Z_1, \ldots, Z_B \bigr) \biggr) \nonumber \\
	&  \lesssim \max \biggl\{ \Bigl( \frac{\Delta_f}{M} \Bigr)^{1/2}, \Bigr( \frac{BA^4}{M} \Bigr)^{1/2}, A D^{1/4}, A^4 D^{1/2}, \nonumber \\
	& \hspace{150pt} \frac{BA^4}{M} \sum_{(j,k) \in \mathcal{M}_2} |a_{j\bullet}a_{\bullet k}|, B^2 \Bigl( \frac{ M^2 A^{10}}{n} \Bigr)^{1/2} \biggr\} \nonumber \\
	& \lesssim_A \max \Bigl\{ \Bigl( \frac{\Delta_f}{M} \Bigr)^{1/2}, \Bigl( \frac{B}{M} \Bigr)^{1/2}, D^{1/4} , \Bigl( \frac{B^4M^2}{n} \Bigr)^{1/2}, \frac{B}{M} \sum_{(j,k) \in \mathcal{M}_2} |a_{j\bullet}a_{\bullet k}| \Bigr\}.
\end{align}

All that remains is to use this bound to approximate the rejection probability of our test by the probability of the corresponding event for the independent Gaussian random variables. We do this by smoothing the indicator function of the rejection event; this is a relatively standard technique for obtaining Berry--Esseen type bounds \citep[e.g.][]{Bentkus2005,CCK2013}, though the shapes of our rejection events can be complicated (see~\eqref{Eq:Complicated} below), so the analysis required some care. Define the function $\varphi: \mathbb{R} \rightarrow \mathbb{R}$ by
\[
	\varphi(t) := \left\{ \begin{array}{ll} e^{-1/t} & \text{if } t>0 \\ 0 & \text{otherwise,} \end{array} \right.
\]
write $\rho(x):=c_B \varphi( 1- \|x\|_2^2)$ for $x \in \mathbb{R}^{B+1}$, where $c_B$ is chosen so that $\int_{\mathbb{R}^{B+1}} \rho(x) \,dx =1$, and for $\epsilon>0$ write $\rho_\epsilon(x) := \epsilon^{-B-1} \rho(x/\epsilon)$. Given a Borel measurable $A \subseteq \mathbb{R}^{B+1}$ and $\epsilon>0$ write $\chi(x):=\mathbbm{1}_{\{x \in A\}}$, write $A^\epsilon := A+B_0(\epsilon)$ and write $\chi_\epsilon(x) := \int_{A^\epsilon} \rho_\epsilon(x-y) \,dy$.  Finally, let $B_x(\epsilon) := \{y \in \mathbb{R}^{B+1}:\|y\|_2 \leq \epsilon\}$.  If $x \in A$ then $B_x(\epsilon) \subseteq A^\epsilon$ and we have $\chi_\epsilon(x)=1$, so that $\chi(x) \leq \chi_\epsilon(x)$ for all $x \in \mathbb{R}^{B+1}$. Moreover, if $ x \not\in A^{2\epsilon}$ then $\chi_\epsilon(x)=0$ so we can only have $\chi(x) \neq \chi_\epsilon(x)$ if $x \in A^{2\epsilon} \setminus A$. Straightforward calculations show that
\begin{align*}
	\nabla \chi_\epsilon (x) & = \frac{2c_B}{\epsilon} \int_{\epsilon^{-1}(A^{\epsilon}-x)} z \varphi'(1-\|z\|^2) \,dz \\
	\nabla^2 \chi_\epsilon(x) &= \frac{2c_B}{\epsilon^2} \int_{\epsilon^{-1}(A^{\epsilon}-x)} \bigl\{ I_{B+1} \varphi'(1-\|z\|^2) - 2 zz^T \varphi''(1-\|z\|^2) \bigr\} \,dz.
\end{align*}
To bound these derivatives, we will first observe that if we write $C_B:= \int_0^1 (1-u)^{B/2} \varphi(u) \,du$, then when $B \geq \frac{6+ 2 \log 4 }{ \log (4/3)} -1 $ we have that
\begin{align*}
	\frac{C_B}{C_{B-1}} &= \frac{\int_0^1 (1-u)^\frac{B}{2} \varphi(u) \,du }{\int_0^1 (1-u)^\frac{B-1}{2} \varphi(u) \,du} \geq 1 - \frac{\int_0^1 u (1-u)^\frac{B-1}{2} \varphi(u) \,du}{\int_0^1 (1-u)^\frac{B-1}{2} \varphi(u) \,du} \\
	& \geq \frac{1}{2} -  \frac{\int_{1/2}^1 (1-u)^\frac{B-1}{2} \varphi(u) \,du}{\int_0^1 (1-u)^\frac{B-1}{2} \varphi(u) \,du} \\
	&\geq \frac{1}{2} -  \frac{\int_{1/2}^1 (1-u)^\frac{B-1}{2} \varphi(u) \,du}{\int_{1/3}^1 (1-u)^\frac{B-1}{2} \varphi(u) \,du} \geq \frac{1}{2} - \frac{\varphi(1)}{\varphi(1/3)} \Bigl( \frac{3}{4} \Bigr)^\frac{B+1}{2} \geq \frac{1}{4}.
\end{align*}
Hence, when $B \geq \frac{6+ 2 \log 4 }{ \log (4/3)}$, we have
\begin{align*}
	M_1 &(\chi_\epsilon) \leq \sup_{x \in \mathbb{R}^{B+1}} \|\nabla \chi_\epsilon(x)\| \leq \frac{2c_B}{\epsilon} \int_{B_0(1)} \|z\| \varphi'(1-\|z\|^2) \,dz  \\
	&= \frac{2 \int_0^1 (1-u)^\frac{B}{2} \varphi'(u) \,du }{\epsilon \int_0^1 (1-u)^\frac{B-1}{2} \varphi(u) \,du} =  \frac{B \int_0^1 (1-u)^\frac{B-2}{2} \varphi(u) \,du }{\epsilon \int_0^1 (1-u)^\frac{B-1}{2}  \varphi(u) \,du} = \frac{B C_{B-2}}{\epsilon C_{B-1}} \leq \frac{4B}{\epsilon},
\end{align*}
where we made the substitution $u = 1 - \|z\|^2$ to obtain the first equality.  Moreover, using the fact that $|\varphi''(u)| \leq 2u^{-4} \varphi(u)$ and integrating by parts, when $B \geq \frac{6+ 2 \log 4 }{ \log (4/3)}$ we have
\begin{align*}
  M_2(\chi_\epsilon) &\leq \sup_{x \in \mathbb{R}^{B+1}} \|\nabla^2 \chi_\epsilon(x)\|_{\mathrm{F}} \\
  &\leq \frac{2c_B}{\epsilon^2} \int_{B_0(1)} \bigl\{ (B+1)^{1/2} \varphi'(1-\|z\|^2) + \|z\|^2 |\varphi''(1-\|z\|^2)| \bigr\} \,dz \\
	& = \frac{2 \int_0^1 r^B \{(B+1)^{1/2} \varphi'(1-r^2) + r^2 |\varphi''(1-r^2)| \} \,dr}{\epsilon^2 \int_0^1 r^B \varphi(1-r^2) \,dr} \\
	& \leq \frac{2 \int_0^1 \{ (B+1)^{1/2} (1-u)^\frac{B-1}{2} \varphi'(u) + 2 (1-u)^\frac{B+1}{2} u^{-2} \varphi'(u) \} \,du}{\epsilon^2 \int_0^1 (1-u)^\frac{B-1}{2} \varphi(u) \,du }\\
	& \leq \frac{7B^2 \int_0^1 (1-u)^\frac{B-3}{2} \varphi(u) \,du }{\epsilon^2 \int_0^1 (1-u)^\frac{B-1}{2} \varphi(u) \,du} = \frac{7 B^2 C_{B-3}}{\epsilon^2 C_{B-1} } \leq \frac{112 B^2}{\epsilon^2}.
\end{align*}
Therefore, there exists a universal constant $c>0$ such that, for all $B \in \mathbb{N}$ and all $\epsilon >0$ we have $ c (\epsilon/B)^2 \chi_\epsilon(\cdot) \in \mathcal{G}_B$. For any random variable $W$ taking values in $\mathbb{R}^{B+1}$ and for $Z$ a standard $(B+1)$-variate normal random vector we therefore have
\begin{align*}
	\mathbb{P}(W \in A) - \mathbb{P}( Z \in A) &= \mathbb{E} \chi(W) - \mathbb{E} \chi(Z) \\
	& \leq \mathbb{E} \chi_\epsilon(W) - \mathbb{E} \chi(Z) \leq \frac{B^2 d_{\mathcal{G}_B}(W,Z)}{c \epsilon^2} + \mathbb{P}( Z \in A^{2\epsilon} \setminus A).
\end{align*}
The rejection regions that we are interested in are of the form
\begin{equation}
\label{Eq:Complicated}
	A_{\delta,\alpha}:= \bigl\{ (w_0,w_1,\ldots,w_B) \in \mathbb{R}^{B+1} : |\{ b \in [B] : w_b > w_0 + \delta\}| \leq (1+B) \alpha -1 \bigr\}.
\end{equation}
Now if $(w_0,w_1,\ldots,w_B)^T \in A_{\delta,\alpha}$ and $(u_0,u_1,\ldots,u_B)^T \in B_0(2 \epsilon)$ then we have
\begin{align*}
  \bigl|\bigl\{ b \in [B] : w_b+u_b > w_0+u_0 + \delta + 4 \epsilon \bigr\}\bigr| &\leq \bigl|\bigl\{ b \in [B] : w_b > w_0 + \delta \bigr\} \bigr| \\
  &\leq (1+B) \alpha -1,
\end{align*}
so that $A_{\delta,\alpha}^{2\epsilon} \subseteq A_{\delta+4\epsilon,\alpha}$. Thus, writing $s = \lceil \alpha(B+1) \rceil -1$ we can use the bounds
\begin{align*}
	\mathbb{P}&(Z \in A_{\delta,\alpha}^{2 \epsilon} \setminus A_{\delta,\alpha}) \leq \mathbb{P}( Z \in A_{\delta+4 \epsilon, \alpha}) - \mathbb{P}( Z \in A_{\delta,\alpha}) \\
	& = \int_0^1 \Bigl\{ \bar{\Phi} \Bigl( \Phi^{-1}(u) - \delta - 4 \epsilon \Bigr) \!-\! \bar{\Phi} \Bigl( \Phi^{-1}(u) - \delta \Bigr) \Bigr\} \mathrm{B}_{B-s,s+1}(u) \,du \leq \frac{4 \epsilon}{(2 \pi)^{1/2}}.
\end{align*}
By considering $A_{\delta,\alpha}^c$ we can establish a matching lower bound and hence see that, for any $\alpha \in (0,1)$ and $\delta \in \mathbb{R}$ we have
\begin{align}
\label{Eq:SmoothedIndicator}
	\bigl| \mathbb{P}(W \in A_{\delta,\alpha}) - \mathbb{P}(Z \in A_{\delta,\alpha}) \bigr| &\leq \inf_{\epsilon >0} \biggl\{ \frac{B^2 d_{\mathcal{G}_B}(W,Z)}{c \epsilon^2} + \frac{4 \epsilon}{(2 \pi)^{1/2}} \biggr\} \nonumber \\
	&\leq \frac{8}{(2 \pi)^{1/2}} \Bigl( \frac{B^2 d_{\mathcal{G}_B}(W,Z)}{c} \Bigr)^{1/3}. 
\end{align}
The first claim now follows from~\eqref{Eq:dGdistance}.

We now turn to the second claim in Theorem~\ref{Thm:PowerFunction}, for which we exploit ideas in \citet{Hoeffding52}. Write $\epsilon_* := B^{-3} \vee \delta_*^3$.  We will use the shorthand
\[
	T_n:= \frac{\binom{n}{2}^{1/2} \hat{D}_n}{ \sigma_M} \quad \text{and} \quad T_n^{(b)}:= \frac{\binom{n}{2}^{1/2} \hat{D}_n^{(b)}}{ \sigma_M}
\]
for $b\in [B]$, and $F_{n,B}(z):= B^{-1} \sum_{b=1}^B \mathbbm{1}_{\{T_n^{(b)} \leq z\}}$. Write $C_{n,B}$ for the $\lceil (B+1)(1-\alpha) \rceil$th smallest element of $(T_n^{(1)},\ldots,T_n^{(B)})$, so that we reject $H_0$ if and only if $T_n > C_{n,B}$, and note that
\[
	C_{n,B} \leq z \quad \text{if and only if} \quad F_{n,B}(z) \geq B^{-1} \lceil (B+1)(1-\alpha) \rceil
\]
for all $z \in \mathbb{R}$. Now, we have by Proposition~\ref{Lemma:PermutedNormality} in the main text, together with~\eqref{Eq:DnPermApprox},~\eqref{Eq:AsymptoticPermVariance},~\eqref{Eq:PermApproxIndep} and~\eqref{Eq:PermFourthMoments} that when $Z \sim N(0,1)$,
\begin{align*}
  \sup_{z \in \mathbb{R}} &| \mathbb{E} F_{n,B}(z) - \Phi(z) | = d_\mathrm{K}( T_n^{(1)}, Z) \\
                          &\lesssim d_\mathrm{W}^{1/2}\biggl( T_n^{(1)}, \frac{1}{\sigma_{M}\binom{n}{2}^{1/2}}\sum_{i_1 < i_2} h_2^{(1)}\bigl((X_{i_1},Y_{\Pi(i_1)}),(X_{i_2},Y_{\Pi(i_2)})\bigr)\biggr) \\
  &\hspace{50pt}+ d_\mathrm{W}^{1/2}\biggl(\frac{1}{\sigma_{M}\binom{n}{2}^{1/2}}\sum_{i_1 < i_2} h_2^{(1)}\bigl((X_{i_1},Y_{\Pi(i_1)}),(X_{i_2},Y_{\Pi(i_2)})\bigr),Z\biggr) \\
  &\lesssim \epsilon_*^{1/2}.
\end{align*}
Moreover, approximating the sets $A_z:=\{(w_1,w_2,w_3) \in \mathbb{R}^3 : w_2,w_3 \leq z\}$ for $z \in \mathbb{R}$ by smoothed indicator functions similarly to~\eqref{Eq:SmoothedIndicator} above, we also have that when $Z_1,Z_2,Z_3 \stackrel{\mathrm{iid}}{\sim} N(0,1)$,
\begin{align*}
	\sup_{z \in \mathbb{R}} \mathrm{Var}& F_{n,B}(z) \leq \sup_{z \in \mathbb{R}} \Bigl\{ B^{-1} \mathrm{Var}( \mathbbm{1}_{\{T_n^{(1)} \leq z\}} ) + \mathrm{Cov}( \mathbbm{1}_{\{T_n^{(1)} \leq z\}}, \mathbbm{1}_{\{T_n^{(2)} \leq z\}}) \Bigr\} \\
	& \leq B^{-1} + \sup_{z \in \mathbb{R}} \bigl\{ \mathbb{P}(T_n^{(1)} \leq z, T_n^{(2)} \leq z) - \Phi(z)^2 \bigl\} + 2 d_\mathrm{K}(T_n^{(1)},Z) \\
	& \lesssim B^{-1} + d_{\mathcal{G}_2}^{1/3} \bigl( (T_n - \mathbb{E} T_n, T_n^{(1)}, T_n^{(2)}), (Z_1,Z_2,Z_3) \bigr) + d_\mathrm{W}^{1/2}(T_n^{(1)},Z) \lesssim \epsilon_*^{1/3}.
\end{align*}
Here, the final bound follows from~\eqref{Eq:dGdistance}.  Using the shorthand $z_{1-\alpha}=\Phi^{-1}(1-\alpha)$, taking $\epsilon = \epsilon_*^{1/9}$ and using Proposition~\ref{Lemma:Normality} in the main text, there therefore exists $c=c(\alpha) > 0$ such that when $\epsilon_* \leq c$ we have
\begin{align*}
	\bigl|& \mathbb{P}(P \leq \alpha) - \bar{\Phi} \bigl( \Phi^{-1}(1-\alpha) - \Delta_f \bigr) \bigr| \\
	& \leq \bigl| \mathbb{P}( T_n > C_{n,B}) - \mathbb{P}\bigl( T_n > z_{1-\alpha} \bigr) \bigr| + d_\mathrm{K}(T_n - \mathbb{E} T_n, Z) \\
	& \leq \mathbb{P}( | T_n - z_{1-\alpha} | < \epsilon) + \mathbb{P}(|C_{n,B} - z_{1-\alpha} | \geq \epsilon ) + d_\mathrm{K}(T_n - \mathbb{E} T_n, Z) \\
	& \leq \frac{2^{1/2} \epsilon}{\pi^{1/2}} + 3 d_\mathrm{K}(T_n - \mathbb{E} T_n, Z) + \mathbb{P}\bigl( F_{n,B}(z_{1-\alpha} + \epsilon) \leq B^{-1} \lceil (B+1)(1-\alpha) \rceil \bigr) \\
	& \hspace{100pt} + \mathbb{P}( F_{n,B}(z_{1-\alpha} - \epsilon) \geq B^{-1} \lceil (B+1)(1-\alpha) \rceil ) \\
	& \leq \frac{2^{1/2} \epsilon}{\pi^{1/2}} + 3 d_\mathrm{K}(T_n \!-\! \mathbb{E} T_n, Z) + \frac{ \mathrm{Var} \, F_{n,B}(z_{1-\alpha} + \epsilon)}{ \{ \mathbb{E} F_{n,B}(z_{1-\alpha} + \epsilon) \!-\! B^{-1} \lceil (B+1)(1-\alpha) \rceil \}^2 } \\
	& \hspace{100pt} + \frac{ \mathrm{Var} \, F_{n,B}(z_{1-\alpha} - \epsilon)}{ \{ B^{-1} \lceil (B+1)(1-\alpha) \rceil \!-\! \mathbb{E} F_{n,B}(z_{1-\alpha} - \epsilon) \}^2 } \\
	& \lesssim \epsilon_*^{1/9} + \frac{\epsilon_*^{1/3}}{\epsilon^2} = 2 \epsilon_*^{1/9}.
\end{align*}
If, on the other hand, $\epsilon_* > c$ then the bound
\[
	\bigl| \mathbb{P}(P \leq \alpha) - \bar{\Phi} \bigl( z_{1-\alpha} - \Delta_f \bigr) \bigr| \leq 2 \lesssim \epsilon_*^{1/9}
\]
is trivial.

The result follows upon noting that $|\frac{d}{dx} \bar{\Phi}(\Phi^{-1}(x) - \Delta_f)| \leq \exp(\Phi^{-1}(x)^2/2)$ and writing
\begin{align*}
	\bigl| \mathbb{E} &\bar{\Phi} \bigl( \Phi^{-1}( \mathrm{B}_{B-s,s+1}) - \Delta_f \bigr) - \bar{\Phi} \bigl( z_{1-\alpha} - \Delta_f \bigr) \bigr| \\
	& \leq \mathbb{P}( |\mathrm{B}_{B-s,s+1} - (1-\alpha)| \geq \alpha /2 ) + e^{\frac{1}{2} z_{1-\alpha/2}^2} \mathbb{E}| \mathrm{B}_{B-s,s+1} - (1-\alpha) | \\
	& \leq  \Bigl( \frac{2}{\alpha} + e^{\frac{1}{2} z_{1-\alpha/2}^2} \Bigr) \mathbb{E}| \mathrm{B}_{B-s,s+1} - (1-\alpha) | \\
	& \leq \Bigl( \frac{2}{\alpha} + e^{\frac{1}{2} z_{1-\alpha/2}^2} \Bigr) \Bigl\{ \frac{1}{B+1} + \frac{1}{(B+1)^{1/2}} \Bigr\} \lesssim \epsilon_*^{1/6}.
\end{align*}
\end{proof}

As mentioned in the main text, Proposition~\ref{Lemma:Normality} is a consequence of \citet[][Theorem~3.3]{DoblerPeccati2019}.  Nevertheless, we provide the proof below, both because the arguments simplify when specialising that result (which holds for $U$-statistics of general order), and because it will aid the understanding of the proof of Proposition~\ref{Lemma:PermutedNormality} which follows.
\begin{proof}[Proof of Proposition~\ref{Lemma:Normality}]
Let $\mathcal{G}_0$ denote the class of twice differentiable functions $G: \mathbb{R} \rightarrow \mathbb{R}$ and define
\[
	\mathcal{G}:= \bigl\{ G \in \mathcal{G}_0 : \max( \|G\|_\infty, \|G'\|_\infty) \leq 2, \|G''\|_\infty \leq (2/\pi)^{1/2} \bigr\}.
\]
The starting point of Stein's method is the bound
\begin{equation}
\label{Eq:SteinBound}
	d_\mathrm{W}(U,W) \leq \sup_{G \in \mathcal{G}} \bigl| \mathbb{E}\{ G'(U) - U G(U)\}\bigr|,
\end{equation}
which is proved, for example, in~\citet[][Theorem~3.1]{Ross11}. For $i \in [n]$, write
\[
	V_i:= \binom{n}{2}^{-1/2} \sum_{i':i' \neq i} h(Z_i,Z_{i'}),
\]
and $U_i:=U-V_i$. 
Now
\[
  U_1 = \binom{n}{2}^{-1/2}\sum_{1 < i < i'} h(Z_i,Z_{i'}),
\]
so that $U_1$ is $\sigma(Z_2,\ldots,Z_n)$-measurable; more generally, $U_i$ is $\sigma(Z_{i'}:i' \neq i)$ measurable, so $U_i$ and $Z_i$ are independent for each $i \in [n]$.  Thus, for all $G \in \mathcal{G}$ and $i \in [n]$, we have
\[
	\mathbb{E}\{V_i G(U_i)\} = \binom{n}{2}^{-1/2} \sum_{i':i' \neq i} \mathbb{E} \bigl[ G(U_i) \mathbb{E} \bigl\{ h(Z_i,Z_{i'}) | (Z_{i'}:i'\neq i) \bigr\} \bigr] = 0,
\]
by the degeneracy of $h$. Moreover, $V_1,\ldots,V_n$ have the same distribution and $\mathbb{E}(V_1^2)= \binom{n}{2}^{-1}\sum_{i=2}^n \mathbb{E}\bigl\{h^2(Z_1,Z_i)\bigr\} = 2/n$.  Since $U = \sum_{i=1}^n V_i/2$, this allows us to write
\begin{align}
\label{Eq:SteinDecomp}
	&d_\mathrm{W}(U,W) \leq \sup_{G \in \mathcal{G}} \biggl| \mathbb{E} \biggl\{ G'(U) - \frac{1}{2} \sum_{i=1}^n V_i G(U) \biggr\} \biggr| \nonumber \\
	& = \sup_{G \in \mathcal{G}} \biggl| \mathbb{E} \biggl[ \biggl\{ \!1 \!-\! \frac{1}{2} \sum_{i=1}^n V_i^2 \biggr\} G'(U) \!-\! \frac{1}{2} \sum_{i=1}^n V_i \bigl\{ G(U) \!-\! G(U_i) \!-\! (U \!-\! U_i) G'(U) \bigr\} \biggr] \biggr| \nonumber \\
	& \leq 2 \mathbb{E} \biggl\{ \biggl| 1 - \frac{1}{2} \sum_{i=1}^n V_i^2 \biggr| \biggr\} + \frac{1}{2^{3/2} \pi^{1/2}} \sum_{i=1}^n \mathbb{E} \{ |V_i|^3 \} \nonumber \\
	& \leq \bigl\{ n\mathrm{Var}(V_1^2) + n(n-1)\mathrm{Cov}(V_1^2,V_2^2) \bigr\}^{1/2} +\frac{n}{2^{3/2} \pi^{1/2}} \mathbb{E}\{ |V_1|^3 \}.
\end{align}
It now remains to bound these moments of $V_1,\ldots,V_n$. First,
\begin{align}
\label{Eq:Moment1}
	&\mathrm{Var}(V_1^2) = \binom{n}{2}^{-2} \mathbb{E} \biggl[ \biggl\{ \sum_{i=2}^n h(Z_1,Z_i) \biggr\}^4 \biggr] - \frac{4}{n^2} \nonumber \\
	& \leq \binom{n}{2}^{-2} \bigl\{ (n-1) \mathbb{E}[ h^4(Z_1,Z_2)] + 3(n-1)(n-2) \mathbb{E}[ h^2(Z_1,Z_2) h^2(Z_1,Z_3) ] \bigr\} \nonumber \\
	& \leq \frac{12}{n^2} \mathbb{E}[ h^4(Z_1,Z_2)].
\end{align}
Now
\begin{align}
\label{Eq:Moment2}
	\binom{n}{2}^2 &\mathrm{Cov}( V_1^2, V_2^2) \nonumber \\
	& = \sum_{\substack{ i_1,i_2 \neq 1 \\ i_3, i_4 \neq 2}} \mathbb{E} \bigl\{ h(Z_1,Z_{i_1}) h(Z_1,Z_{i_2}) h(Z_2,Z_{i_3}) h(Z_2,Z_{i_4}) \bigr\} -(n-1)^2 \nonumber\\
	&=(n-2)(n-3)-(n-1)^2 +  \mathbb{E}\{ h^4(Z_1,Z_2)\}   \nonumber\\
	& \hspace{0.4cm}+ 4(n-2) \mathbb{E} \Bigl[ h(Z_1,Z_3) \bigl\{ h(Z_1,Z_2) + h(Z_1,Z_3) \bigr\} \bigl\{ h(Z_2,Z_1) + h(Z_2,Z_3) \bigr\}^2\Bigr] \nonumber\\
	& \hspace{0.4cm}+2(n\!-\!2)(n\!-\!3) \mathbb{E}\bigl\{ h(Z_1,Z_3) h(Z_1,Z_4) h(Z_2,Z_3) h(Z_2,Z_4) \bigr\}\nonumber\\
	&= -3n + 5 +  2(n-2)(n-3) \mathbb{E}\bigl\{ g^2(Z_1,Z_2)\bigr\} + \mathbb{E}\{ h^4(Z_1,Z_2)\} \nonumber \\
	& \hspace{0.4cm}+ 20(n-2) \mathbb{E} \bigl\{ h^2(Z_1,Z_2) h(Z_1,Z_3) h(Z_2,Z_3) \bigr\}  \nonumber \\
	& \hspace{0.4cm}+ 8(n\!-\!2) \mathbb{E} \bigl\{ h^2(Z_1,Z_2) h^2(Z_1,Z_3)\bigr\} \nonumber\\
	& \leq 2n(n-1) \mathbb{E}\{ g^2(Z_1,Z_2) \} + 28(n-1)\mathbb{E}\{h^4(Z_1,Z_2)\}.
\end{align}
Finally, by Cauchy--Schwarz,
\begin{align}
\label{Eq:Moment3}
\mathbb{E} (|V_1|^3) \leq \mathbb{E}^{1/2}(V_1^2) \mathbb{E}^{1/2}(V_1^4) \leq \biggl\{ \frac{2}{n} \biggl( \frac{4}{n^2} + \frac{12}{n^2} \mathbb{E}\bigl\{h^4(Z_1,Z_2)\bigr\} \biggr) \biggr\}^{1/2}.
\end{align}
It now follows from~\eqref{Eq:SteinDecomp}, \eqref{Eq:Moment1}, \eqref{Eq:Moment2} and~\eqref{Eq:Moment3} that
\begin{align*}
	d_\mathrm{W}(U,W) \leq \biggl( 124 \frac{\mathbb{E}\{ h^4(Z_1,Z_2)\}}{n}& + 8 \mathbb{E}\{g^2(Z_1,Z_2)\}  \biggr)^{1/2} \! \! \! \\
	&+ \frac{1}{(n\pi)^{1/2}} \bigl[ 1 + 3\mathbb{E}\bigl\{ h^4(Z_1,Z_2)\bigr\} \bigr]^{1/2}  
\end{align*}
and the result is immediate.
\end{proof}

\begin{proof}[Proof of Proposition~\ref{Lemma:PermutedNormality}]
In the proof of this result we broadly follow the structure of Proposition~\ref{Lemma:Normality}, though there are of course extra difficulties in accommodating the random permutation $\Pi$. We first show that we may effectively ignore those $i \in [n]$ that fall in short cycles of $\Pi$ as they make up a small proportion of $[n]$. More precisely, for $m \in \mathbb{N}$, define $\mathcal{C}_m = \mathcal{C}_m(\Pi) :=\{i \in [n] : i \text{ falls in a cycle of length} \geq m\}$ and
\[
	V:= \frac{1}{2} \binom{n}{2}^{-1/2} \sum_{\substack{i_1 \neq i_2 \\ i_1,i_2 \in \mathcal{C}_m}} h \bigl( (X_{i_1},Y_{\Pi(i_1)}), (X_{i_2}, Y_{\Pi(i_2)}) \bigr).
\]
It will suffice for our purposes to take $m=6$. We now bound $\mathbb{E}\{(W-V)^2\}$. For a function $b: \mathbb{N} \rightarrow \mathbb{R}$ and $\sigma \in \mathcal{S}_n$ write $B(\sigma):=\sum_{c \in \sigma} b(|c|)$, where this sum is over the cycles of $\sigma$ and $|c|$ denotes the length of cycle $c$. Now, for $z \in \mathbb{C}$ with $|z| < 1$, and for $u > 0$, define the generating function
\[
	g(z,u) := 1 + \sum_{n=1}^\infty \biggl( \sum_{\sigma \in \mathcal{S}_n} u^{B(\sigma)} \biggr) \frac{z^n}{n!} 
\]
We may partition $\mathcal{S}_n$ by the number $m$ of cycles the permutations contain and the sizes $k_1,\ldots,k_m$ of these cycles to write
\begin{align*}
	 g(z,u) &= 1 + \sum_{n=1}^\infty \biggl( \sum_{\sigma \in \mathcal{S}_n} u^{B(\sigma)} \biggr) \frac{z^n}{n!} \\
	 &=1 \!+\! \sum_{n=1}^\infty \frac{z^n}{n!} \sum_{m=1}^n \frac{1}{m!} \!\!\! \sum_{\substack{ k_1,\ldots,k_m \in \mathbb{N} \\ k_1 + \ldots + k_m = n}} \!\!\!\!\!\!\!\!\! u^{b(k_1) + \ldots + b(k_m)} \! \binom{n}{k_1,\ldots,k_m} \! (k_1 \!-\! 1)! \ldots (k_m \!-\! 1)! \\
	& = 1 + \sum_{n=1}^\infty z^n \sum_{m=1}^n \frac{1}{m!} \sum_{\substack{ k_1,\ldots,k_m \in \mathbb{N} \\ k_1 + \ldots + k_m = n}} \frac{u^{b(k_1)}}{k_1} \ldots \frac{u^{b(k_m)}}{k_m} \\
	& = 1 + \sum_{m=1}^\infty \frac{1}{m!} \biggl( \sum_{k=1}^\infty \frac{u^{b(k)} z^k}{k} \biggr)^m  = \frac{1}{1-z} \exp \biggl( \sum_{k=1}^\infty \frac{(u^{b(k)}-1) z^k}{k} \biggr).
\end{align*}
By definition, this $g$ has the property that
\[
	\sum_{n=1}^\infty \frac{z^n}{n!} \sum_{\sigma \in \mathcal{S}_n} \frac{B(\sigma)!}{\{B(\sigma)-\ell\}!}= \frac{\partial^\ell g}{\partial u^\ell} (z,1)
\]
for $\ell \in \mathbb{N}_0$, so that the coefficient of $z^n$ in this expansion is equal to $\mathbb{E}\bigl\{B(\Pi)!/\{B(\Pi)-\ell\}!\bigr\}$. We will take $b(k)=k \mathbbm{1}_{\{k <m\}}$, so that this expectation coincides with $\mathbb{E}\{|\mathcal{C}_m^c|! /(|\mathcal{C}_m^c|-\ell)!\}$. With this choice of $b$, we have
\[
	(1-z)g(z,u) = \exp\biggl( \sum_{k=1}^{m-1} (u^k - 1) z^k/k \biggr).
\]
Moreover, by differentiation we can see that
\begin{align*}
	(1-z)\frac{\partial g}{\partial u} (z,1) &= \sum_{k=1}^{m-1} z^k \\
	(1-z)\frac{\partial^2 g}{\partial u^2} (z,1) &= \biggl( \sum_{k=1}^{m-1}z^k \biggr)^2 +\sum_{k=1}^{m-2} k z^{k+1} \\
	(1-z)\frac{\partial^3 g}{\partial u^3} (z,1) &=\biggl( \sum_{k=1}^{m-1}z^k \biggr)^3 + 3\biggl( \sum_{k=1}^{m-1}z^k \biggr) \sum_{k=1}^{m-2}kz^{k+1} + \sum_{k=1}^{m-3} k(k+1)z^{k+2} \\
	(1-z)\frac{\partial^4 g}{\partial u^4} (z,1) &=\biggl( \sum_{k=1}^{m-1}z^k \biggr)^4 + 6 \biggl( \sum_{k=1}^{m-1}z^k \biggr)^2  \sum_{k=1}^{m-2}kz^{k+1}  + 3\biggl( \sum_{k=1}^{m-2}kz^{k+1} \biggr)^2 \\
	&  + 4\biggl( \sum_{k=1}^{m-1}z^k \biggr) \sum_{k=1}^{m-2}k(k+1)z^{k+2} +  \sum_{k=1}^{m-4}k(k+1)(k+2)z^{k+3}.
\end{align*}
Finding the coefficients of $z^n$ in the expansions of these functions yields the facts that, when $n \geq 4(m-1)$,
\begin{align*}
	&\mathbb{E}\{|\mathcal{C}_m^c|\} = (m-1), \quad \mathbb{E}\{|\mathcal{C}_m^c| (|\mathcal{C}_m^c|-1)\} = \frac{1}{2}(m-1)(3m-4), \\
	&\mathbb{E}\biggl\{3! \binom{|\mathcal{C}_m^c|}{3}\biggr\} = \frac{1}{6}(m-1)(17m^2 - 49m + 36), \\
	&\mathbb{E}\biggl\{4! \binom{|\mathcal{C}_m^c|}{4}\biggr\} = \frac{1}{6}(m-1)(38m^3-174m^2+271m-144).
\end{align*}
In particular, $\mathbb{E}\{|\mathcal{C}_m^c|^2\} = (m-1)(3m-2)/2$ and $\mathbb{E}\{|\mathcal{C}_m^c|^4\} = (m-1)(19m^3-36m^2+20m-3)/3$. It now follows that
\begin{align}
\label{Eq:SmallCyclesBound}
	\mathbb{E}\bigl\{(U-V&)^2 \bigr\} \leq \frac{3}{2\binom{n}{2}} \mathbb{E}\biggl[ \biggl\{ \sum_{\substack{i_1 \neq i_2 \\ i_1,i_2 \in \mathcal{C}_m^c}} h \bigl( (X_{i_1},Y_{\pi(i_1)}), (X_{i_2}, Y_{\pi(i_2)}) \bigr) \biggr\}^2 \biggr] \nonumber \\
	& \hspace{50pt} + \frac{3}{\binom{n}{2}} \mathbb{E}\biggl[ \biggl\{ \sum_{\substack{i_1 \in \mathcal{C}_m \\ i_2 \in \mathcal{C}_m^c}} h \bigl( (X_{i_1},Y_{\pi(i_1)}), (X_{i_2}, Y_{\pi(i_2)}) \bigr) \biggr\}^2 \biggr] \nonumber \\
	& \leq \frac{3}{2 \binom{n}{2}} \mathbb{E}\bigl\{|\mathcal{C}_m^c|^2 (|\mathcal{C}_m^c|-1)^2 \bigr\} \max_{\sigma \in \mathcal{S}_4} \mathbb{E}\bigl\{ h^2\bigl( (X_1,Y_{\sigma(1)}), (X_2, Y_{\sigma(2)}) \bigr) \bigr\} \nonumber \\
	& \hspace{50pt} + \frac{3n}{\binom{n}{2}} \mathbb{E}\bigl(|\mathcal{C}_m^c|^2\bigr) \max_{\sigma \in \mathcal{S}_2} \mathbb{E} \bigl\{ h^2 \bigl((X_1,Y_{\sigma(1)}),(X_3,Y_4) \bigr) \bigr\} \nonumber \\
	& \leq \frac{28m^4}{n-1} \max_{\sigma \in \mathcal{S}_4} \mathbb{E}\bigl\{ h^2\bigl( (X_1,Y_{\sigma(1)}), (X_2, Y_{\sigma(2)}) \bigr) \bigr\}.
\end{align}

Now that we have bounded the difference between $U$ and $V$, we aim to establish the approximate normality of $V$. For $i \in \mathcal{C}_m$, write
\[
	Z_i:= \binom{n}{2}^{-1/2} \sum_{i' \in \mathcal{C}_m \setminus \{i\}} h \bigl((X_i,Y_{\Pi(i)}),(X_{i'},Y_{\Pi(i')}) \bigr).
\]
so that $\sum_{i \in \mathcal{C}_m} Z_i = 2V$. With this definition, for $i \in \mathcal{C}_m$, also define
\begin{align*}
	V_i & := V - Z_i - Z_{\Pi^{-1}(i)} + \binom{n}{2}^{-1/2} h \bigl((X_i,Y_{\Pi(i)}),(X_{\Pi^{-1}(i)},Y_i) \bigr) \\
	&\phantom{:} = \frac{1}{2} \binom{n}{2}^{-1/2}  \sum_{\substack{i_1,i_2 \in \mathcal{C}_m \setminus \{i,\Pi^{-1}(i)\} \\ i_1 \neq i_2}} h \bigl((X_{i_1},Y_{\Pi(i_1)}),(X_{i_2},Y_{\Pi(i_2)}) \bigr),
\end{align*}
which is $\sigma\bigl(\Pi,\{(X_{i'},Y_{i'}):i' \neq i\}\bigr)$-measurable.  With $\mathcal{G}$ as in the proof of Proposition~\ref{Lemma:Normality}, for any $G \in \mathcal{G}$ we therefore have that for $m \geq 3$,
\begin{align}
\label{Eq:VGV}
	&\biggl| \mathbb{E} \biggl[ VG(V) - \frac{1}{4} \sum_{i \in \mathcal{C}_m} (V-V_i) \bigl\{G(V) - G(V_i) \bigr\} \biggr] \biggr| \nonumber \\
	&= \biggl| \mathbb{E} \biggl[ VG(V) \! - \! \frac{1}{2} \sum_{i \in \mathcal{C}_m} Z_i G(V) \! \nonumber \\
	& \hspace{5pt} + \!\!\! \! \! \! \! \!  \sum_{ \substack{i \in \mathcal{C}_m \\ i' \in \mathcal{C}_m \setminus \{i, \Pi^{-1}(i)\}}}  \! \! \! \! \! \! \! \! \! \! \!  \frac{G(V_i)}{4\binom{n}{2}^\frac{1}{2}} \bigl\{ \! h \bigl((X_i,Y_{\Pi(i)}), \!(X_{i'},Y_{\Pi(i')}) \bigr) \! +\! h \bigl((X_{\Pi^{-1}(i)},Y_i),\!(X_{i'},Y_{\Pi(i')}) \bigr) \bigr\} \!  \nonumber \\
	& \hspace{5pt}+ \! \frac{1}{4\binom{n}{2}^{1/2}}  \sum_{i \in \mathcal{C}_m} \! \bigl\{G(V) \! + \! G(V_i)\bigr\} h \bigl((X_i,Y_{\Pi(i)}),(X_{\Pi^{-1}(i)},Y_i) \bigr) \biggr] \biggr| \nonumber \\
	& = \frac{1}{4}\binom{n}{2}^{-1/2} \biggl| \mathbb{E} \sum_{i \in \mathcal{C}_m} \bigl\{2G(V_i) + G(V)-G(V_i)\bigr\} h \bigl((X_i,Y_{\Pi(i)}),(X_{\Pi^{-1}(i)},Y_i) \bigr) \biggr| \nonumber \\
	& \leq n \binom{n}{2}^{-1/2} \mathbb{E} \bigl| \mathbb{E} \bigl\{ h \bigl( (X_1,Y_2),(X_3,Y_1) \bigr) | Y_2, X_3 \bigr\} \bigr|  \nonumber \\
	& \hspace{10pt}+ \frac{n}{2} \binom{n}{2}^{-1/2} \mathbb{E}^{1/2}\bigl\{ h^2 \bigl((X_1,Y_2),(X_3,Y_1)\bigr) \bigr\} \times \mathbb{E}\Bigl[ \max_{i \in \mathcal{C}_m} \mathbb{E}\bigl\{(V-V_i)^2 | \Pi \bigr\}^{1/2} \Bigr].
\end{align}
We will now study the behaviour of $\mathbb{E}\{(V-V_i)^2 | \Pi \}$ for $i \in \mathcal{C}_m$.  As long as $m\geq 4$, we have for all $i \in \mathcal{C}_m$ that
\begin{align*}
	&\mathbb{E}(Z_i^2 | \Pi) \\
	&= \!\! \binom{n}{2}^{-1} \!\!\!\!\!\!\!\! \sum_{i_1',i_2' \in \mathcal{C}_m \setminus \{i\}} \!\!\!\!\!\!\! \mathbb{E} \bigl\{ h \bigl( (X_i, Y_{\Pi(i)}), (X_{i_1'}, Y_{\Pi(i_1')}) \bigr) h \bigl( (X_i, Y_{\Pi(i)}), (X_{i_2'}, Y_{\Pi(i_2')}) \bigr) \! \bigm| \! \Pi\bigr\} \\
	&=\binom{n}{2}^{-1} \bigl[ (|\mathcal{C}_m| -3)  \mathbb{E}  h^2\bigl( (X_1,Y_2), (X_3,Y_4) \bigr)+ 2  \mathbb{E} h^2 \bigl( (X_1,Y_2), (X_2,Y_3) \bigr)  \bigr]
\end{align*}
Moreover, as long as $m\geq 5$,
\begin{align*}
	&\mathbb{E}(Z_i Z_{\Pi^{-1}(i)} | \Pi ) \\
	&= \binom{n}{2}^{-1} \sum_{\substack{ i_1' \in \mathcal{C}_m \setminus \{i\} \\ i_2' \in \mathcal{C}_m \setminus \{\Pi^{-1}(i)\}}} \mathbb{E} \bigl\{ h \bigl( (X_i, Y_{\Pi(i)}), (X_{i_1'}, Y_{\Pi(i_1')}) \bigr) \\
	& \hspace{175pt} \times h \bigl( (X_{\Pi^{-1}(i)}, Y_i), (X_{i_2'}, Y_{\Pi(i_2')}) \bigr) \bigm| \Pi\bigr\} \\
	&= \binom{n}{2}^{-1} \mathbb{E} \bigl\{ h^2 \bigl((X_1,Y_2),(X_2,Y_3) \bigr) \bigr\}.
\end{align*}
As a result, for any $i \in \mathcal{C}_m$ and $m \geq 5$,
\begin{align}
\label{Eq:VVi}
	\bigl|& \mathbb{E} \{(V-V_i)^2 | \Pi \} - 4/n \bigr| \nonumber \\
	& \leq \bigl| \mathbb{E}(Z_i^2 | \Pi ) - 2/n\bigr| + \bigl|\mathbb{E}(Z_{\Pi^{-1}(i)}^2 | \Pi) -2/n\bigr| \nonumber \\
	&\hspace{40pt}+ \biggl| \mathbb{E} \biggl\{ 2Z_i Z_{\Pi^{-1}(i)} + \binom{n}{2}^{-1} h^2 \bigl( (X_i,Y_{\Pi(i)}),(X_{\Pi^{-1}(i)},Y_i) \bigr) \nonumber \\
	& \hspace{40pt} -2\binom{n}{2}^{-1/2} (Z_i+Z_{\Pi^{-1}(i)}) h \bigl( (X_i,Y_{\Pi(i)}),(X_{\Pi^{-1}(i)},Y_i) \bigr) \biggm| \Pi \biggr\} \biggr| \nonumber \\
	& \leq 2(n - |\mathcal{C}_m| +2 ) \binom{n}{2}^{-1}+ 5\binom{n}{2}^{-1} \mathbb{E} \bigl\{ h^2 \bigl((X_1,Y_2),(X_2,Y_3) \bigr) \bigr\}.
\end{align}
It further follows, using the fact that $\mathbb{E}|\mathcal{C}_m|=n-m+1$, that for any $G \in \mathcal{G}$,
\begin{align}
\label{Eq:PermutedSteinDecomp}
	\biggl| \mathbb{E} \biggl[ &\frac{1}{4} \sum_{i \in \mathcal{C}_m} (V-V_i)\bigl\{G(V)-G(V_i)\bigr\} - G'(V) \biggr] \biggr| \nonumber \\
	& \leq 2 \mathbb{E} \biggl| \frac{1}{4} \sum_{i \in \mathcal{C}_m} (V-V_i)^2 -1 \biggr| + \frac{1}{4}  \mathbb{E}\biggl[ \sum_{i \in \mathcal{C}_m}  \bigl| (V-V_i)\{G(V)-G(V_i) -(V-V_i)G'(V)\} \bigr| \biggr] \nonumber \\
	& \leq \frac{1}{2} \mathbb{E}\biggl\{\mathrm{Var}^{1/2} \biggl( \sum_{i \in \mathcal{C}_m} (V-V_i)^2 \biggm| \Pi\biggr)\biggr\}  + \frac{1}{(32 \pi)^{1/2}} \mathbb{E} \biggl( \sum_{i \in \mathcal{C}_m} |V-V_i|^3  \biggr) \nonumber \\
	& \hspace{100pt} + \frac{4(m-1)+5 \mathbb{E} \bigl\{h^2\bigl((X_1,Y_2),(X_2,Y_3)\bigr)\bigr\}}{n-1}.
\end{align}
The main task in the rest of the proof is to bound the first two terms on the right-hand side of~\eqref{Eq:PermutedSteinDecomp}.  To this end, for any $i \in \mathcal{C}_m$ we have for $m \geq 6$ that
\begin{align*}
	&\mathbb{E}\{(V-V_i)^4 | \Pi\} \leq 9 \mathbb{E}(Z_i^4 | \Pi ) \\
	& \hspace{75pt} + 9 \mathbb{E}( Z_{\Pi^{-1}(i)}^4 | \Pi) + \frac{9}{\binom{n}{2}^{2}} \mathbb{E} \bigl\{ h^4 \bigl( (X_i, Y_{\Pi(i)}),(X_{\Pi^{-1}(i)},Y_i) \bigr) | \Pi \bigr\} \\
	& \leq 18 \binom{n}{2}^{-2} \mathbb{E} \biggl\{ \frac{5}{2} h^4 \bigl((X_1,Y_2),(X_2,Y_3) \bigr) + (|\mathcal{C}_m|-3) h^4 \bigl((X_1,Y_2),(X_3,Y_4) \bigr)  \\
	& \hspace{30pt} + 6 h^2\bigl((X_1,Y_2),(X_2,Y_3)\bigr)h^2\bigl((X_1,Y_2),(X_4,Y_1)\bigr) \\
	& \hspace{30pt} + 12(|\mathcal{C}_m|-3) h^2\bigl((X_1,Y_2),(X_2,Y_3)\bigr)h^2\bigl((X_1,Y_2),(X_4,Y_5)\bigr) \\
	& \hspace{30pt} + 12(|\mathcal{C}_m|-3) h^2 \bigl((X_1,Y_2),(X_3,Y_4) \bigr) h^2 \bigl( (X_1,Y_2),(X_4,Y_5) \bigr) \\
	& \hspace{30pt} + 3(|\mathcal{C}_m|-3)(|\mathcal{C}_m|-6) h^2 \bigl((X_1,Y_2),(X_3,Y_4) \bigr) h^2 \bigl( (X_1,Y_2),(X_5,Y_6) \bigr) \biggr\} \\
	& \leq \frac{216}{(n-1)^2} \max_{\sigma \in \mathcal{S}_4} \mathbb{E} \bigl\{ h^4 \bigl((X_1,Y_{\sigma(1)}),(X_2,Y_{\sigma(2)}) \bigr) \bigr\}.
\end{align*}
Consequently, similarly to~\eqref{Eq:Moment3} in the proof of Proposition~\ref{Lemma:Normality},
\begin{align}
\label{Eq:ThirdMomentBound}
	\mathbb{E}& \biggl(\sum_{i \in \mathcal{C}_m}|V-V_i|^3 \biggr) \leq n \mathbb{E} \Bigl[ \max_{i \in \mathcal{C}_m} \mathbb{E} \{(V-V_i)^2 | \Pi\}^{1/2} \mathbb{E}\{(V-V_i)^4 | \Pi \}^{1/2} \Bigr] \nonumber \\
	& \leq \biggl[ 4 + \frac{4(n+2)}{n-1} + \frac{10}{n-1} \mathbb{E} \bigl\{ h^2 \bigl( (X_1,Y_2),(X_2,Y_3) \bigr) \bigr\} \biggr]^{1/2} \nonumber \\
	& \hspace{100pt} \times \biggl[ \frac{216n}{(n-1)^2} \max_{\sigma \in \mathcal{S}_4}\mathbb{E} \bigl\{ h^4 \bigl((X_1,Y_{\sigma(1)}),(X_2,Y_{\sigma(2)}) \bigr) \bigr\} \biggr]^{1/2}.
\end{align}
It remains to bound the conditional variance term on the right-hand side of~\eqref{Eq:PermutedSteinDecomp}.  For $r \in \mathbb{N}$ and $i \in \mathcal{C}_m$, it is convenient to define the index set
\[
  I_r(i) := \bigl\{\Pi^{-r}(i),\ldots, \Pi^{-1}(i), i,\Pi(i),\ldots,\Pi^{r}(i)\bigr\}.
\]
We first observe that, if $i_1,i_2 \in \mathcal{C}_m$ and $i_2 \not\in I_5(i_1)$ (and if $m\geq 5$), then
\begin{align*}
	&\binom{n}{2}^{2} \bigl| \mathrm{Cov}\bigl(Z_{i_1}^2, Z_{i_2}^2\bigm|\Pi\bigr) \bigr| \\
	&= \Biggl| \sum_{\substack{i_1',i_2' \in \mathcal{C}_m \setminus \{i_1\} \\ i_3',i_4' \in \mathcal{C}_m \setminus \{i_2\}}} \!\!\!\!\!\!\!\! \mathrm{Cov}\Bigl( h \bigl((X_{i_1},Y_{\Pi(i_1)}), (X_{i_1'},Y_{\Pi(i_1')}) \bigr) h \bigl((X_{i_1},Y_{\Pi(i_1)}), (X_{i_2'},Y_{\Pi(i_2')}) \bigr), \\
	& \hspace{50pt} h \bigl((X_{i_2},Y_{\Pi(i_2)}), (X_{i_3'},Y_{\Pi(i_3')}) \bigr) h \bigl((X_{i_2},Y_{\Pi(i_2)}), (X_{i_4'},Y_{\Pi(i_4')}) \bigr) \Bigm| \Pi\Bigr) \Biggr| \\
	& \lesssim n^2 \Bigl| \mathrm{Cov} \Bigl( h \bigl( (X_1,Y_2), (X_3,Y_4) \bigr) h \bigl( (X_1,Y_2), (X_5,Y_6) \bigr), \\
	& \hspace{100pt} h \bigl( (X_7,Y_8), (X_3,Y_4) \bigr) h \bigl( (X_7,Y_8), (X_5,Y_6) \bigr) \Bigr) \Bigr| \\
	& \hspace{150pt} + n \max_{\sigma \in \mathcal{S}_4} \mathbb{E} \bigl\{ h^4 \bigl( (X_1, Y_{\sigma(1)}), (X_2,Y_{\sigma(2)}) \bigr) \bigr\} \\
	& = n^2 \mathbb{E}\bigl\{ g^2 \bigl( (X_3,Y_4),(X_5,Y_6) \bigr) \bigr\} + n \max_{\sigma \in \mathcal{S}_4} \mathbb{E} \bigl\{ h^4 \bigl( (X_1, Y_{\sigma(1)}), (X_2,Y_{\sigma(2)}) \bigr) \bigr\}.
\end{align*}
Similarly, if $i_1,i_2 \in \mathcal{C}_m$ and $i_2 \not\in I_7(i_1)$ and $m \geq 6$, then
\begin{align*}
	\binom{n}{2}^{2} \max \Bigl\{ \Bigl| \mathrm{Cov} \bigl(Z_{i_1}^2, &Z_{i_2}Z_{\Pi^{-1}(i_2)} \bigm| \Pi\bigr) \Bigr|, \Bigl|  \mathrm{Cov} \bigl(Z_{i_1}Z_{\Pi^{-1}(i_1)} , Z_{i_2}Z_{\Pi^{-1}(i_2)} \bigm| \pi\bigr) \Bigr| \Bigr\} \\
	&\lesssim n \max_{\sigma \in \mathcal{S}_4} \mathbb{E} \bigl\{ h^4 \bigl( (X_1, Y_{\sigma(1)}), (X_2,Y_{\sigma(2)}) \bigr) \bigr\}.
\end{align*}
Applying similar bounds to the remaining terms we have that 
\begin{align*}
	&\mathbb{E} \Biggl\{ \max_{\substack{i_1 \in \mathcal{C}_m \\ i_2 \in \mathcal{C}_m \setminus I_7(i_1)}} \bigl| \mathrm{Cov}\bigl( (V-V_{i_1})^2, (V-V_{i_2})^2 \bigm|\Pi\bigr) \bigr| \Biggr\} \\
	& \lesssim n^{-2} \mathbb{E}\bigl\{ g^2 \bigl( (X_3,Y_4),(X_5,Y_6) \bigr) \bigr\} + n^{-3} \max_{\sigma \in \mathcal{S}_4} \mathbb{E} \bigl\{ h^4 \bigl( (X_1, Y_{\sigma(1)}), (X_2,Y_{\sigma(2)}) \bigr) \bigr\}.
\end{align*}
Thus,
\begin{align}
\label{Eq:FourthMomentFinal}
	&\mathrm{Var} \biggl( \sum_{i \in \mathcal{C}_m} (V-V_i)^2 \biggm| \Pi\biggr) = \sum_{i_1,i_2 \in \mathcal{C}_m} \mathrm{Cov}\bigl( (V-V_{i_1})^2, (V-V_{i_2})^2\bigm|\Pi \bigr) \nonumber \\
	& \lesssim \! |\mathcal{C}_m| \! \max_{i \in \mathcal{C}_m} \mathbb{E}\bigl\{(V-V_i)^4\bigm|\Pi\bigr\} \!+\! |\mathcal{C}_m|^2 \!\!\!\!\!\!\! \max_{\substack{i_1 \in \mathcal{C}_m \\ i_2 \in \mathcal{C}_m \setminus I_7(i_1) }} \!\!\!\!\!\! \bigl| \mathrm{Cov} \bigl( (V-V_{i_1})^2, (V-V_{i_2})^2 \bigm| \Pi\bigr) \bigr| \nonumber \\
	& \lesssim \frac{1}{n} \max_{\sigma \in \mathcal{S}_4} \mathbb{E} \bigl\{ h^4 \bigl( (X_1, Y_{\sigma(1)}), (X_2,Y_{\sigma(2)}) \bigr) \bigr\} + \mathbb{E}\bigl\{ g^2 \bigl( (X_1,Y_2),(X_3,Y_4) \bigr) \bigr\}.
\end{align}
From~\eqref{Eq:SteinBound},~\eqref{Eq:SmallCyclesBound},~\eqref{Eq:VGV},~\eqref{Eq:VVi},~\eqref{Eq:PermutedSteinDecomp},~\eqref{Eq:ThirdMomentBound} and~\eqref{Eq:FourthMomentFinal}, and taking $m=6$, we deduce that
\begin{align}
\label{Eq:WassersteinBound1}
&d_{\mathrm{W}}(U,W)^2 \leq 2d_{\mathrm{W}}(U,V)^2 + 2d_{\mathrm{W}}(V,W)^2 \nonumber \\
&\lesssim \max \biggl[ \mathbb{E}\bigl\{ g^2 \bigl((X_1,Y_2),(X_3,Y_4)\bigr) \bigr\}, \Bigl(\mathbb{E} \bigl| \mathbb{E} \bigl\{ h \bigl( (X_1,Y_2), (X_3,Y_1) \bigr) | X_3,Y_2 \bigr\} \bigr|\Bigr)^2,  \nonumber \\
                      &\max_{\sigma \in \mathcal{S}_4} \biggl\{  \frac{\mathbb{E}\bigl\{ h^4 \bigl((X_1,Y_{\sigma(1)}),(X_2,Y_{\sigma(2)})\bigr) \bigr\}}{n} ,  \frac{\mathbb{E}^2 \bigl\{ h^4 \bigl((X_1,Y_{\sigma(1)}),(X_2,Y_{\sigma(2)})\bigr) \bigr\}}{n^2} \biggr\}  \biggr].
\end{align}

An alternative, cruder, bound can be found by reasoning similarly to the calculations leading up to~\eqref{Eq:Var2}:
\begin{align}
\label{Eq:WassersteinBound2}
  &d_\mathrm{W}(U,W)^2 \leq (\mathbb{E}|W| + \mathbb{E}|U|)^2 \leq 4/\pi+ 2 \mathbb{E}(U^2) \nonumber \\
  &\lesssim 1 \!+\! \frac{1}{n^2} \!\!\!\! \sum_{\substack{(i_1,i_2) \in \mathcal{I}_2\\(i_3,i_4) \in \mathcal{I}_2}} \!\!\!\!\! \mathbb{E}\bigl\{h\bigl((X_{i_1},Y_{\Pi(i_1)}),(X_{i_2},Y_{\Pi(i_2)}) \bigr) h \bigl((X_{i_3},Y_{\Pi(i_3)}),(X_{i_4},Y_{\Pi(i_4)}) \bigr)\bigr\} \nonumber \\
	& \lesssim 1 + n^2 \bigl| \mathbb{E} \bigl\{ h \bigl((X_1,Y_{\Pi(1)}),(X_2,Y_{\Pi(2)}) \bigr) h \bigl((X_3,Y_{\Pi(3)}),(X_4,Y_{\Pi(4)}) \bigr) \bigr\} \bigr| \nonumber \\
	& \hspace{50pt} + n \bigl| \mathbb{E} \bigl\{ h \bigl((X_1,Y_{\Pi(1)}),(X_2,Y_{\Pi(2)}) \bigr) h \bigl((X_1,Y_{\Pi(1)}),(X_3,Y_{\Pi(3)}) \bigr) \bigr\} \bigr| \nonumber \\	
              & \hspace{50pt} + \mathbb{E} \bigl\{ h^2 \bigl((X_1,Y_{\Pi(1)}),(X_2,Y_{\Pi(2)}) \bigr) \bigr) \bigr\} \nonumber \\
  &\lesssim 1 + \frac{1}{n}\max_{\sigma \in \mathcal{S}_4} \mathbb{E}\bigl\{h^2 \bigl((X_1,Y_{\sigma(1)}),(X_2,Y_{\sigma(2)})\bigr)\bigr\}\nonumber \\
	& \lesssim 1 + \frac{1}{n} \max_{\sigma \in \mathcal{S}_4} \mathbb{E}^{1/2} \bigl\{ h^4 \bigl((X_1,Y_{\sigma(1)}),(X_2,Y_{\sigma(2)})\bigr)\bigr\}.
\end{align}
The result follows from~\eqref{Eq:WassersteinBound1} and~\eqref{Eq:WassersteinBound2} by separately considering the two cases in which $n^{-1}\max_{\sigma \in \mathcal{S}_4} \mathbb{E}\bigl\{ h^4 \bigl((X_1,Y_{\sigma(1)}),(X_2,Y_{\sigma(2)})\bigr)\bigr\}$ is less than or greater than or equal to one.
\end{proof}


Unfortunately, Propositions~\ref{Lemma:Normality} and~\ref{Lemma:PermutedNormality} are not quite strong enough for us to be able to apply in the proof of Theorem~\ref{Thm:PowerFunction} because they do not consider the joint asymptotic normality of our test statistic computed on the original data set together with the null statistics computed on the permuted data sets.  Our next aim is to build on these propositions to provide such a result.

For $B \in \mathbb{N}$, write $\mathcal{G}_{B,1}$ for the set of differentiable functions from $\mathbb{R}^{B+1}$ to $\mathbb{R}$.  For $B \in \mathbb{N}$ and $g \in \mathcal{G}_{B,1}$, define
\[
	M_1(g):= \sup_{x \neq y} \frac{|g(x)-g(y)|}{\|x-y\|} \quad \text{ and} \quad  M_2(g):= \sup_{x \neq y} \frac{\| \nabla g(x) - \nabla g(y)\|}{ \|x-y\|}.
\]
Moreover, define
\[
	\mathcal{G}_B:= \bigl\{ g \in \mathcal{G}_{B,1}: \max( \|g\|_\infty, M_1(g),M_2(g)) \leq 1 \bigr\},
\]
and, for random vectors $W$ and $Z$ taking values in $\mathbb{R}^{B+1}$, define
\[
  d_{\mathcal{G}_B}(W,Z) := \sup_{g \in \mathcal{G}_B} | \mathbb{E} g(W) - \mathbb{E} g(Z)|.
\]

\begin{lemma}
\label{Lemma:JointNormality}
Let $(X_1,Y_1),\ldots,(X_n,Y_n)$ be independent and identically distributed random elements in a product space $\mathcal{Z}=\mathcal{X} \times \mathcal{Y}$, let $B \in \mathbb{N}$, and let $\Pi_1,\ldots,\Pi_B$ be a sequence of independent, uniformly random elements of $\mathcal{S}_n$, independent of $(X_i,Y_i)_{i=1}^n$.  Let $h,h_\mathrm{p} : \mathcal{Z} \times \mathcal{Z} \rightarrow \mathbb{R}$ be symmetric measurable functions that satisfy $\mathbb{E} h^2\bigl((X_1,Y_1),(X_2,Y_2)\bigr)=1$, $\mathbb{E} h_\mathrm{p}^2\bigl((X_1,Y_2),(X_3,Y_4)\bigr) =1$, and
\[
	\mathbb{E} h \bigl( (x,y), (X_1,Y_1) \bigr) = \mathbb{E} h_\mathrm{p} \bigl((x,y),(x',Y_1) \bigr) = \mathbb{E} h_\mathrm{p} \bigl((x,y),(X_1,y') \bigr) = 0
\]
for all $(x,y),(x',y') \in \mathcal{Z}$. Write
\begin{align*}
  g\bigl((x,y),(x',y')\bigr) &:=\mathbb{E}\bigl\{h\bigl((x,y),(X_1,Y_1)\bigr) h\bigl((x',y'),(X_1,Y_1)\bigr)\bigr\}, \\
  g_\mathrm{p}\bigl((x,y),(x',y')\bigr) &:= \mathbb{E}\bigl\{ h_\mathrm{p} \bigl((x,y),(X_1,Y_2)\bigr) h_\mathrm{p}\bigl((x',y'),(X_1,Y_2)\bigr) \bigr\}.
\end{align*}                                         
Further, set 
\[
	U_0:= \frac{1}{2}\binom{n}{2}^{-1/2} \sum_{(i_1,i_2) \in \mathcal{I}_2} h\bigl((X_{i_1},Y_{i_1}),(X_{i_2},Y_{i_2}) \bigr).
\]
and, for $b=1,\ldots,B$, 
\[
	U_b:= \frac{1}{2}\binom{n}{2}^{-1/2} \sum_{(i_1,i_2) \in \mathcal{I}_2} h_\mathrm{p}\bigl((X_{i_1},Y_{\Pi_b(i_1)}),(X_{i_2},Y_{\Pi_b(i_2)}) \bigr).
\]
Then, letting $W \sim N_{B+1}(0,I_{B+1})$, there exists a universal constant $C>0$ such that
\begin{align*}
	&d_{\mathcal{G}_B} \bigl( (U_0,U_1,\ldots,U_B), W \bigr) \\
	 &\leq C \max \biggl[ \mathbb{E}^{1/2}\bigl\{ g^2 \bigl((X_1,Y_1),(X_2,Y_2)\bigr) \bigr\}, \frac{\mathbb{E}^{1/2}\bigl\{ h^4 \bigl((X_1,Y_1),(X_2,Y_2)\bigr) \bigr\}}{n^{1/2}}, \\
	& \hspace{10pt} B\mathbb{E}^{1/2}\bigl\{ g_\mathrm{p}^2 \bigl((X_1,Y_2),(X_3,Y_4)\bigr) \bigr\}, B\mathbb{E}^{1/2} \bigl| \mathbb{E} \bigl\{ h_\mathrm{p} \bigl( (X_1,Y_2), (X_3,Y_1) \bigr) | X_3, Y_2\bigr\} \bigr|,  \\
	& \hspace{10pt} \frac{B^2}{n^{1/2}} \max_{\sigma \in \mathcal{S}_4} \mathbb{E}^{1/2}\bigl\{ h_\mathrm{p}^4 \bigl((X_1,Y_{\sigma(1)}),(X_2,Y_{\sigma(2)})\bigr)\bigr\}   \biggr].
\end{align*}
\end{lemma}

\begin{proof}[Proof of Lemma~\ref{Lemma:JointNormality}]
For each $b=1,\ldots,B$ and $m \in \mathbb{N}$ define $\mathcal{C}_m^{(b)}:= \{i \in [n] : i \text{ falls in a cycle of length } \geq m \text{ in } \Pi_b \}$ and 
\[
	V^{(b)}:= \frac{1}{2} \binom{n}{2}^{-1/2} \sum_{\substack{i_1 \neq i_2 \\ i_1,i_2 \in \mathcal{C}_m^{(b)} }} h \bigl( (X_{i_1},Y_{\Pi_b(i_1)}), (X_{i_2}, Y_{\Pi_b(i_2)}) \bigr).
\]
Following the same arguments as in the proof of Proposition~\ref{Lemma:PermutedNormality} leading up to~\eqref{Eq:SmallCyclesBound}, we have
\[
	\mathbb{E}\{(U_{b} - V^{(b)})^2\} \leq \frac{28m^4}{n-1} \max_{\sigma \in \mathcal{S}_4} \mathbb{E}\bigl\{ h^2\bigl( (X_1,Y_{\sigma(1)}), (X_2, Y_{\sigma(2)}) \bigr) \bigr\}
\]
for each $b=1,\ldots,B$.  We will now recall and redefine various pieces of notation from the proofs of Propositions~\ref{Lemma:Normality} and~\ref{Lemma:PermutedNormality}. Write
\[
	Z_i:= \binom{n}{2}^{-1/2} \sum_{i' \in [n] \setminus \{i\}} h \bigl((X_i,Y_i),(X_{i'},Y_{i'}) \bigr)
\]
and, for $b=1,\ldots,B$ and for $i \in \mathcal{C}_m^{(b)}$, write
\[
		Z_i^{(b)}:= \binom{n}{2}^{-1/2} \sum_{i' \in \mathcal{C}_m^{(b)} \setminus \{i\}} h_\mathrm{p} \bigl((X_i,Y_{\Pi_b(i)}),(X_{i'},Y_{\Pi_b(i')}) \bigr)
\]
so that $\sum_{i=1}^n Z_i = 2U_0$ and $\sum_{i \in \mathcal{C}_m^{(b)}} Z_i^{(b)} = 2 V^{(b)}$. Also define $V^{(0)}:=U_0$, $V^{(0)}_i := V^{(0)} - Z_i$ and
\begin{align*}
	V_i^{(b)} & := V^{(b)} - Z_i^{(b)} - Z_{\Pi_b^{-1}(i)}^{(b)} + \binom{n}{2}^{-1/2} h_\mathrm{p} \bigl((X_i,Y_{\Pi_b(i)}),(X_{\Pi_b^{-1}(i)},Y_i) \bigr) \\
	& = \frac{1}{2} \binom{n}{2}^{-1/2}  \sum_{\substack{i_1,i_2 \in \mathcal{C}_m^{(b)} \setminus \{i,\Pi_b^{-1}(i)\} \\ i_1 \neq i_2}} h_\mathrm{p} \bigl((X_{i_1},Y_{\Pi_b(i_1)}),(X_{i_2},Y_{\Pi_b(i_2)}) \bigr),
\end{align*}
so that $V_i^{(0)}$ is $\sigma\bigl(\bigl\{(X_{i'},Y_{i'}):i' \neq i\bigr\}\bigr)$-measurable and, for $b=1,\ldots,B$, we have that $V_i^{(b)}$ is $\sigma\bigl(\Pi_b,\bigl\{(X_{i'},Y_{i'}):i' \neq i\bigr\}\bigr)$-measurable. 

Write $\mathcal{G}_{B,2}$ for the set of twice differentiable functions $G : \mathbb{R}^{B+1} \rightarrow \mathbb{R}$, and for $G \in \mathcal{G}_{B,2}$, write $\nabla^2 G = (G_{b_1b_2})_{b_1,b_2=0}^B$ for the Hessian matrix of $G$, and define
\[
	M_3(G):= \sup_{x \neq y}\frac{ \| \nabla^2 G(x) - \nabla^2 G(y) \|_\mathrm{op}}{\|x-y\|}.
\]
We will also write $\Delta G(x) = \sum_{b=0}^B G_{bb}(x)$ for the Laplacian of $G$. We now introduce the function class 
\[
\mathcal{G}':= \bigl\{ G \in \mathcal{G}_{B,2} : M_1(G) \leq 1, M_2(G) \leq 1/2, M_3(G) \leq (2\pi)^{1/2}/4 \bigr\}.
\]
Write $V=(V^{(0)}, V^{(1)}, \ldots, V^{(B)})$, and for $i =1,\ldots,n$, define $V_i$ taking values in $\mathbb{R}^{B+1}$ by $V_{i0}:= V_i^{(0)}$ and for $b = 1,\ldots,B$,
\[
	V_{ib} := \left\{ \begin{array}{ll} V_i^{(b)} & \text{if } i \in \mathcal{C}_m^{(b)} \\ V^{(b)} & \text{if } i \not\in \mathcal{C}_m^{(b)}, \end{array} \right.
\]
so that $V_i$ is $\sigma\bigl(\bigl\{\Pi_b:b=1,\ldots,B\bigr\},\bigl\{(X_{i'},Y_{i'}):i' \neq i\bigr\}\bigr)$-measurable.  We now seek to apply \citet[][Lemma~1]{Raic2004}, which states that, given any $g \in \mathcal{G}_B$, there exists $G \in \mathcal{G}_{B,2}$ with $M_2(G) \leq 1/2$ and $M_3(G) \leq (2\pi)^{1/2}/4$ such that
\[
  g(v) - \mathbb{E}g(W) = \Delta G(v) - v^T \nabla G(v)
\]
for all $v \in \mathbb{R}^{B+1}$.  In fact, by examining the proof of this result, we see that from Rai\v{c}'s construction, $M_1(G) \leq 1$, and hence $G \in \mathcal{G}'$.  Writing $\nabla G = (G_0,G_1,\ldots,G_B)^T$, we deduce that
\begin{align*}
	&d_{\mathcal{G}_B}( V, W) \leq \sup_{G \in \mathcal{G}'} \bigl| \mathbb{E} \bigl\{ \Delta G(V) - V^T \nabla G(V) \bigr\} \bigr| \\
	& = \sup_{G \in \mathcal{G}'} \biggl| \mathbb{E} \biggl[ G_{00}(V) \biggl\{ 1- \frac{1}{2} \sum_{i=1}^n(V^{(0)} - V_i^{(0)})^2 \biggr\} \\
	& \hspace{30pt}+ \sum_{b=1}^B G_{bb}(V) \biggl\{1- \frac{1}{4} \sum_{i \in \mathcal{C}_m^{(b)}} (V^{(b)}-V_i^{(b)})^2 \biggr\} \\
	& \hspace{30pt} - \frac{3}{4} \sum_{b=1}^B G_{0b}(V) \sum_{i \in \mathcal{C}_m^{(b)}} (V^{(b)} - V_i^{(b)})(V^{(0)} - V_i^{(0)}) \\
	& \hspace{30pt} - \frac{1}{2} \sum_{1 \leq b < b' \leq B} G_{bb'}(V) \sum_{i \in \mathcal{C}_m^{(b)} \cap \mathcal{C}_m^{(b')}} (V^{(b)} - V_i^{(b)}) (V^{(b')} - V_i^{(b')}) \\
	& \hspace{30pt} - \frac{1}{2} \sum_{i=1}^n \bigl\{G_0(V) - G_0(V_i) - (V-V_i)^T \nabla G_0(V) \bigr\}(V^{(0)} - V_i^{(0)}) \\
	& \hspace{30pt} - \frac{1}{4} \sum_{b=1}^B \sum_{i \in \mathcal{C}_m^{(b)} } \bigl\{G_b(V) - G_b(V_i) - (V-V_i)^T \nabla G_b(V) \bigr\}(V^{(b)} - V_i^{(b)}) \\
	& \hspace{30pt} +\sum_{b=1}^B G_b(V) \biggl\{ \frac{1}{4} \sum_{i \in \mathcal{C}_m^{(b)}} (V^{(b)} - V_i^{(b)}) - V^{(b)} \biggr\} \\
	& \hspace{30pt}- \frac{1}{4} \sum_{b=1}^B \sum_{i \in \mathcal{C}_m^{(b)}} G_b(V_i) (V^{(b)} - V_i^{(b)}) \biggr] \biggr|,
\end{align*}
where have used the fact that for each $i \in [n]$ we have $\mathbb{E} \bigl\{G_0(V_i)(V^{(0)}-V_i^{(0)})\bigr\}=0$.  Compared with the proofs of Propositions~\ref{Lemma:Normality} and~\ref{Lemma:PermutedNormality}, the only new terms that we need to bound are the third and fourth terms involving the interactions between $V^{(b)} - V_i^{(b)}$ and $V^{(b')} - V_i^{(b')}$ for $b \neq b' \in \{0,1,\ldots,B\}$. The third term can be bounded by very similar calculations to those leading up to~\eqref{Eq:FourthMomentFinal}, with the main difference being that $\mathbb{E} \bigl\{ h \bigl( (X_1,Y_1),(X_2,Y_2) \bigr) h_\mathrm{p} \bigl( (X_1, Y_{\sigma(1)}), (X_3, Y_{\sigma(3)}) \bigr) \bigr\} =0$ for any $\sigma \in \mathcal{S}_n$ with $\{\sigma(1),\sigma(3)\} \cap \{1,2,3\} = \emptyset$, so that the term that corresponds to $\mathbb{E}\bigl\{ g^2((X_1,Y_2),(X_3,Y_4)) \bigr\}$ does not appear.  We therefore have that
\begin{align*}
	& \sup_{G \in \mathcal{G}'} \biggl| \mathbb{E} \biggl\{ \sum_{b=1}^B G_{0b}(V) \sum_{i \in \mathcal{C}_m^{(b)}} (V^{(b)} - V_i^{(b)})(V^{(0)} - V_i^{(0)}) \biggr\} \biggr| \\
	& \lesssim B \mathbb{E} \biggl| \sum_{i \in \mathcal{C}_m^{(1)}} (V^{(1)} - V_i^{(1)})(V^{(0)} - V_i^{(0)}) \biggr| \\
	& \lesssim B \mathbb{E}^{1/2} \biggl[ \biggl\{ \!\! \sum_{i \in \mathcal{C}_m^{(1)}} \!\! Z_i \biggl( Z_i^{(1)} \!+\! Z_{\Pi^{-1}(i)}^{(1)} \!-\! \binom{n}{2}^{-1/2} \!\!\!\!\!\!\! h_\mathrm{p} \bigl( (X_i, Y_{\Pi(i)}), (X_{\Pi^{-1}(i)},Y_i) \bigr) \biggr) \biggr\}^2 \biggr] \\
	& \lesssim \frac{B}{n^{1/2}} \mathbb{E}^{1/4} \bigl\{ h^4 \bigl( (X_1,Y_1), (X_2,Y_2) \bigr) \bigr\}  \max_{\sigma \in \mathcal{S}_4}  \mathbb{E}^{1/4} \bigl\{ h_\mathrm{p}^4 \bigl( (X_1,Y_{\sigma(1)}), (X_2,Y_{\sigma(2)}) \bigr) \bigr\}.
\end{align*}
Similarly, for the fourth term,
\begin{align*}
	\sup_{G \in \mathcal{G}'} \biggl| \mathbb{E} \biggl\{ \sum_{1 \leq b < b' \leq B} G_{bb'}(V) &\sum_{i \in \mathcal{C}_m^{(b)} \cap \mathcal{C}_m^{(b')}} (V^{(b)} - V_i^{(b)}) (V^{(b')} - V_i^{(b')}) \biggr\} \biggr| \\
	& \lesssim B^2 \biggl( \frac{1}{n} \max_{\sigma \in \mathcal{S}_4}  \mathbb{E}\bigl\{ h_\mathrm{p}^4 \bigl( (X_1,Y_{\sigma(1)}), (X_2,Y_{\sigma(2)}) \bigr) \bigr\} \biggr)^{1/2}.
\end{align*}
The result follows.
\end{proof}

\section{Auxiliary results}

\begin{lemma}
\label{Lemma:Separability}
Let $(\mathcal{Z},\mathcal{C},\nu)$ be a separable measure space.  Then $L^q(\nu)$ is separable for $q \in [1,\infty)$.
\end{lemma}
\begin{proof}
  Since $(\mathcal{Z},\mathcal{C},\nu)$ is separable, there exists a sequence $(C_n)$ of sets in $\mathcal{C}$ such that, given any $A \in \mathcal{C}$, we can find a subsequence of integers $1 \leq n_1 < n_2 < \ldots$ with the property that $\nu(C_{n_k} \triangle A) \rightarrow 0$ as $k \rightarrow \infty$.  Consider the countable set $\mathcal{G}_+$ of functions in $L^q(\nu)$ of the form $g = \sum_{\ell=1}^m c_\ell \mathbbm{1}_{C_\ell}$, where $m \in \mathbb{N}$, $c_\ell \in (\mathbb{Q} \cup \{\infty\}) \cap [0,\infty]$ for all~$\ell$, and we have $c_\ell \in \mathbb{Q}$ if $\nu(C_\ell) > 0$ and $c_\ell = 0$ if $\nu(C_\ell) = \infty$.  We claim that $\mathcal{G}_+$ is dense in the set of non-negative functions in $L^q(\nu)$.  To see this, first suppose that $f = \sum_{\ell=1}^m a_\ell \mathbbm{1}_{A_\ell} \in L^q(\nu)$, where $a_\ell \in [0,\infty]$ and $A_\ell \in \mathcal{C}$.  Then we must have $a_\ell < \infty$ whenever $\nu(A_\ell) > 0$ and $a_\ell = 0$ whenever $\nu(A_\ell) = \infty$.  Given $\epsilon > 0$, for each $\ell=1,\ldots,m$, find $c_\ell \in (\mathbb{Q} \cup \{\infty\}) \cap [0,\infty]$ such that
  \[
    |a_\ell - c_\ell|^q \nu(A_\ell) < \frac{\epsilon}{(2m)^q}.
  \]
  Here, we must have $c_\ell = 0$ if $\nu(A_\ell) = \infty$.  Now, for each $\ell = 1,\ldots,m$, choose $C_\ell$ from our countable set such that
  \[
    c_\ell^q \nu(A_\ell \triangle C_\ell) < \frac{\epsilon}{(2m)^q}.
  \]
  Here, we must have $\nu(C_\ell) = 0$ if $c_\ell = \infty$.  Then $g = \sum_{\ell=1}^m c_\ell \mathbbm{1}_{C_\ell} \in \mathcal{G}_+$ and
  \begin{align*}
   & \int_{\mathcal{Z}} |f - g|^q \, d\nu \\
    &\leq 2^{q-1} \int_{\mathcal{Z}} \biggl|\sum_{\ell=1}^m a_\ell \mathbbm{1}_{A_\ell} - \sum_{\ell=1}^m c_\ell \mathbbm{1}_{A_\ell}\biggr|^q \, d\nu + 2^{q-1} \int_{\mathcal{Z}} \biggl|\sum_{\ell=1}^m c_\ell \mathbbm{1}_{A_\ell} - \sum_{\ell=1}^m c_\ell \mathbbm{1}_{C_\ell}\biggr|^q \, d\nu \\
    &\leq (2m)^{q-1} \sum_{\ell=1}^m |a_\ell - c_\ell|^q \nu(A_\ell) + (2m)^{q-1} \sum_{\ell=1}^m c_\ell^q \nu(A_\ell \triangle C_\ell) < \epsilon.
  \end{align*}
  Now suppose that $f$ is any non-negative function in $L^q(\nu)$.  Then, given $\epsilon > 0$, choose $f_* = \sum_{\ell=1}^m a_\ell \mathbbm{1}_{A_\ell} \in L^q(\nu)$ such that $a_\ell \in [0,\infty]$, $A_\ell \in \mathcal{C}$, $f_* \leq f$ and $\nu(f_*^q) > \nu(f^q) - \epsilon/2^q$.  Now, by what we have proved above, we can find $g \in \mathcal{G}_+$ such that $\int_{\mathcal{Z}} |f_* - g|^q \, d\nu < \epsilon/2^q$.  Then
  \begin{align*}
    \int_{\mathcal{Z}} |f - g|^q \, d\nu &\leq 2^{q-1} \int_{\mathcal{Z}} |f - f_*|^q \, d\nu + 2^{q-1} \int_{\mathcal{Z}} |f_* - g|^q \, d\nu \\
    &\leq 2^{q-1}\bigl\{\nu(f^q) - \nu(f_*^q)\bigr\} + 2^{q-1} \int_{\mathcal{Z}} |f_* - g|^q \, d\nu < \epsilon.
  \end{align*}
  Here, we have used the fact that $x^q + y^q \leq 1$ for $x,y \geq 0$ with $x+y = 1$.  This establishes our claim for non-negative $f \in L^q(\nu)$.  Finally, if $f \in L^q(\nu)$ is arbitrary, then we can write $f = f_+ - f_{-}$, where $f_+ := \max(f,0)$, $f_{-} := \max(-f,0)$, and given $\epsilon > 0$, find $g_+,g_{-} \in \mathcal{G}_+$ such that $\int_{\mathcal{Z}} |f_+ - g_+|^q \, d\nu < \epsilon/2^q$ and $\int_{\mathcal{Z}} |f_{-} - g_{-}|^q \, d\nu < \epsilon/2^q$.  Then, writing $g := g_+ - g_{-}$, 
  \[
\int_{\mathcal{Z}} |f - g|^q \, d\nu \leq 2^{q-1}\int_{\mathcal{Z}} |f_+ - g_{+}|^q \, d\nu + 2^{q-1}\int_{\mathcal{Z}} |f_- - g_{-}|^q \, d\nu < \epsilon.
\]
Noting that $\{g_1 - g_2: g_1,g_2 \in \mathcal{G}_+\}$ is a countable subset of $L^q(\nu)$, the result follows. 
\end{proof}

\begin{lemma}
  \label{Lemma:Separability2}
Let $(\mathcal{X},\mathcal{A},\mu_X)$ and $(\mathcal{Y},\mathcal{B},\mu_Y)$ be separable, $\sigma$-finite measure spaces.  Then, writing $\mu = \mu_X \otimes \mu_Y$ for the product measure on $\mathcal{X} \times \mathcal{Y}$, the space $L^q(\mu)$ is separable for $q \in [1,\infty)$.
\end{lemma}
\begin{proof}
  From Lemma~\ref{Lemma:Separability}, we know that $L^q(\mu_X)$ and $L^q(\mu_Y)$ are separable.  We can therefore find countable orthonormal bases $(p_{j}^X )$ and $(p_k^Y)$ for $L^q(\mu_X)$ and $L^q(\mu_Y)$ respectively.  Then the set of functions of the form $p_{jk}(\cdot,\ast) = p_j^X(\cdot)p_k^Y(\ast)$ for $j,k \in \mathbb{N}$ is countable.  Suppose that $f \in L^q(\mu)$ is such that 
  \[
    \int_{\mathcal{X} \times \mathcal{Y}} fp_{jk} \, d\mu = 0
  \]
  for every $j,k \in \mathbb{N}$.  Then, by Fubini's theorem,
  \[
    0 = \int_{\mathcal{X}} \biggl( \int_{\mathcal{Y}} f(x,y)p_k^Y(y) \, d\mu_Y(y)\biggr) \, p_j^X(x) \, d\mu_X(x),
  \]
  so, for every $k \in \mathbb{N}$, the function $x \mapsto \int_{\mathcal{Y}} f(x,y)p_k^Y(y) \, d\mu_Y(y)$ is zero $\mu_X$-almost everywhere.  Now, for $k \in \mathbb{N}$, set
  \[
    \Omega_k := \biggl\{x \in \mathcal{X}: \int_{\mathcal{Y}} f(x,y)p_k^Y(y) \, d\mu_Y(y) \neq 0\biggr\}.
  \]
  Since $\mu_X(\Omega_k) = 0$ for each $k \in \mathbb{N}$, we have that $\mu_X(\Omega) = 0$, where $\Omega := \cup_{k=1}^\infty \Omega_k$.  But for $x \in \mathcal{X} \setminus \Omega$,
  \[
    \int_{\mathcal{Y}} f(x,y)p_k^Y(y) \, d\mu_Y(y) = 0
  \]
  for every $k \in \mathbb{N}$, so for such $x$, we have that $y \mapsto f(x,y)$ is zero $\mu_Y$-almost everywhere.  Since $f \in L^q(\mu)$, we deduce that
  \begin{align*}
    \int_{\mathcal{X} \times \mathcal{Y}} |f|^q \, d\mu &= \int_{\mathcal{X}} \int_{\mathcal{Y}} |f(x,y)|^q \, d\mu_Y(y) \, d\mu_X(x) \\
    &= \int_{\mathcal{X} \setminus \Omega} \int_{\mathcal{Y}} |f(x,y)|^q \, d\mu_Y(y) \, d\mu_X(x) = 0.
  \end{align*}
Hence $f$ is zero, $\mu$-almost everywhere, as required.  
\end{proof}

\begin{lemma}
\label{Lemma:StoneWeierstrass}
Assume the setting of Example~\ref{Ex:InfDim}.  Then the collection of functions $\mathcal{B}:=\{p_{a,m}^X(\cdot) : (a,m) \in \mathcal{J}\}$ is an orthonormal basis of $L^2(\mu_X)$.
\end{lemma}
\begin{proof}[Proof of Lemma~\ref{Lemma:StoneWeierstrass}]
By Tychonoff's theorem, $\mathcal{X}$ is a compact Hausdorff space as it is  a product of compact Hausdorff spaces. The linear span of $\mathcal{B}$ is closed under multiplication, contains the constant functions and separates points, so by the Stone--Weierstrass theorem, it is dense with respect to the supremum norm in the space of real-valued continuous functions on $\mathcal{X}$. The continuous functions on $\mathcal{X}$ are dense in $L^2(\mu_X)$ and so, since the $L^2$ norm is bounded above by the supremum norm on our probability space, it follows that the linear span of $\mathcal{B}$ is dense in $L^2(\mu_X)$.

It now remains to prove that $\mathcal{B}$ is orthonormal. When $m = (m_1,m_2,\ldots) \in \mathbb{N}_0^{< \infty}$ satisfies $|m|>0$ we may write
\[
	p_{a,m}^X(x) = 2^{1/2} \mathrm{Re} \bigl( e^{-2 \pi i \langle m, x \rangle - a \pi i/2} \bigr),
\]
with $\langle m, x \rangle := \sum_{\ell=1}^\infty m_\ell x_\ell$. Then if $|m|,|m'| >0$ and $a,a' \in \{0,1\}$ we have
\begin{align*}
	&\int_\mathcal{X} p_{a,m}^X(x) p_{a',m'}^X(x) \,d \mu_X(x) \\
	&= \int_\mathcal{X} \bigl\{ \mathrm{Re} \bigl( e^{-2 \pi i \langle m + m' ,x \rangle - (a+a')\pi i/2} \bigr) + \mathrm{Re} \bigl( e^{-2 \pi i \langle m - m' ,x \rangle - (a-a')\pi i/2} \bigr) \bigr\} \,d \mu_X(x) \\
	&= \mathbbm{1}_{\{m=m'\}} \mathrm{Re}( e^{-(a-a') \pi i/2}) = \mathbbm{1}_{\{(a,m)=(a',m')\}}.
\end{align*}
Moreover, it is clear that $\int_{\mathcal{X}} p_{0,0}^X(x) p_{a,m}^X(x) \,d \mu_X(x)=0$ for any $|m|>0$ and $a \in \{0,1\}$ and that $\int_\mathcal{X} p_{0,0}^X(x)^2 \,d \mu_X(x) =1$.
\end{proof}

\begin{lemma}
\label{Prop:SquareSummableSequences}
Let $(a_j)_{j=1}^\infty$ be a sequence of real numbers such that we have $\sum_{j=1}^\infty a_j^2 =1$, and for $m \in \mathbb{N}$ write $F(m):= \sum_{j=m}^\infty a_j^2$. Then, for any $M \in \mathbb{N}$, we have
\[
	\sum_{j=1}^M |a_j| \leq \biggl( 2 \sum_{j=1}^M F(j)^{1/2} \biggr)^{1/2}.
\]
\end{lemma}
\begin{proof}[Proof of Lemma~\ref{Prop:SquareSummableSequences}]
The idea of the proof is to construct another sequence $(b_j)$ that has a heavier tail than $(a_j)$ but that is still square-summable, and then to use Cauchy--Schwarz. Indeed, for $j \in \mathbb{N}$ define
\[
	b_j:= \{ F(j)^{1/2} - F(j+1)^{1/2}\}^{1/2},
\]
for which we have $\sum_{j=m}^\infty b_j^2 =F(m)^{1/2}$. Then
\begin{align*}
	\biggl( \sum_{j=1}^M & |a_j| \biggr)^2 \leq \biggl( \sum_{j=1}^M \frac{a_j^2}{b_j^2} \biggr) \biggl( \sum_{j=1}^M b_j^2 \biggr) = \{ F(1)^{1/2} - F(M+1)^{1/2} \} \sum_{j=1}^M \frac{a_j^2}{b_j^2}  \\
	& \leq \sum_{j=1}^M \frac{a_j^2}{b_j^2} = \sum_{j=1}^M F(j)^{1/2} \biggl\{ 1+ \frac{F(j+1)^{1/2}}{F(j)^{1/2}} \biggr\} \leq 2 \sum_{j=1}^M F(j)^{1/2},
\end{align*}
as required.
\end{proof}

\end{document}